\documentclass{article}

\usepackage[centertags,sumlimits,intlimits,namelimits,reqno]{amsmath}
\usepackage{latexsym,amsfonts,amssymb,exscale,enumerate,amsthm}

\usepackage[utf8]{inputenc}
\usepackage[english]{babel}

\usepackage[centertags,sumlimits,intlimits,namelimits,reqno]{amsmath}
\usepackage[all,2cell]{xy} \UseAllTwocells

\newtheorem{theorem}{Theorem}[section]

\newtheorem{proposition}[theorem]{Proposition}

\newtheorem{lemma}[theorem]{Lemma}   
\newtheorem{corollary}[theorem]{Corollary}  

\newtheorem*{theorem*}{Theorem}
\newtheorem*{proposition*}{Proposition}

\theoremstyle{definition}
\newtheorem{definition}[theorem]{Definition}   
\newtheorem{example}[theorem]{Example}

\usepackage[hang,small,bf]{caption}
\usepackage{color}
\definecolor{myurlcolor}{rgb}{0.6,0,0}
\definecolor{mycitecolor}{rgb}{0,0,0.9}
\definecolor{myrefcolor}{rgb}{0,0,0.9}
\usepackage[pagebackref]{hyperref}
\hypersetup{colorlinks, linkcolor=myrefcolor, citecolor=mycitecolor, urlcolor=myurlcolor}

\newcommand{\define}[1]{{\bf \boldmath{#1}}}

\newcommand{\N}{\mathbb{N}}
\newcommand{\Z}{\mathbb{Z}}
\newcommand{\R}{\mathbb{R}}
\renewcommand{\L}{\mathcal{L}}

\newcommand{\Rel}{\mathrm{Rel}}
\newcommand\Span{\mathrm{Span}}
\newcommand{\Corel}{\mathrm{Corel}}
\newcommand{\Cospan}{\mathrm{Cospan}}
\newcommand{\Fin}{\mathrm{Fin}}
\newcommand{\Set}{\mathrm{Set}}

\newcommand{\SigFlow}{\mathrm{SigFlow}}
\newcommand{\Lag}{\mathrm{Lag}}
\newcommand{\Aff}{\mathrm{Aff}}
\newcommand{\PROP}{\mathrm{PROP}}
\newcommand{\Symm}{\mathrm{Symm}}
\newcommand{\Mon}{\mathrm{Mon}}
\newcommand{\Cat}{\mathrm{Cat}}

\newcommand{\Circ}{{\mathrm{Circ}}}
\newcommand{\Ccirc}{\widetilde{\mathrm{Circ}}}
\newcommand{\RLCCirc}{\mathrm{Circ}_{RLC}}

\newcommand{\Mod}{\mathrm{Mod}}

\newcommand{\A}{A}
\newcommand{\C}{C}
\newcommand{\T}{T}
\newcommand{\NN}{N}

\newcommand{\maps}{\colon}

\renewcommand{\hom}{\mathrm{hom}}

\newcommand{\To}{\Rightarrow}
\newcommand{\asrelto}{\nrightarrow}

\usepackage{tikz}  

\newcommand{\midarrow}{\tikz \draw[thick, ->-=.5] (0,0) -- +(.1,0);}

\usetikzlibrary{intersections,decorations.markings}
\usetikzlibrary{arrows,positioning,fit,matrix,shapes.geometric,external}
\usetikzlibrary{backgrounds,circuits,circuits.ee.IEC,shapes,fit,matrix}

\pgfdeclarelayer{edgelayer}
\pgfdeclarelayer{nodelayer}
\pgfsetlayers{edgelayer,nodelayer,main}

\tikzset{font=\footnotesize}
\tikzset{->-/.style={decoration={
  markings,
  mark=at position #1 with {\arrow{>}}},postaction={decorate}}}

\tikzstyle{none}=[inner sep=0pt]
\tikzstyle{circ}=[circle,fill=black,draw,inner sep=3pt]

\newcommand{\mult}[1]
{
\begin{aligned}
    \resizebox{#1}{!}{
\begin{tikzpicture}
	\begin{pgfonlayer}{nodelayer}
		\node [style=none] (0) at (1, -0) {};
		\node [style=circ] (1) at (0.125, -0) {};
		\node [style=none] (2) at (-1, 0.5) {};
		\node [style=none] (3) at (-1, -0.5) {};
	\end{pgfonlayer}
	\begin{pgfonlayer}{edgelayer}
		\draw[line width=2pt] (0.center) to (1.center);
		\draw[line width=2pt] [in=0, out=120, looseness=1.20] (1.center) to (2.center);
		\draw[line width=2pt] [in=0, out=-120, looseness=1.20] (1.center) to (3.center);
	\end{pgfonlayer}
      \end{tikzpicture}}
\end{aligned}
}

\newcommand{\unit}[1]
{
  \begin{aligned}
    \resizebox{#1}{!}{
\begin{tikzpicture}
	\begin{pgfonlayer}{nodelayer}
		\node [style=none] (0) at (1, -0) {};
		\node [style=none] (1) at (-1, -0) {};
		\node [style=circ] (2) at (0, -0) {};
	\end{pgfonlayer}
	\begin{pgfonlayer}{edgelayer}
		\draw[line width=2pt] (0.center) to (2);
	\end{pgfonlayer}
      \end{tikzpicture}}
  \end{aligned}
}

\newcommand{\comult}[1]
{
\begin{aligned}
    \resizebox{#1}{!}{
\begin{tikzpicture}
	\begin{pgfonlayer}{nodelayer}
		\node [style=none] (0) at (-1, -0) {};
		\node [style=circ] (1) at (-0.125, -0) {};
		\node [style=none] (2) at (1, 0.5) {};
		\node [style=none] (3) at (1, -0.5) {};
	\end{pgfonlayer}
	\begin{pgfonlayer}{edgelayer}
		\draw[line width=2pt] (0.center) to (1.center);
		\draw[line width=2pt] [in=180, out=60, looseness=1.20] (1.center) to (2.center);
		\draw[line width=2pt] [in=180, out=-60, looseness=1.20] (1.center) to (3.center);
	\end{pgfonlayer}
      \end{tikzpicture}}
\end{aligned}
}

\newcommand{\counit}[1]
{
  \begin{aligned}
    \resizebox{#1}{!}{
\begin{tikzpicture}
	\begin{pgfonlayer}{nodelayer}
		\node [style=none] (0) at (-1, -0) {};
		\node [style=none] (1) at (1, -0) {};
		\node [style=circ] (2) at (0, -0) {};
	\end{pgfonlayer}
	\begin{pgfonlayer}{edgelayer}
		\draw[line width=2pt] (0.center) to (2);
	\end{pgfonlayer}
      \end{tikzpicture}}
  \end{aligned}
}

\newcommand{\idone}[1]
{
  \begin{aligned}
    \resizebox{#1}{!}{
     \begin{tikzpicture}
	\begin{pgfonlayer}{nodelayer}
		\node [style=none] (0) at (-1, -0) {};
		\node [style=none] (1) at (1, -0) {};
		\node [style=none] (2) at (0, 0.5) {};
		\node [style=none] (3) at (0, -0.5) {};
	\end{pgfonlayer}
	\begin{pgfonlayer}{edgelayer}
		\draw[line width=2pt] (1.center) to (0.center);
	\end{pgfonlayer}
\end{tikzpicture} 
    }
  \end{aligned}
}

\newcommand{\singlegen}[1]
{
  \begin{aligned}
    \resizebox{#1}{!}{
     \begin{tikzpicture}
	\begin{pgfonlayer}{nodelayer}
		\node [style=none] (0) at (-1, -0) {};
		\node [style=none] (1) at (1, -0) {};
		\node [style=none] (2) at (0, 0.5) {};
		\node [style=none] (3) at (0, -0.5) {};
	\end{pgfonlayer}
	\begin{pgfonlayer}{edgelayer}
		\draw[line width=2pt] (1.center) to (0.center) node[midway,above] {\scalebox{2.5}{$\ell$}};
	\end{pgfonlayer}
\end{tikzpicture} 
    }
  \end{aligned}
}

\newcommand{\swap}[1]
{
  \begin{aligned}
    \resizebox{#1}{!}{
\begin{tikzpicture}
	\begin{pgfonlayer}{nodelayer}
		\node [style=none] (2) at (-0.5, -0.5) {};
		\node [style=none] (3) at (-2, 0.5) {};
		\node [style=none] (4) at (-0.5, 0.5) {};
		\node [style=none] (5) at (-2, -0.5) {};
	\end{pgfonlayer}
	\begin{pgfonlayer}{edgelayer}
		\draw[line width=2pt] [in=180, out=0, looseness=1.00] (3.center) to (2.center);
		\draw[line width=2pt] [in=0, out=180, looseness=1.00] (4.center) to (5.center);
	\end{pgfonlayer}
\end{tikzpicture}
    }
  \end{aligned}
}
\newcommand{\assocl}[1]
{
  \begin{aligned}
    \resizebox{#1}{!}{
\begin{tikzpicture}
	\begin{pgfonlayer}{nodelayer}
		\node [style=circ] (0) at (0.125, -0) {};
		\node [style=none] (1) at (-1, 0.5) {};
		\node [style=none] (2) at (-1, -0.5) {};
		\node [style=none] (3) at (0, -1) {};
		\node [style=none] (4) at (2.25, -0.5) {};
		\node [style=none] (5) at (0.25, -0) {};
		\node [style=circ] (6) at (1.25, -0.5) {};
		\node [style=none] (7) at (-1, -1) {};
	\end{pgfonlayer}
	\begin{pgfonlayer}{edgelayer}
		\draw[line width=2pt] [in=0, out=120, looseness=1.20] (0.center) to (1.center);
		\draw[line width=2pt] [in=0, out=-120, looseness=1.20] (0.center) to (2.center);
		\draw[line width=2pt] (4.center) to (6);
		\draw[line width=2pt] [in=0, out=120, looseness=1.20] (6) to (5.center);
		\draw[line width=2pt] [in=0, out=-120, looseness=1.20] (6) to (3.center);
		\draw[line width=2pt] (3.center) to (7.center);
	\end{pgfonlayer}
      \end{tikzpicture}}
  \end{aligned}
}

\newcommand{\assocr}[1]
{
  \begin{aligned}
    \resizebox{#1}{!}{
\begin{tikzpicture}
	\begin{pgfonlayer}{nodelayer}
		\node [style=circ] (0) at (0.125, -0.5) {};
		\node [style=none] (1) at (-1, -1) {};
		\node [style=none] (2) at (-1, 0) {};
		\node [style=none] (3) at (0, 0.5) {};
		\node [style=none] (4) at (2.25, 0) {};
		\node [style=none] (5) at (0.25, -0.5) {};
		\node [style=circ] (6) at (1.25, 0) {};
		\node [style=none] (7) at (-1, 0.5) {};
	\end{pgfonlayer}
	\begin{pgfonlayer}{edgelayer}
		\draw[line width=2pt] [in=0, out=-120, looseness=1.20] (0.center) to (1.center);
		\draw[line width=2pt] [in=0, out=120, looseness=1.20] (0.center) to (2.center);
		\draw[line width=2pt] (4.center) to (6);
		\draw[line width=2pt] [in=0, out=-120, looseness=1.20] (6) to (5.center);
		\draw[line width=2pt] [in=0, out=120, looseness=1.20] (6) to (3.center);
		\draw[line width=2pt] (3.center) to (7.center);
	\end{pgfonlayer}
      \end{tikzpicture}}
  \end{aligned}
}

\newcommand{\coassocl}[1]
{
  \begin{aligned}
    \resizebox{#1}{!}{
\begin{tikzpicture}
	\begin{pgfonlayer}{nodelayer}
		\node [style=circ] (0) at (1.125, -0.5) {};
		\node [style=none] (1) at (2.25, -1) {};
		\node [style=none] (2) at (2.25, 0) {};
		\node [style=none] (3) at (1.25, 0.5) {};
		\node [style=none] (4) at (-1, 0) {};
		\node [style=none] (5) at (1, -0.5) {};
		\node [style=circ] (6) at (0, 0) {};
		\node [style=none] (7) at (2.25, 0.5) {};
	\end{pgfonlayer}
	\begin{pgfonlayer}{edgelayer}
		\draw[line width=2pt] [in=180, out=-60, looseness=1.20] (0.center) to (1.center);
		\draw[line width=2pt] [in=180, out=60, looseness=1.20] (0.center) to (2.center);
		\draw[line width=2pt] (4.center) to (6);
		\draw[line width=2pt] [in=180, out=-60, looseness=1.20] (6) to (5.center);
		\draw[line width=2pt] [in=180, out=60, looseness=1.20] (6) to (3.center);
		\draw[line width=2pt] (3.center) to (7.center);
	\end{pgfonlayer}
      \end{tikzpicture}}
  \end{aligned}
}

\newcommand{\coassocr}[1]
{
  \begin{aligned}
    \resizebox{#1}{!}{
\begin{tikzpicture}
	\begin{pgfonlayer}{nodelayer}
		\node [style=circ] (0) at (1.125, 0) {};
		\node [style=none] (1) at (2.25, 0.5) {};
		\node [style=none] (2) at (2.25, -0.5) {};
		\node [style=none] (3) at (1.25, -1) {};
		\node [style=none] (4) at (-1, -0.5) {};
		\node [style=none] (5) at (1, 0) {};
		\node [style=circ] (6) at (0, -0.5) {};
		\node [style=none] (7) at (2.25, -1) {};
	\end{pgfonlayer}
	\begin{pgfonlayer}{edgelayer}
		\draw[line width=2pt] [in=180, out=60, looseness=1.20] (0.center) to (1.center);
		\draw[line width=2pt] [in=180, out=-60, looseness=1.20] (0.center) to (2.center);
		\draw[line width=2pt] (4.center) to (6);
		\draw[line width=2pt] [in=180, out=60, looseness=1.20] (6) to (5.center);
		\draw[line width=2pt] [in=180, out=-60, looseness=1.20] (6) to (3.center);
		\draw[line width=2pt] (3.center) to (7.center);
	\end{pgfonlayer}
      \end{tikzpicture}}
  \end{aligned}
}

\newcommand{\unitl}[1]
{
  \begin{aligned}
    \resizebox{#1}{!}{
\begin{tikzpicture}
	\begin{pgfonlayer}{nodelayer}
		\node [style=none] (0) at (1, -0) {};
		\node [style=circ] (1) at (0.125, -0) {};
		\node [style=circ] (2) at (-1, 0.5) {};
		\node [style=none] (3) at (-1, -0.5) {};
		\node [style=none] (4) at (-2, -0.5) {};
	\end{pgfonlayer}
	\begin{pgfonlayer}{edgelayer}
		\draw[line width=2pt] (0.center) to (1.center);
		\draw[line width=2pt] [in=0, out=120, looseness=1.20] (1.center) to (2.center);
		\draw[line width=2pt] [in=0, out=-120, looseness=1.20] (1.center) to (3.center);
		\draw[line width=2pt] (4.center) to (3.center);
	\end{pgfonlayer}
\end{tikzpicture}
    }
  \end{aligned}
}

\newcommand{\counitl}[1]
{
  \begin{aligned}
    \resizebox{#1}{!}{
\begin{tikzpicture}
	\begin{pgfonlayer}{nodelayer}
		\node [style=none] (0) at (-2, -0) {};
		\node [style=circ] (1) at (-1.125, -0) {};
		\node [style=circ] (2) at (0, 0.5) {};
		\node [style=none] (3) at (0, -0.5) {};
		\node [style=none] (4) at (1, -0.5) {};
	\end{pgfonlayer}
	\begin{pgfonlayer}{edgelayer}
		\draw[line width=2pt] (0.center) to (1.center);
		\draw[line width=2pt] [in=180, out=60, looseness=1.20] (1.center) to (2.center);
		\draw[line width=2pt] [in=180, out=-60, looseness=1.20] (1.center) to (3.center);
		\draw[line width=2pt] (4.center) to (3.center);
	\end{pgfonlayer}
\end{tikzpicture}
    }
  \end{aligned}
}

\newcommand{\commute}[1]
{
  \begin{aligned}
    \resizebox{#1}{!}{
\begin{tikzpicture}
	\begin{pgfonlayer}{nodelayer}
		\node [style=none] (0) at (1.25, -0) {};
		\node [style=circ] (1) at (0.375, -0) {};
		\node [style=none] (2) at (-0.5, -0.5) {};
		\node [style=none] (3) at (-2, 0.5) {};
		\node [style=none] (4) at (-0.5, 0.5) {};
		\node [style=none] (5) at (-2, -0.5) {};
	\end{pgfonlayer}
	\begin{pgfonlayer}{edgelayer}
		\draw[line width=2pt] (0.center) to (1.center);
		\draw[line width=2pt] [in=0, out=-120, looseness=1.20] (1.center) to (2.center);
		\draw[line width=2pt] [in=180, out=0, looseness=1.00] (3.center) to (2.center);
		\draw[line width=2pt] [in=0, out=120, looseness=1.20] (1.center) to (4.center);
		\draw[line width=2pt] [in=0, out=180, looseness=1.00] (4.center) to (5.center);
	\end{pgfonlayer}
\end{tikzpicture}
    }
  \end{aligned}
}

\newcommand{\cocommute}[1]
{
  \begin{aligned}
    \resizebox{#1}{!}{
\begin{tikzpicture}
	\begin{pgfonlayer}{nodelayer}
		\node [style=none] (0) at (-2, -0) {};
		\node [style=circ] (1) at (-1.125, -0) {};
		\node [style=none] (2) at (-0.25, -0.5) {};
		\node [style=none] (3) at (1.25, 0.5) {};
		\node [style=none] (4) at (-0.25, 0.5) {};
		\node [style=none] (5) at (1.25, -0.5) {};
	\end{pgfonlayer}
	\begin{pgfonlayer}{edgelayer}
		\draw[line width=2pt] (0.center) to (1.center);
		\draw[line width=2pt] [in=180, out=-60, looseness=1.20] (1.center) to (2.center);
		\draw[line width=2pt] [in=0, out=180, looseness=1.00] (3.center) to (2.center);
		\draw[line width=2pt] [in=180, out=60, looseness=1.20] (1.center) to (4.center);
		\draw[line width=2pt] [in=180, out=0, looseness=1.00] (4.center) to (5.center);
	\end{pgfonlayer}
\end{tikzpicture}
    }
  \end{aligned}
}

\newcommand{\frobs}[1]
{
  \begin{aligned}
    \resizebox{#1}{!}{
\begin{tikzpicture}
	\begin{pgfonlayer}{nodelayer}
		\node [style=none] (0) at (-1.5, 0.5) {};
		\node [style=circ] (1) at (-0.75, 0.5) {};
		\node [style=none] (2) at (0.25, -0) {};
		\node [style=none] (3) at (0.25, 1) {};
		\node [style=circ] (4) at (1, -0.5) {};
		\node [style=none] (5) at (0, -0) {};
		\node [style=none] (6) at (1.75, -0.5) {};
		\node [style=none] (7) at (0, -1) {};
		\node [style=none] (8) at (1.75, 1) {};
		\node [style=none] (9) at (-1.5, -1) {};
	\end{pgfonlayer}
	\begin{pgfonlayer}{edgelayer}
		\draw[line width=2pt] [in=180, out=-60, looseness=1.20] (1) to (2.center);
		\draw[line width=2pt] [in=180, out=60, looseness=1.20] (1) to (3.center);
		\draw[line width=2pt] (0.center) to (1);
		\draw[line width=2pt] (6.center) to (4);
		\draw[line width=2pt] [in=0, out=120, looseness=1.20] (4) to (5.center);
		\draw[line width=2pt] [in=0, out=-120, looseness=1.20] (4) to (7.center);
		\draw[line width=2pt] (3.center) to (8.center);
		\draw[line width=2pt] (7.center) to (9.center);
	\end{pgfonlayer}
\end{tikzpicture}
    }
  \end{aligned}
}

\newcommand{\frobx}[1]
{
  \begin{aligned}
    \resizebox{#1}{!}{
\begin{tikzpicture}
	\begin{pgfonlayer}{nodelayer}
		\node [style=circ] (0) at (-0.5, -0) {};
		\node [style=none] (1) at (-1.5, -0.5) {};
		\node [style=none] (2) at (-1.5, 0.5) {};
		\node [style=circ] (3) at (0.5, -0) {};
		\node [style=none] (4) at (1.5, -0.5) {};
		\node [style=none] (5) at (1.5, 0.5) {};
	\end{pgfonlayer}
	\begin{pgfonlayer}{edgelayer}
		\draw[line width=2pt] [in=0, out=-120, looseness=1.20] (0.center) to (1.center);
		\draw[line width=2pt] [in=0, out=120, looseness=1.20] (0.center) to (2.center);
		\draw[line width=2pt] [in=180, out=-60, looseness=1.20] (3) to (4.center);
		\draw[line width=2pt] [in=180, out=60, looseness=1.20] (3) to (5.center);
		\draw[line width=2pt] (0) to (3);
	\end{pgfonlayer}
\end{tikzpicture}
    }
  \end{aligned}
}

\newcommand{\frobz}[1]
{
  \begin{aligned}
    \resizebox{#1}{!}{
\begin{tikzpicture}
	\begin{pgfonlayer}{nodelayer}
		\node [style=none] (0) at (1.75, 0.5) {};
		\node [style=circ] (1) at (1, 0.5) {};
		\node [style=none] (2) at (0, -0) {};
		\node [style=none] (3) at (0, 1) {};
		\node [style=circ] (4) at (-0.75, -0.5) {};
		\node [style=none] (5) at (0.25, -0) {};
		\node [style=none] (6) at (-1.5, -0.5) {};
		\node [style=none] (7) at (0.25, -1) {};
		\node [style=none] (8) at (-1.5, 1) {};
		\node [style=none] (9) at (1.75, -1) {};
	\end{pgfonlayer}
	\begin{pgfonlayer}{edgelayer}
		\draw[line width=2pt] [in=0, out=-120, looseness=1.20] (1) to (2.center);
		\draw[line width=2pt] [in=0, out=120, looseness=1.20] (1) to (3.center);
		\draw[line width=2pt] (0.center) to (1);
		\draw[line width=2pt] (6.center) to (4);
		\draw[line width=2pt] [in=180, out=60, looseness=1.20] (4) to (5.center);
		\draw[line width=2pt] [in=180, out=-60, looseness=1.20] (4) to (7.center);
		\draw[line width=2pt] (3.center) to (8.center);
		\draw[line width=2pt] (7.center) to (9.center);
	\end{pgfonlayer}
\end{tikzpicture}
    }
  \end{aligned}
}

\newcommand{\spec}[1]
{
  \begin{aligned}
    \resizebox{#1}{!}{
\begin{tikzpicture}
	\begin{pgfonlayer}{nodelayer}
		\node [style=none] (0) at (1.75, -0) {};
		\node [style=circ] (1) at (0.75, -0) {};
		\node [style=none] (2) at (0, -0.5) {};
		\node [style=none] (3) at (0, 0.5) {};
		\node [style=circ] (4) at (-0.75, -0) {};
		\node [style=none] (5) at (0, -0.5) {};
		\node [style=none] (6) at (-1.75, -0) {};
		\node [style=none] (7) at (0, 0.5) {};
	\end{pgfonlayer}
	\begin{pgfonlayer}{edgelayer}
		\draw[line width=2pt] (0.center) to (1.center);
		\draw[line width=2pt] [in=0, out=-120, looseness=1.20] (1.center) to (2.center);
		\draw[line width=2pt] [in=0, out=120, looseness=1.20] (1.center) to (3.center);
		\draw[line width=2pt] (6.center) to (4);
		\draw[line width=2pt] [in=180, out=-60, looseness=1.20] (4) to (5.center);
		\draw[line width=2pt] [in=180, out=60, looseness=1.20] (4) to (7.center);
	\end{pgfonlayer}
\end{tikzpicture}
    }
  \end{aligned}
}

\newcommand{\extral}[1]
{
  \begin{aligned}
    \resizebox{#1}{!}{
\begin{tikzpicture}
	\begin{pgfonlayer}{nodelayer}
		\node [style=none] (0) at (1.75, -0) {};
		\node [style=circ] (1) at (0.75, -0) {};
		\node [style=circ] (4) at (-0.75, -0) {};
		\node [style=none] (6) at (-1.75, -0) {};
	\end{pgfonlayer}
	\begin{pgfonlayer}{edgelayer}
	  \draw[line width=2pt] (1.center) to (4.center);
	\end{pgfonlayer}
\end{tikzpicture}
    }
  \end{aligned}
}

\def\sigflowadd{
   \node[plus] (adder) {};
   \node[coordinate] (f) at (-0.5,0.65) {};
   \node[coordinate] (g) at (0.5,0.65) {};
   \node (out) [right of=adder] {};
   \node (pref) [left of=f] {};
   \node (preg) [left of=g] {};
   \draw[rounded corners,thick] (pref) -- (f) -- (adder.addend);
   \draw[rounded corners,thick] (preg) -- (g) -- (adder.summand);
   \draw[thick] (adder) -- (out);
}

\def\sigflowaddsideways{
\begin{scope}[rotate=90]
\sigflowadd
\end{scope}
}

\newcommand{\sigflowpicadd}[1]
{
  \begin{aligned}
    \resizebox{#1}{!}{
\begin{tikzpicture}[thick]
	\begin{pgfonlayer}{nodelayer}
\sigflowaddsideways
	\end{pgfonlayer}
\end{tikzpicture}
    }
  \end{aligned}
}

\def\sigflowcoadd{
   \node[coplus] (adder) {};
   \node[coordinate] (g) at (-0.5,0.65) {};
   \node[coordinate] (f) at (0.5,0.65) {};
   \node (out) [left of=adder] {};
   \node (pref) [right of=f] {};
   \node (preg) [right of=g] {};
   \draw[rounded corners,thick] (pref) -- (f) -- (adder.addend);
   \draw[rounded corners,thick] (preg) -- (g) -- (adder.summand);
   \draw[thick] (adder) -- (out);
}

\def\sigflowcoaddsideways{
\begin{scope}[rotate=90]
\sigflowcoadd
\end{scope}
}

\newcommand{\sigflowpiccoadd}[1]
{
  \begin{aligned}
    \resizebox{#1}{!}{
\begin{tikzpicture}[thick, xscale = -1]
	\begin{pgfonlayer}{nodelayer}
\sigflowcoaddsideways
	\end{pgfonlayer}
\end{tikzpicture}
    }
  \end{aligned}
}

\def\sigflowdup{
   \node[delta] (dupe){};
   \node[coordinate] (o2) at (-0.5,-0.65) {};
   \node[coordinate] (o1) at (0.5,-0.65) {};
   \node (in) [right of=dupe] {};
   \node (posto1) [left of=o1] {};
   \node (posto2) [left of=o2] {};
   \draw[rounded corners,thick] (posto1) -- (o1) -- (dupe.left out);
   \draw[rounded corners,thick] (posto2) -- (o2) -- (dupe.right out);
   \draw[thick] (in) -- (dupe);
}

\def\sigflowdupsideways{
\begin{scope}[rotate=90]
\sigflowdup
\end{scope}
}

\newcommand{\sigflowpiccodup}[1]
{
  \begin{aligned}
    \resizebox{#1}{!}{
\begin{tikzpicture}[thick, xscale = -1]
\begin{pgfonlayer}{nodelayer}
\sigflowdupsideways
	\end{pgfonlayer}
\end{tikzpicture}
    }
  \end{aligned}
}

\def\sigflowcodup{
   \node[codelta] (dupe){};
   \node[coordinate] (o1) at (-0.5,-0.65) {};
   \node[coordinate] (o2) at (0.5,-0.65) {};
   \node (in) [left of=dupe] {};
   \node (posto1) [right of=o1] {};
   \node (posto2) [right of=o2] {};
   \draw[rounded corners,thick] (posto1) -- (o1) -- (dupe.left out);
   \draw[rounded corners,thick] (posto2) -- (o2) -- (dupe.right out);
   \draw[thick] (in) -- (dupe);
}

\def\sigflowcodupsideways{
\begin{scope}[rotate=90]
\sigflowcodup
\end{scope}
}

\newcommand{\sigflowpicdup}[1]
{
  \begin{aligned}
    \resizebox{#1}{!}{
\begin{tikzpicture}[thick]
\begin{pgfonlayer}{nodelayer}
\sigflowcodupsideways
	\end{pgfonlayer}
\end{tikzpicture}
    }
  \end{aligned}
}

\def\SigMult{
\begin{scope}[shift={(0,0)}, xscale=1]
\sigflowaddsideways
\end{scope}
\begin{scope}[shift={(0,1.5)}, xscale=-1]
\sigflowdupsideways
\end{scope}
\begin{scope}[shift={(-1.95,.75)}, scale=.63, rotate=90]
\sigflowswap
\end{scope}
\begin{scope}[rotate=90]
   \node[coordinate] (a) at (-.501175,1.15) {};
   \node[coordinate] (b) at (-.501175,2.35) {};
   \draw[rounded corners,thick] (b) -- (a);
   \node[coordinate] (c) at (2.0001,1.15) {};
   \node[coordinate] (d) at (2.0001,2.35) {};
   \draw[rounded corners,thick] (d) -- (c);
\end{scope}
}

\newcommand{\SigMultpic}[1]
{
  \begin{aligned}
    \resizebox{#1}{!}{
\begin{tikzpicture}[thick]
\begin{pgfonlayer}{nodelayer}
  \SigMult
	\end{pgfonlayer}
\end{tikzpicture}
    }
  \end{aligned}
}

\def\SigCoMult{
\begin{scope}[shift={(0,-1.5)}, xscale=-1]
\sigflowcoaddsideways
\end{scope}
\begin{scope}[shift={(0,0)}, xscale=1]
\sigflowcodupsideways
\end{scope}
\begin{scope}[shift={(1.95,-.75)}, scale=.63, rotate = -90]
\sigflowswap
\end{scope}
\begin{scope}[rotate=-90]
   \node[coordinate] (a) at (-.5,1.15) {};
   \node[coordinate] (b) at (-.5,2.35) {};
   \draw[rounded corners,thick] (b) -- (a);
   \node[coordinate] (c) at (2,1.15) {};
   \node[coordinate] (d) at (2,2.35) {};
   \draw[rounded corners,thick] (d) -- (c);
\end{scope}
}

\newcommand{\SigCoMultpic}[1]
{
  \begin{aligned}
    \resizebox{#1}{!}{
\begin{tikzpicture}[thick]
\begin{pgfonlayer}{nodelayer}
  \SigCoMult
	\end{pgfonlayer}
\end{tikzpicture}
    }
  \end{aligned}
}

\def\sigflowzero{
   \node (out1) {};
   \node [zero] (ins1) [left of=out1, shift={(-.2,0)}] {};
   \draw (out1) -- (ins1);
}

\newcommand{\sigflowpiczero}[1]
{
  \begin{aligned}
    \resizebox{#1}{!}{
\begin{tikzpicture}[thick, node distance=0.85cm]
	\begin{pgfonlayer}{nodelayer}
\sigflowzero
	\end{pgfonlayer}
\end{tikzpicture}
    }
  \end{aligned}
}

\def\sigflowcozero{
   \node (out1) {};
   \node [zero] (ins1) [right of=out1, shift={(.2,0)}] {};
   \draw (out1) -- (ins1);
}

\newcommand{\sigflowpiccozero}[1]
{
  \begin{aligned}
    \resizebox{#1}{!}{
\begin{tikzpicture}[thick, node distance=0.85cm]
	\begin{pgfonlayer}{nodelayer}
\sigflowcozero
	\end{pgfonlayer}
\end{tikzpicture}
    }
  \end{aligned}
}

\def\sigflowdel{
   \node (in1) {};
   \node [bang] (del1) [left of=in1, shift={(-.2,0)}] {};
   \draw (in1) -- (del1);
}

\newcommand{\sigflowpiccodel}[1]
{
  \begin{aligned}
    \resizebox{#1}{!}{
\begin{tikzpicture}[thick, node distance=0.85cm]
	\begin{pgfonlayer}{nodelayer}
\sigflowdel
	\end{pgfonlayer}
\end{tikzpicture}
    }
  \end{aligned}
}

\def\sigflowcodel{
   \node (in1) {};
   \node [bang] (del1) [right of=in1, shift={(.2,0)}] {};
   \draw (in1) -- (del1);
}

\newcommand{\sigflowpicdel}[1]
{
  \begin{aligned}
    \resizebox{#1}{!}{
\begin{tikzpicture}[thick, node distance=0.85cm]
	\begin{pgfonlayer}{nodelayer}
\sigflowcodel
	\end{pgfonlayer}
\end{tikzpicture}
    }
  \end{aligned}
}

\def\SigUnit{
\begin{scope}[shift={(0,0)}, xscale=.3]
\sigflowzero
\end{scope}
\begin{scope}[shift={(0,.5)}, xscale=-.3]
\sigflowdel
\end{scope}
}

\newcommand{\SigUnitpic}[1]
{
  \begin{aligned}
    \resizebox{#1}{!}{
\begin{tikzpicture}[thick, node distance=0.85cm]
	\begin{pgfonlayer}{nodelayer}
\SigUnit
	\end{pgfonlayer}
\end{tikzpicture}
    }
  \end{aligned}
}

\def\SigCoUnit{
\begin{scope}[shift={(0,0)}, xscale=.5]
\sigflowcozero
\end{scope}
\begin{scope}[shift={(0,.5)}, xscale=-.5]
\sigflowcodel
\end{scope}
}

\newcommand{\SigCoUnitpic}[1]
{
  \begin{aligned}
    \resizebox{#1}{!}{
\begin{tikzpicture}[thick, node distance=0.85cm]
	\begin{pgfonlayer}{nodelayer}
\SigCoUnit
	\end{pgfonlayer}
\end{tikzpicture}
    }
  \end{aligned}
}

\def\SigLabel{
   \node[coordinate] (in) {};
   \node [multiply] (mult) [right of=in] {$\mathrm{c}$};
   \node[coordinate] (out) [right of=mult] {};
   \node (label) [below of=out, shift={(.1,0)}] {};
   \draw  [thick] (in) -- (mult) -- (out);
}

\newcommand{\SigLabelpic}[1]
{
  \begin{aligned}
    \resizebox{#1}{!}{
\begin{tikzpicture}[thick, node distance=0.85cm]
	\begin{pgfonlayer}{nodelayer}
\SigLabel
	\end{pgfonlayer}
\end{tikzpicture}
    }
  \end{aligned}
}

\def\sigflowswap{
   \node (UpUpLeft) at (-0.4,0.9) {};
   \node [coordinate] (UpLeft) at (-0.4,0.4) {};
   \node (mid) at (0,0) {};
   \node [coordinate] (DownRight) at (0.4,-0.4) {};
   \node (DownDownRight) at (0.4,-0.9) {};
   \node [coordinate] (UpRight) at (0.4,0.4) {};
   \node (UpUpRight) at (0.4,0.9) {};
   \node [coordinate] (DownLeft) at (-0.4,-0.4) {};
   \node (DownDownLeft) at (-0.4,-0.9) {};
   \draw [rounded corners=2mm,thick] (UpUpLeft) -- (UpLeft) -- (mid) --
   (DownRight) -- (DownDownRight) (UpUpRight) -- (UpRight) -- (DownLeft) -- (DownDownLeft);
}

\def\SigLabelEdge{
\begin{scope}[shift={(0,0)}, xscale=-1]
\sigflowcoaddsideways
\end{scope}
\begin{scope}[shift={(4,2)}, xscale=-1]
\sigflowdupsideways
\end{scope}
   \node[bang2] (mult) at (2,1) {$Z$};
   \draw[rounded corners = 2mm,thick] (posto2.east) -- (posto2.west) -- (mult);
   \draw[rounded corners = 2mm,thick] (pref.west) -- (pref.east) -- (mult);
   \draw[rounded corners = 2mm,line width= 1.pt] (preg.west) -- +(3.75,0);
   \draw[rounded corners = 2mm,line width= 1.2pt] (posto1.east) -- +(-3.75,0);
}

\newcommand{\SigLabelEdgepic}[1]
{
  \begin{aligned}
    \resizebox{#1}{!}{
\begin{tikzpicture}[thick]
	\begin{pgfonlayer}{nodelayer}
\SigLabelEdge
	\end{pgfonlayer}
\end{tikzpicture}
    }
  \end{aligned}
}

\newcommand*\pgfdeclareanchoralias[3]{%
  \expandafter\def\csname pgf@anchor@#1@#3\expandafter\endcsname
     \expandafter{\csname pgf@anchor@#1@#2\endcsname}}

\tikzset{
circnode/.style={
  circle, draw=red, very thin, outer sep=0.025em, minimum size=2em,
  fill=red, text centered},
integral/.style={
  regular polygon, regular polygon sides=3, shape border rotate=180, draw=black, very thick,
  outer sep=0.025em, inner sep=0, minimum size=2em, fill=blue!5, text centered},
multiply/.style={
  regular polygon, regular polygon sides=3, shape border rotate=-90, draw=black,  thick,
  outer sep=0.025em, inner sep=0, minimum size=2em, fill=blue!5, text centered},
vsource/.style={
  regular polygon, regular polygon sides=5, shape border rotate=-90, draw=black, thick,
  outer sep=0.025em, inner sep=0, minimum size=2em, fill=blue!5, text centered},
isource/.style={
  regular polygon, regular polygon sides=5, shape border rotate=0, draw=black, very thick,
  outer sep=0.025em, inner sep=0, minimum size=2em, fill=blue!5, text centered},
upmultiply/.style={
  regular polygon, regular polygon sides=3, draw=black, very thick,
  outer sep=0.025em, inner sep=0, minimum size=2em, fill=blue!5, text centered},
zero/.style={
  circle, draw=black, very thick, minimum size=0.15cm, fill=black,
  inner sep=0, outer sep=0},
hole/.style={
  circle, draw=white, very thick, minimum size=0.25cm, fill=white,
  inner sep=0, outer sep=0},
bang/.style={
  circle, draw=black, very thick, minimum size=0.15cm, fill=green!10,
  inner sep=0, outer sep=0},
bang2/.style={
  circle, draw=black, thick, minimum size=0.5cm, fill=green!10,
  inner sep=0, outer sep=0},
delta/.style={
  regular polygon, regular polygon sides=3, minimum size=0.4cm,  shape border rotate=-90, inner
  sep=0, outer sep=0.025em, draw=black, very thick, fill=green!10},
codelta/.style={
  regular polygon, regular polygon sides=3, shape border rotate=90, minimum size=0.4cm,
  inner sep=0, outer sep=0.025em, draw=black, very thick, fill=green!10},
plus/.style={
  regular polygon, regular polygon sides=3, shape border rotate=-90, minimum size=0.4cm,
  inner sep = 0, outer sep=0.025em, draw=black, very thick, fill=black},
coplus/.style={
  regular polygon, regular polygon sides=3, minimum size=0.4cm,shape border rotate=90,
  inner sep = 0, outer sep=0.025em, draw=black, very thick, fill=black},
sqnode/.style={
  regular polygon, regular polygon sides=4, minimum size=2.6em,
  draw=black, very thick, inner sep=0.2em, outer sep=0.025em,
  fill=yellow!10, text centered},
blackbox/.style={
  regular polygon, regular polygon sides=4, minimum size=2.6em,
  draw=black, very thick, inner sep=0.2em, outer sep=0.025em, fill=black},
bigcirc/.style={
  circle, draw=black, very thick, text width=1.6em, outer sep=0.025em,
  minimum height=1.6em, fill=blue!5, text centered}
 }
\pgfdeclareanchoralias{regular polygon}{corner 1}{io}
\pgfdeclareanchoralias{regular polygon}{corner 2}{left out}
\pgfdeclareanchoralias{regular polygon}{corner 3}{right out}
\pgfdeclareanchoralias{regular polygon}{corner 2}{left copy}
\pgfdeclareanchoralias{regular polygon}{corner 3}{right copy}
\pgfdeclareanchoralias{regular polygon}{corner 3}{addend}
\pgfdeclareanchoralias{regular polygon}{corner 2}{summand}
\pgfdeclareanchoralias{regular polygon}{corner 3}{left in}
\pgfdeclareanchoralias{regular polygon}{corner 2}{right in}

\begin{document} 

\title{Props in Network Theory}
\author{
\begin{tabular}{ccc}
John C.\ Baez\footnote{Department of Mathematics, University of California, Riverside CA, 92521, USA}{ }\footnote{Centre for Quantum Technologies, National University of Singapore, 117543, Singapore} & Brandon Coya$^*$ & Franciscus Rebro$^*$
\\
\small baez@math.ucr.edu & \small bcoya001@ucr.edu & \small franciscus.rebro@gmail.com
\end{tabular}
}

\date{January 23, 2018}

\maketitle

\begin{abstract}

\noindent
Long before the invention of Feynman diagrams, engineers were using similar diagrams to 
reason about electrical circuits and more general networks containing mechanical, hydraulic, thermodynamic and chemical components.  We can formalize this reasoning using props: that is, strict symmetric monoidal categories where the objects are natural numbers, with the tensor product of objects given by addition.  In this approach, each kind of network corresponds to a prop, and each network of this kind is a morphism in that prop.  A network with $m$ inputs and $n$ outputs is a morphism from $m$ to $n$, putting networks together in series is composition, and setting them side by side is tensoring.   Here we work out the details of this approach for various kinds of electrical circuits, starting with circuits made solely of ideal perfectly conductive wires, then circuits with passive linear components, and then circuits that also have voltage and current sources.  Each kind of circuit corresponds to a mathematically natural prop.  We describe the `behavior' of these circuits using morphisms between props.  In particular, we give a new construction of the black-boxing functor of Fong and the first author; unlike the original construction, this new one easily generalizes to circuits with nonlinear components.  We also use a morphism of props to clarify the relation between circuit diagrams and the signal-flow diagrams in control theory.  Technically, the key tools are the Rosebrugh--Sabadini--Walters result relating circuits to special commutative Frobenius monoids, the monadic adjunction between props and signatures, and a result saying which symmetric monoidal categories are equivalent to props.
\end{abstract}

\section{Introduction}

In his 1963 thesis, Lawvere \cite{Law} introduced functorial semantics: the use of categories with specified extra structure as `theories' whose `models' are structure-preserving functors into other such categories.   In particular, a `Lawvere theory' is a category with finite cartesian products and a distinguished object $X$ such that each object is a power $X^n$ for some unique $n$.   These can serve as theories of mathematical structures that are sets $X$ equipped with $n$-ary operations $f \maps X^n \to X$ obeying equational laws.  However, structures of a more linear-algebraic nature are often vector spaces equipped with operations of the form $f \maps X^{\otimes m} \to X^{\otimes n}$.  To extend functorial semantics to these, Mac Lane \cite{Ma65} introduced props---or as he called them, `PROPs', where the acronym stands for `products and permutations'.   A prop is a symmetric monoidal category equipped with a distinguished object $X$ such that every object is a tensor power $X^{\otimes n}$ for some unique $n$.   Working with tensor products rather than cartesian products puts operations having multiple outputs on an equal footing with operations having multiple inputs.  

Already in 1949 Feynman had introduced his famous diagrams, which he used to describe theories of elementary particles \cite{Fe}.   For a theory with just one type of particle, Feynman's method amounts to specifying a prop where an operation $f \maps X^{\otimes m} \to X^{\otimes n}$ describes a process with $m$ particles coming in and $n$ going out.   Although Feynman diagrams quickly caught on in physics \cite{Ka}, only in the 1980s did it become clear that they were a method of depicting morphisms in symmetric monoidal categories.  A key step was the work of Joyal and Street \cite{JS1}, which rigorously justified reasoning in any symmetric monoidal category using `string diagrams'---a generalization of Feynman diagrams.

By now, many mathematical physicists are aware of props and the general idea of functorial semantics.  In constrast, props seem to be virtually unknown in engineering.   However, long before physicists began using Feynman diagrams, engineers were using similar diagrams to describe electrical circuits.  In the 1940's Olson \cite{Ol} explained how to apply circuit diagrams to networks of mechanical, hydraulic, thermodynamic and chemical components.   By 1961, Paynter \cite{Pa} had made the analogies between these various systems mathematically precise.  By 1963 Forrester \cite{Fo} was using circuit diagrams in economics, and in 1984 Odum \cite{Od} published an influential book on their use in biology and ecology.    

We can use props to study circuit diagrams of all these kinds.  The underlying mathematics is similar in each case, so we focus on just one example: electrical circuits.  The reader interested in applying props to other examples can do so with the help of various textbooks that explain the analogies between electrical circuits and other networks \cite{Brown,KMR}.

We illustrate the usefulness of props by giving a new, shorter construction of the `black-boxing functor' introduced by Fong and the first author \cite[Thm.\ 1.1]{BF}.  A `black box' is a system with inputs and outputs whose internal mechanisms are unknown or ignored.    We can treat an electrical circuit as a black box by forgetting its inner workings and recording only the relation it imposes between its inputs and outputs.   Circuits of a given kind with inputs and outputs can be seen as morphisms in a category, where composition uses the outputs of the one circuit as the inputs of another.   Black-boxing is a functor from this category to a suitable category of relations.

In an electrical circuit, associated to each wire there is a pair of variables called the potential $\phi$ and current $I$.   When we black-box such a circuit, we record only the relation it imposes between these variables on its input and output wires.  Since these variables come in pairs, this is a relation between even-dimensional vector spaces.  But these vector spaces turn out to be equipped with extra structure: they are `symplectic' vector spaces, meaning they are equipped with a nondegenerate antisymmetric bilinear form.    Black-boxing gives a relation that respects this extra structure: it is a `Lagrangian' relation.   

Why does symplectic geometry show up when we black-box an electrical circuit?   A circuit made of linear resistors acts to minimize the total power dissipated in the form of heat.   More generally, any circuit made of linear resistors, inductors and capacitors obeys a generalization of this `principle of minimum power'.  Whenever a system obeys a minimum principle, it establishes a Lagrangian relation between input and output variables.  This fact was first noticed in classical mechanics, where systems obey the `principle of least action'.   Indeed, symplectic geometry has its origins in classical mechanics \cite{GS,We1,We2}.   But it applies more generally: for any sort of system governed by a minimum principle, black-boxing should give a functor to some category where the morphisms are Lagrangian relations \cite[Sec.\ 13]{BFP}.

The first step toward proving this for electrical circuits is to treat circuits as morphisms in a suitable category.   We start with circuits made only of ideal perfectly conductive wires.  These are morphisms in a prop we call $\Circ$, defined in Section \ref{sec:conductive}.   In Section \ref{sec:black-boxing_conductive} we construct a black-boxing functor
\[   \blacksquare \maps \Circ \to \Lag\Rel_k \]
sending each such circuit to the relation it defines between its input and output potentials and currents.   Here $\Lag\Rel_k$ is a prop with symplectic vector spaces of the form $k^{2n}$ as objects and linear Lagrangian relations as morphisms, and $\blacksquare$ is a morphism of props.  We work in a purely algebraic fashion, so $k$ here can be any field.
 
In Section \ref{sec:black-boxing} we extend black-boxing to a larger class of circuits that include linear resistors, inductors and capacitors.   This gives a new proof of a result of Fong and the first author: namely, there is a morphism of props
\[   \blacksquare \maps \Circ_k \to \Lag\Rel_k \]
sending each such linear circuit to the Lagrangian relation it defines between its input and output potentials and currents.     The ease with which we can extend the black-boxing functor is due to the fact that all our categories with circuits as morphisms are props.  We can describe these props using generators and relations, so that constructing a black-boxing functor simply requires that we choose where it sends each generator and check that all the relations hold.  In Section \ref{sec:signal} we explain how electric circuits are related to signal-flow diagrams, used in control theory.  Finally, in Section \ref{sec:affine}, we illustrate how props can be used to study nonlinear circuits.

\subsubsection*{Plan of the paper}

In Section \ref{sec:circuits} we explain a general notion of `$\L$-circuit' first introduced by  Rosebrugh, Sabadini and Walters \cite{RSW2}.   This is a cospan of finite sets where the apex is the set of nodes of a graph whose edges are labelled by elements of some set $\L$.  In applications to electrical engineering, the elements of $\L$ describe different `circuit elements' such as resistors, inductors and capacitors.  We discuss a symmetric monoidal category $\Circ_\L$ whose objects are finite sets and whose morphisms are (isomorphism classes of) $\L$-circuits.    

In Section \ref{sec:conductive} we consider $\Circ_\L$ when $\L$ is a 1-element set.   We call this category simply $\Circ$.  In applications to engineering, a morphism in $\Circ$ describes circuit made solely of ideal conductive wires.  We show how such a circuit can be simplified in two successive stages, described by symmetric monoidal functors:
\[     \Circ \stackrel{G}{\longrightarrow} \Fin\Cospan \stackrel{H}{\longrightarrow} 
\Fin\Corel .\]
Here $\Fin\Cospan$ is the category of cospans of finite sets, while $\Fin\Corel$ is the category of `corelations' between finite sets.  Corelations, categorically dual to relations, are already known to play an important role in network theory \cite{BF,CF,Fo,Fo3}.  Just as a relation can be seen as a jointly monic span, a corelation can be seen as a jointly epic cospan.  The functor $G$ crushes any graph down to its set of components, while $H$ reduces any cospan to a jointly epic one.

In Section \ref{sec:props} we turn to props.  Propositions \ref{prop:strictification_1} and \ref{prop:strictification_2}, proved in Appendix \ref{sec:symmoncats} with the help of Steve Lack, characterize which symmetric monoidal categories are equivalent to props and which symmetric monoidal functors are isomorphic to morphisms of props.   We use these to find props equivalent to $\Circ_\L$, $\Circ$, $\Fin\Cospan$ and $\Fin\Corel$, and to reinterpret $G$ and $H$ as morphisms of props.   In Section \ref{sec:presenting_props} we discuss presentations of props.  Proposition \ref{prop:monadic}, proved in Appendix \ref{sec:monadic} using a result of Todd Trimble, shows that the category of props is monadic over the category of `signatures', $\Set^{\N \times \N}$.  This lets us work with props using generators and relations.   We conclude by recalling a presentation of $\Fin\Cospan$ due to Lack \cite{La} and a presentation of $\Fin\Corel$ due to Fong and the second author \cite{CF}.

In Section \ref{sec:linear_relations} we introduce the prop $\Fin\Rel_k$.  This prop is equivalent to the symmetric monoidal category with finite-dimensional vector spaces over the field $k$ as objects and linear relations as morphisms, with \emph{direct sum} as its tensor product.  A presentation of this prop was given by Erbele and the first author \cite{BE,E}, and independently by Bonchi,  Soboci\'nski and Zanasi \cite{BSZ,BSZ2,Za}.  In Section \ref{sec:prop_of_circuits} we state a fundamental result of Rosebrugh, Sabadini and Walters \cite{RSW2}.  This result can be seen as giving a presentation of $\Circ_\L$.  Equivalently, it says that $\Circ_\L$ is the coproduct, in the category of props, of $\Fin\Cospan$ and the free prop on a set of unary operations, one for each element of $\L$.  This result makes it easy to construct morphisms from $\Circ_\L$ to other props.

In Section \ref{sec:black-boxing_conductive} we introduce the prop $\Lag\Rel_k$ where morphisms are Lagrangian linear relations between symplectic vector spaces, and
construct the black-boxing functor $\blacksquare \maps \Circ \to \Lag\Rel_k$.    Mathematically, this functor is the composite
\[     \Circ \stackrel{G}{\longrightarrow} \Fin\Cospan \stackrel{H}{\longrightarrow} 
\Fin\Corel \stackrel{K}{\longrightarrow} \Lag\Rel_k \]
where $K$ is a symmetric monoidal functor defined by its action on the generators of $\Fin\Corel$.  In applications to electrical engineering, the black-boxing functor maps any circuit of ideal conductive wires to its `behavior': that is, to the relation that it imposes on the potentials and currents at its inputs and outputs.   

In Section \ref{sec:black-boxing} we extend the black-boxing functor to 
circuits that include resistors, inductors, capacitors and certain other linear circuit
elements.   The most elegant prop having such circuits as morphisms is
$\Circ_k$, meaning $\Circ_\L$ with the label set $\L$ taken to be the field
$k$.   We characterize this black-boxing functor $\blacksquare \maps \Circ_k \to \Lag\Rel_k$ in Theorem \ref{thm:black-boxing_2}.

In Section \ref{sec:signal} we expand the scope of inquiry to include `signal-flow diagrams', a type of diagram used in control theory.  We recall the work of Erbele and others showing that signal-flow diagrams are a syntax for linear relations \cite{BE,E,BSZ,BSZ2,Za}.  Concretely, this means that signal-flow diagrams are morphisms in a free prop $\SigFlow_k$ with the same generators as $\Fin\Rel_k$, but no relations.  There is thus a morphism of props 
\[         \square \maps \SigFlow_k \to \Fin\Rel_k \]
mapping any signal-flow diagrams to the linear relation that it denotes.   It is natural to wonder how this is related to the black-boxing functor 
\[           \blacksquare \maps \Circ_k \to \Lag\Rel_k. \]  
The answer involves the free prop $\Ccirc_k$ which arises when we take the simplest presentation of $\Circ_k$ and omit the relations.  This comes with a map $P \maps \Ccirc_k \to \Circ_k$ which reinstates those relations, and in Theorem \ref{thm:black-and-white-boxing} we show there is a map of props $T \maps \Ccirc_k \to \SigFlow_k$ making this diagram commute:
\[  \xymatrix{
 \Ccirc_k \phantom{ |} \ar[r]^-{P} \ar[d]_{T} & \Circ_k \phantom{} \ar[rr]^-{\blacksquare} & & \Lag\Rel_k \phantom{|} \ar@{^{(}->}[d] \\
      \SigFlow_k \phantom{ |}\ar[rrr]^-{\square} & & & \Fin\Rel_k. \phantom{|}
      }
\]

Finally, in Section \ref{sec:affine} we illustrate how props can also be used to study nonlinear circuits.  Namely, we show how to include voltage and current sources.  Black-boxing these gives Lagrangian \emph{affine} relations between symplectic vector spaces.  

\subsubsection*{Acknowledgements}  

We thank Jason Erbele, Brendan Fong, Blake Pollard, and Nick Woods for helpful conversations.  We are especially grateful to Steve Lack for his help with Appendix \ref{sec:symmoncats}, Charles Rezk and Todd Trimble for their help with Appendix \ref{sec:monadic}, and Jamie Vicary for suggesting early on that we undertake a project of this sort.  We also thank the referees for their corrections.

\section{Circuits}
\label{sec:circuits}

Rosebrugh, Sabadini and Walters \cite{RSW2} explained how to construct a category where the morphisms are circuits made of wires with `circuit elements' on them, such as resistors, inductors and capacitors.    We use their work to construct a symmetric monoidal category $\Circ_\L$ for any set $\L$ of circuit elements. 

To begin with, a circuit consists of a finite graph with `wires' as edges:

\begin{definition}
A \define{graph} is a finite set $E$ of \define{edges} and a finite set $N$ of \define{nodes} equipped with a pair of functions $s,t \maps E \to N$ assigning to 
each edge its \define{source} and \define{target}.  We say that $e \in E$ is an edge \define{from} $s(e)$ \define{to} $t(e)$.  
\end{definition}

We then label the edges with elements of some set $\L$:

\begin{definition}
Given a set $\L$ of \define{labels}, an \define{$\L$-graph} is a graph $s,t\maps E\to N$
 equipped with a function $\ell \maps E \to \L$ assigning a label to each edge.
\end{definition}

For example, we can describe a circuit made from resistors by labeling each
edge with a `resistance' chosen from the set $\L = (0,\infty)$.   Here is a typical
$\L$-graph in this case:

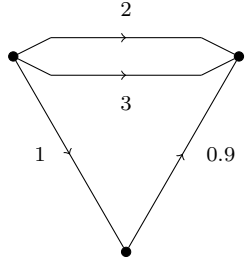
\begin{center}
    \begin{tikzpicture}[circuit ee IEC, set resistor graphic=var resistor IEC graphic]
\scalebox{1}{	
      \node[contact]         (A) at (0,0) {};
      \node[contact]         (B) at (3,0) {};
      \node[contact]         (C) at (1.5,-2.6) {};
      \coordinate         (ua) at (.5,.25) {};
      \coordinate         (ub) at (2.5,.25) {};
      \coordinate         (la) at (.5,-.25) {};
      \coordinate         (lb) at (2.5,-.25) {};
      \path (A) edge (ua);
      \path (A) edge (la);
      \path (B) edge (ub);
      \path (B) edge (lb);
      \path (ua) edge  [->-=.5] node[label={[label distance=1pt]90:{$2$}}] {} (ub);
      \path (la) edge  [->-=.5] node[label={[label distance=1pt]270:{$3$}}] {} (lb);
      \path (A) edge  [->-=.5] node[label={[label distance=2pt]180:{$1$}}] {} (C);
      \path (C) edge  [->-=.5] node[label={[label distance=2pt]0:{$0.9$}}] {} (B);
}
    \end{tikzpicture}
\end{center}

\noindent 
However, to make circuits into the morphisms of a category, we need to specify
some `input' and `output' nodes.  We could do this by picking out two subsets of the set of nodes, but it turns out to be better to use two maps into this set.

\begin{definition}
Given a set $\L$ and finite sets $X$ and $Y$, an \define{$\L$-circuit from $X$ to $Y$} is a cospan of finite sets
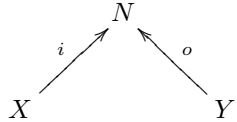
\[
  \xymatrix{
    & N \\
    X \ar[ur]^{i} && Y \ar[ul]_o
  }
\]
together with an $\L$-graph
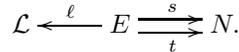
\[ \xymatrix{\L & E \ar@<2.5pt>[r]^{s} \ar@<-2.5pt>[r]_{t} \ar[l]_{\ell} & N.} \] 
We call the sets $i(X)$, $o(Y)$, and $\partial N := i(X) \cup o(Y)$ the \define{inputs},  \define{outputs}, and \define{terminals} of the $\L$-circuit, respectively.
\end{definition}

\noindent
Here is an example of an $\L$-circuit:

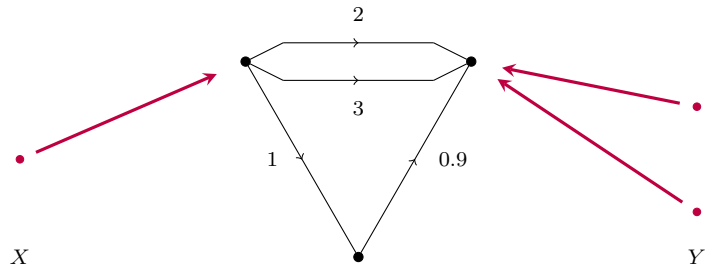
\begin{center}
    \begin{tikzpicture}[circuit ee IEC, set resistor graphic=var resistor IEC graphic]
\scalebox{1}{
      {\node[circle,draw,inner sep=1pt,fill=gray,color=purple]         (x) at
	(-3,-1.3) {};
	\node at (-3,-2.6) {$X$};}
      \node[contact]         (A) at (0,0) {};
      \node[contact]         (B) at (3,0) {};
      \node[contact]         (C) at (1.5,-2.6) {};
      {\node[circle,draw,inner sep=1pt,fill=gray,color=purple]         (y1) at
	(6,-.6) {};
	  \node[circle,draw,inner sep=1pt,fill=gray,color=purple]         (y2) at
	  (6,-2) {};
	  \node at (6,-2.6) {$Y$};}
      \coordinate         (ua) at (.5,.25) {};
      \coordinate         (ub) at (2.5,.25) {};
      \coordinate         (la) at (.5,-.25) {};
      \coordinate         (lb) at (2.5,-.25) {};
      \path (A) edge (ua);
      \path (A) edge (la);
      \path (B) edge (ub);
      \path (B) edge (lb);
      \path (ua) edge  [->-=.5] node[label={[label distance=1pt]90:{$2$}}] {} (ub);
      \path (la) edge  [->-=.5] node[label={[label distance=1pt]270:{$3$}}] {} (lb);
      \path (A) edge  [->-=.5] node[label={[label distance=2pt]180:{$1$}}] {} (C);
      \path (C) edge  [->-=.5] node[label={[label distance=2pt]0:{$0.9$}}] {} (B);
      {
	\path[color=purple, very thick, shorten >=10pt, shorten <=5pt, ->, >=stealth] (x) edge (A);
	\path[color=purple, very thick, shorten >=10pt, shorten <=5pt, ->, >=stealth] (y1) edge (B);
	\path[color=purple, very thick, shorten >=10pt, shorten <=5pt, ->, >=stealth] (y2)
      edge (B);}
}
    \end{tikzpicture}
  \end{center}
\noindent
Note that in this example, two points of $Y$ map to the same node.

We can now compose $\L$-circuits by gluing the outputs of one to the inputs of another.  
For example, we can compose these two $\L$-circuits:
 
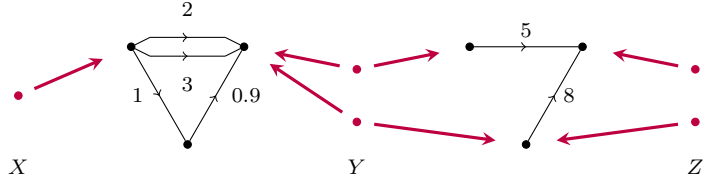
\begin{center}
    \begin{tikzpicture}[circuit ee IEC, set resistor graphic=var resistor IEC
      graphic,scale=.5]
      \node[circle,draw,inner sep=1pt,fill=gray,color=purple]         (x) at
      (-3,-1.3) {};
      \node at (-3,-3.2) {\footnotesize $X$};
      \node[circle,draw,inner sep=1pt,fill]         (A) at (0,0) {};
      \node[circle,draw,inner sep=1pt,fill]         (B) at (3,0) {};
      \node[circle,draw,inner sep=1pt,fill]         (C) at (1.5,-2.6) {};
      \node[circle,draw,inner sep=1pt,fill=gray,color=purple]         (y1) at
      (6,-.6) {};
      \node[circle,draw,inner sep=1pt,fill=gray,color=purple]         (y2) at
      (6,-2) {};
      \node at (6,-3.2) {\footnotesize $Y$};
      \coordinate         (ua) at (.5,.25) {};
      \coordinate         (ub) at (2.5,.25) {};
      \coordinate         (la) at (.5,-.25) {};
      \coordinate         (lb) at (2.5,-.25) {};
      \path (A) edge (ua);
      \path (A) edge (la);
      \path (B) edge (ub);
      \path (B) edge (lb);
      \path (ua) edge  [->-=.5] node[label={[label distance=1pt]90:{\footnotesize $2$}}] {} (ub);
      \path (la) edge  [->-=.5] node[label={[label distance=1pt]270:{\footnotesize $3$}}] {} (lb);
      \path (A) edge  [->-=.5] node[label={[label distance=-1pt]180:{\footnotesize $1$}}] {} (C);
      \path (C) edge  [->-=.5] node[label={[label distance=-1pt]0:{\footnotesize $0.9$}}] {} (B);
      \path[color=purple, very thick, shorten >=10pt, shorten <=5pt, ->, >=stealth] (x) edge (A);
      \path[color=purple, very thick, shorten >=10pt, shorten <=5pt, ->, >=stealth] (y1) edge (B);
      \path[color=purple, very thick, shorten >=10pt, shorten <=5pt, ->, >=stealth] (y2)
      edge (B);

      \node[circle,draw,inner sep=1pt,fill]         (A') at (9,0) {};
      \node[circle,draw,inner sep=1pt,fill]         (B') at (12,0) {};
      \node[circle,draw,inner sep=1pt,fill]         (C') at (10.5,-2.6) {};
      \node[circle,draw,inner sep=1pt,fill=gray,color=purple]         (z1) at
      (15,-.6) {};
      \node[circle,draw,inner sep=1pt,fill=gray,color=purple]         (z2) at (15,-2) {};
      \node at (15,-3.2) {\footnotesize $Z$};
      \path (A') edge  [->-=.5] node[above] {\footnotesize $5$} (B');
      \path (C') edge  [->-=.5] node[right] {\footnotesize $8$} (B');
      \path[color=purple, very thick, shorten >=10pt, shorten <=5pt, ->, >=stealth] (y1) edge (A');
      \path[color=purple, very thick, shorten >=10pt, shorten <=5pt, ->, >=stealth] (y2)
      edge (C');
      \path[color=purple, very thick, shorten >=10pt, shorten <=5pt, ->, >=stealth] (z1) edge (B');
      \path[color=purple, very thick, shorten >=10pt, shorten <=5pt, ->, >=stealth]
      (z2) edge (C');
    \end{tikzpicture}
\end{center}

\noindent
and obtain this one: 

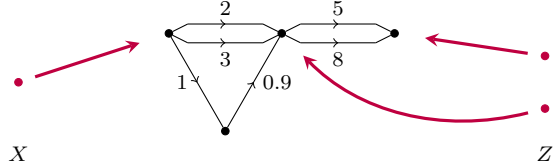
\begin{center}
    \begin{tikzpicture}[circuit ee IEC, set resistor graphic=var resistor IEC
      graphic,scale=0.5]
      \node[circle,draw,inner sep=1pt,fill=gray,color=purple]         (x) at (-4,-1.3) {};
      \node at (-4,-3.2) {\footnotesize $X$};
      \node[circle,draw,inner sep=1pt,fill]         (A) at (0,0) {};
      \node[circle,draw,inner sep=1pt,fill]         (B) at (3,0) {};
      \node[circle,draw,inner sep=1pt,fill]         (C) at (1.5,-2.6) {};
      \node[circle,draw,inner sep=1pt,fill]         (D) at (6,0) {};
      \coordinate         (ua) at (.5,.25) {};
      \coordinate         (ub) at (2.5,.25) {};
      \coordinate         (la) at (.5,-.25) {};
      \coordinate         (lb) at (2.5,-.25) {};
      \coordinate         (ub2) at (3.5,.25) {};
      \coordinate         (ud) at (5.5,.25) {};
      \coordinate         (lb2) at (3.5,-.25) {};
      \coordinate         (ld) at (5.5,-.25) {};
      \path (A) edge (ua);
      \path (A) edge (la);
      \path (B) edge (ub);
      \path (B) edge (lb);
      \path (B) edge (ub2);
      \path (B) edge (lb2);
      \path (D) edge (ud);
      \path (D) edge (ld);
      \node[circle,draw,inner sep=1pt,fill=gray,color=purple]         (z1) at
      (10,-.6) {};
      \node[circle,draw,inner sep=1pt,fill=gray,color=purple]         (z2) at (10,-2) {};
      \node at (10,-3.2) {\footnotesize $Z$};
      \path (ua) edge  [->-=.5] node[above] {\footnotesize $2$} (ub);
      \path (la) edge  [->-=.5] node[below] {\footnotesize $3$} (lb);
      \path (A) edge  [->-=.5] node[left] {\footnotesize $1$} (C);
      \path (C) edge  [->-=.5] node[right] {\footnotesize $0.9$} (B);
      \path (ub2) edge  [->-=.5] node[above] {\footnotesize $5$} (ud);
      \path (lb2) edge  [->-=.5] node[below] {\footnotesize $8$} (ld);
      \path[color=purple, very thick, shorten >=10pt, shorten <=5pt, ->, >=stealth] (x) edge (A);
      \path[color=purple, very thick, shorten >=10pt, shorten <=5pt, ->, >=stealth] (z1)
      edge (D);
      \path[bend left, color=purple, very thick, shorten >=10pt, shorten <=5pt, ->, >=stealth] (z2)
      edge (B);
    \end{tikzpicture}
  \end{center}

We formalize this using composition of cospans.  Given cospans 
$X \stackrel{i}{\rightarrow} N \stackrel{o}{\leftarrow} Y$ and 
$Y\stackrel{i'}{\rightarrow} N' \stackrel{o'}{\leftarrow} Z$, their composite is $X \stackrel{fi}{\longrightarrow} N+_Y N'
\stackrel{f'o'}{\longleftarrow} Z$, where the finite set $N+_Y N'$ and the functions $f$ and $f'$ are defined by this pushout:
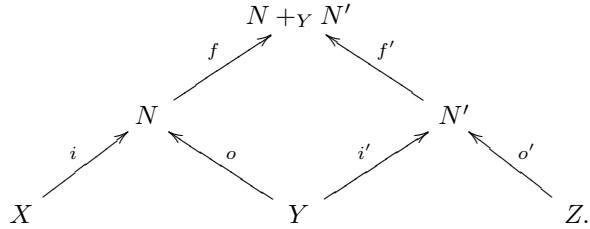
\[
  \xymatrix{
    && N+_YN' \\
    & \; N \; \ar[ur]^f && \; N' \ar[ul]_{f'} \\
 \; \;\; X \ar[ur]^i \;\; && \; Y \; \ar[ul]_o \ar[ur]^{i'} && \;\; Z. \;\; \ar[ul]_{o'}
  }
\]
To make the composite cospan into an $\L$-circuit, we must choose an $\L$-graph whose set of nodes is $N +_Y N'$.  We use this $\L$-graph:
\[   \xymatrix{\L & E + E' \ar@<2.5pt>[rr]^{j(s+s')} \ar@<-2.5pt>[rr]_{j(t+t')} \ar[l]_{\langle \ell, \ell'\rangle} && N +_Y N' } \]
where 
\[  \langle \ell, \ell' \rangle \maps E + E' \to \L \]
is the `copairing' of $\ell$ and $\ell$, i.e., the function that
equals $\ell$ on $E$ and $\ell'$ on $E'$, while the functions
\[  s + s', t + t' \maps E + E' \to N + N' \]
are coproducts, and
\[   j_{N,N'} \maps N + N' \to N +_Y N' \]
is the natural map from the coproduct to the pushout.  

However, the pushout is only unique up to isomorphism, so to make composition
associative we must use isomorphism classes of $\L$-circuits. Since an 
$\L$-circuit can be seen as a graph with an extra structure, an isomorphism 
between $\L$-circuits is an isomorphism of graphs that preserves this extra
structure:

\begin{definition}
Two $\L$-circuits
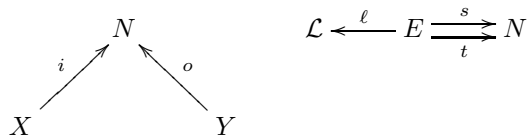
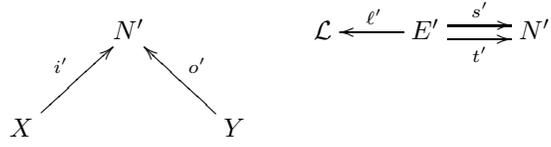
\[
  \xymatrix{
    & N \\
    X \ar[ur]^{i} && Y \ar[ul]_o
  } \qquad
\xymatrix{\L & E \ar@<2.5pt>[r]^{s} \ar@<-2.5pt>[r]_{t} \ar[l]_{\ell} & N} 
\]
and
\[ 
  \xymatrix{
    & N' \\
    X \ar[ur]^{i'} && Y \ar[ul]_{o'}
  } \qquad
 \xymatrix{\L & E' \ar@<2.5pt>[r]^{s'} \ar@<-2.5pt>[r]_{t'} \ar[l]_{\ell'} & N'} \] 
are \define{isomorphic} if there are bijections 
\[   f_E \maps E \to E', \quad f_N \maps N \to N'  \]
such that these diagrams commute:
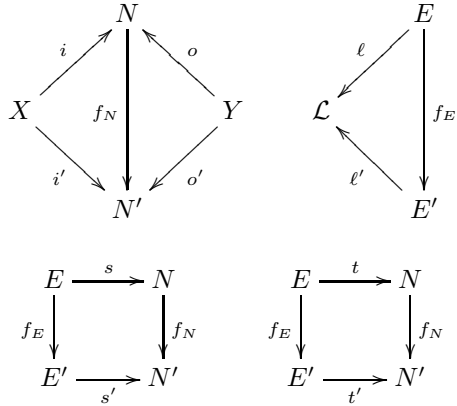
\[ 
  \xymatrix{
    & N \ar[dd]_{f_N} \\
    X \ar[ur]^{i} \ar[dr]_{i'} && Y \ar[ul]_o \ar[dl]^{o'} \\
    & N' 
  }   \qquad  
\xymatrix{&  E \ar[dd]^{f_E} \ar[dl]_{\ell} \\  \L \\  & E' \ar[ul]^{\ell'} }
\]
\[ 
 \xymatrix{E \ar[d]_{f_E} \ar[r]^s & N \ar[d]^{f_N} \\ E' \ar[r]_{s'} & N'} 
\qquad 
\xymatrix{E \ar[d]_{f_E} \ar[r]^t & N \ar[d]^{f_N} \\ E' \ar[r]_{t'} & N'} 
\] 
\end{definition}
\noindent
It is easy to check that composition 
of $\L$-circuits is well-defined and associative at the level of isomorphism classes.  

We can also `tensor' two $\L$-circuits by setting them side by side.  This uses coproducts, or disjoint unions, both of sets and of $\L$-graphs.  For example, tensoring this
$\L$-circuit:
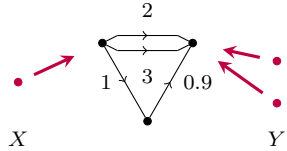
\begin{center}
      \begin{tikzpicture}[circuit ee IEC, set resistor graphic=var resistor IEC
	graphic,scale=.4]
	\node[circle,draw,inner sep=1pt,fill=gray,color=purple]         (x) at
	(-2.8,-1.3) {};
	\node at (-2.8,-3.2) {\footnotesize $X$};
	\node[circle,draw,inner sep=1pt,fill]         (A) at (0,0) {};
	\node[circle,draw,inner sep=1pt,fill]         (B) at (3,0) {};
	\node[circle,draw,inner sep=1pt,fill]         (C) at (1.5,-2.6) {};
	\node[circle,draw,inner sep=1pt,fill=gray,color=purple]         (y1) at
	(5.8,-.6) {};
	\node[circle,draw,inner sep=1pt,fill=gray,color=purple]         (y2) at
	(5.8,-2) {};
	\node at (5.8,-3.2) {\footnotesize $Y$};
	\coordinate         (ua) at (.5,.25) {};
	\coordinate         (ub) at (2.5,.25) {};
	\coordinate         (la) at (.5,-.25) {};
	\coordinate         (lb) at (2.5,-.25) {};
	\path (A) edge (ua);
	\path (A) edge (la);
	\path (B) edge (ub);
	\path (B) edge (lb);
	\path (ua) edge  [->-=.5] node[label={[label distance=0pt]90:{\footnotesize $2$}}] {} (ub);
	\path (la) edge  [->-=.5] node[label={[label distance=0pt]270:{\footnotesize $3$}}] {} (lb);
	\path (A) edge  [->-=.5] node[label={[label distance=-2pt]180:{\footnotesize $1$}}] {} (C);
	\path (C) edge  [->-=.5] node[label={[label distance=-2pt]0:{\footnotesize $0.9$}}] {} (B);
	\path[color=purple, very thick, shorten >=10pt, shorten <=5pt, ->, >=stealth] (x) edge (A);
	\path[color=purple, very thick, shorten >=10pt, shorten <=5pt, ->, >=stealth] (y1) edge (B);
	\path[color=purple, very thick, shorten >=10pt, shorten <=5pt, ->, >=stealth] (y2)
	edge (B);
      \end{tikzpicture}
\end{center}
with this one:
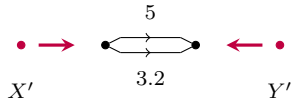
\begin{center}
      \begin{tikzpicture}[circuit ee IEC, set resistor graphic=var resistor IEC
	graphic,scale=.4]
	\node[circle,draw,inner sep=1pt,fill=gray,color=purple]         (x) at
	(-2.8,0) {};
	\node at (-2.8,-1.5) {\footnotesize $X'$};
	\node[circle,draw,inner sep=1pt,fill]         (A) at (0,0) {};
	\node[circle,draw,inner sep=1pt,fill]         (B) at (3,0) {};
	\node[circle,draw,inner sep=1pt,fill=gray,color=purple]         (y1) at
	(5.8,0) {};
	\node at (5.8,-1.5) {\footnotesize $Y'$};
	\coordinate         (ua) at (.5,.25) {};
	\coordinate         (ub) at (2.5,.25) {};
	\coordinate         (la) at (.5,-.25) {};
	\coordinate         (lb) at (2.5,-.25) {};
	\path (A) edge (ua);
	\path (A) edge (la);
	\path (B) edge (ub);
	\path (B) edge (lb);
	\path (ua) edge  [->-=.5] node[label={[label distance=0pt]90:{\footnotesize $5$}}] {} (ub);
	\path (la) edge  [->-=.5] node[label={[label distance=0pt]270:{\footnotesize $3.2$}}] {} (lb);
	\path[color=purple, very thick, shorten >=10pt, shorten <=5pt, ->, >=stealth] (x) edge (A);
	\path[color=purple, very thick, shorten >=10pt, shorten <=5pt, ->, >=stealth] (y1) edge (B);
      \end{tikzpicture}
\end{center}
give this one:
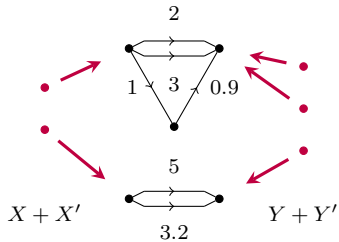
\begin{center}
      \begin{tikzpicture}[circuit ee IEC, set resistor graphic=var resistor IEC
	graphic,scale=.4]
	\node[circle,draw,inner sep=1pt,fill=gray,color=purple]         (x) at
	(-2.8,-1.3) {};
	\node at (-2.8,-5.5) {\footnotesize $X+X'$};
	\node[circle,draw,inner sep=1pt,fill]         (A) at (0,0) {};
	\node[circle,draw,inner sep=1pt,fill]         (B) at (3,0) {};
	\node[circle,draw,inner sep=1pt,fill]         (C) at (1.5,-2.6) {};
	\node[circle,draw,inner sep=1pt,fill=gray,color=purple]         (y1) at
	(5.8,-.6) {};
	\node[circle,draw,inner sep=1pt,fill=gray,color=purple]         (y2) at
	(5.8,-2) {};
	\node at (5.8,-5.5) {\footnotesize $Y+Y'$};
	\coordinate         (ua) at (.5,.25) {};
	\coordinate         (ub) at (2.5,.25) {};
	\coordinate         (la) at (.5,-.25) {};
	\coordinate         (lb) at (2.5,-.25) {};
	\path (A) edge (ua);
	\path (A) edge (la);
	\path (B) edge (ub);
	\path (B) edge (lb);
	\path (ua) edge  [->-=.5] node[label={[label distance=1pt]90:{\footnotesize $2$}}] {} (ub);
	\path (la) edge  [->-=.5] node[label={[label distance=1pt]270:{\footnotesize $3$}}] {} (lb);
	\path (A) edge  [->-=.5] node[label={[label distance=-2pt]180:{\footnotesize $1$}}] {} (C);
	\path (C) edge  [->-=.5] node[label={[label distance=-2pt]0:{\footnotesize $0.9$}}] {} (B);
	\path[color=purple, very thick, shorten >=10pt, shorten <=5pt, ->, >=stealth] (x) edge (A);
	\path[color=purple, very thick, shorten >=10pt, shorten <=5pt, ->, >=stealth] (y1) edge (B);
	\path[color=purple, very thick, shorten >=10pt, shorten <=5pt, ->, >=stealth] (y2)
	edge (B);
	\node[circle,draw,inner sep=1pt,fill=gray,color=purple]         (x2) at
	(-2.8,-2.7) {};
	\node[circle,draw,inner sep=1pt,fill]         (A1) at (0,-5) {};
	\node[circle,draw,inner sep=1pt,fill]         (B1) at (3,-5) {};
	\node[circle,draw,inner sep=1pt,fill=gray,color=purple]         (y3) at
	(5.8,-3.4) {};
	\coordinate         (ua1) at (.5,-4.75) {};
	\coordinate         (ub1) at (2.5,-4.75) {};
	\coordinate         (la1) at (.5,-5.25) {};
	\coordinate         (lb1) at (2.5,-5.25) {};
	\path (A1) edge (ua1);
	\path (A1) edge (la1);
	\path (B1) edge (ub1);
	\path (B1) edge (lb1);
	\path (ua1) edge  [->-=.5] node[label={[label distance=0pt]90:{\footnotesize $5$}}] {} (ub1);
	\path (la1) edge  [->-=.5] node[label={[label distance=0pt]270:{\footnotesize $3.2$}}] {} (lb1);
	\path[color=purple, very thick, shorten >=10pt, shorten <=5pt, ->,
	>=stealth] (x2) edge (A1);
	\path[color=purple, very thick, shorten >=10pt, shorten <=5pt, ->,
	>=stealth] (y3) edge (B1);
      \end{tikzpicture}
\end{center}
In general, given $\L$-circuits
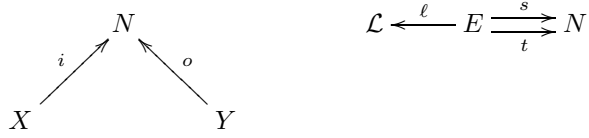
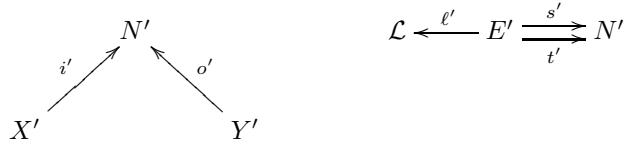
\[  \xymatrix{
    & N \\
    X \ar[ur]^{i} && Y \ar[ul]_o
  } \qquad \qquad \; 
\xymatrix{\L & E \ar@<2.5pt>[r]^{s} \ar@<-2.5pt>[r]_{t} \ar[l]_{\ell} & N} 
\]
and
\[ 
  \xymatrix{
    & N' \\
    X' \ar[ur]^{i'} && Y' \ar[ul]_{o'}
  } \qquad \qquad \;
 \xymatrix{\L & E' \ar@<2.5pt>[r]^{s'} \ar@<-2.5pt>[r]_{t'} \ar[l]_{\ell'} & N'} \] 
their tensor product is
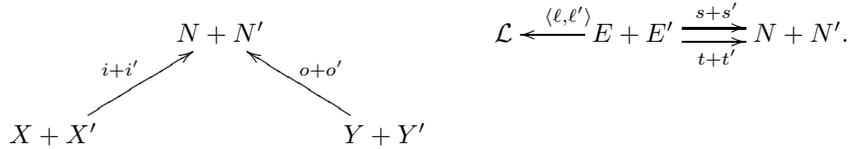
\[  \xymatrix{
    & N + N'\\
    X + X' \ar[ur]^{i+i'} && Y + Y'\ar[ul]_{o+o'}
  } \qquad
\xymatrix{\L & E + E' \ar@<2.5pt>[r]^{s+s'} \ar@<-2.5pt>[r]_{t+t'} \ar[l]_{\langle \ell, \ell' \rangle} & N + N'.} 
\]
This operation is well-defined at the level of isomorphism classes.  Indeed, we obtain a symmetric monoidal category:

\begin{proposition}
\label{prop:lcirc_as_symmoncat}
For any set $\L$, there is a symmetric monoidal category \define{$\Circ_\L$} where the objects are finite sets, the morphisms are isomorphism classes of $\L$-circuits, and composition and the tensor product are defined as above.
\end{proposition}

\begin{proof} 
This was proved by Rosebrugh, Sabadini and Walters \cite{RSW2}, who call this 
category \break $\mathrm{Csp}(\mathrm{Graph}_\mathrm{fin}/\Sigma)$,
where $\Sigma$ is their name for the label set $\L$.  Another style of proof uses Fong's theory of decorated cospans, which implies that $\Circ_\L$ is a special sort of symmetric monoidal
category called a `hypergraph category'.  Fong used this method in the special case where $\L$ is the set of positive real numbers \cite[Sec.\ 5.1]{Fo1}, but the argument does not depend on the nature of the set $\L$.  
\end{proof}

In fact, the work of Courser \cite{Co} and Stay \cite{St} implies that $\Circ_\L$ comes from a compact closed symmetric monoidal bicategory.  The objects in this bicategory are still finite sets, but the morphisms in this bicategory are actual $\L$-circuits, not isomorphism classes.  We expect this bicategorical approach to become important, but for now we are content to work with mere categories.

\section{Circuits of ideal conductive wires}
\label{sec:conductive}

The simplest circuits are those made solely of ideal perfectly conductive wires.  Electrical engineers would consider these circuits trivial.   Nonetheless, they provide the foundation on which all our
results rest.   In this case the underlying graph of the circuit has unlabeled edges---or equivalently, the set of labels has a single element.    So, we make the following definition:

\begin{definition}
\label{defn:Circ}
Let $\Circ$ be symmetric monoidal category $\Circ_\L$ where $\L = \{\ell\}$ is a 1-element set.
\end{definition}

We can treat a morphism in $\Circ$ as an isomorphism class of cospans of finite sets
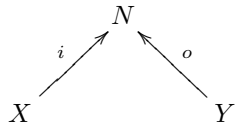
\[
  \xymatrix{
    & N \\
    X \ar[ur]^{i} && Y \ar[ul]_o
  }
\]
together with a graph $\Gamma$ having $N$ as its set of vertices:
\[\Gamma = \{ \xymatrix{ E \ar@<2.5pt>[r]^{s} \ar@<-2.5pt>[r]_{t} & N} \} \] 
For example, a morphism in $\Circ$ might look like this:
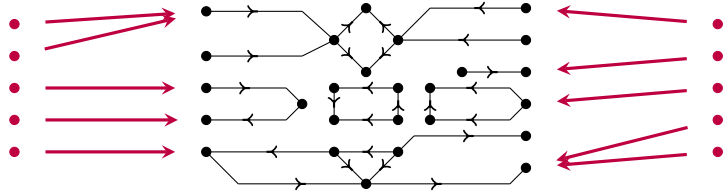
\[
\begin{aligned}
\begin{tikzpicture}[scale=.85, circuit ee IEC]
\begin{pgfonlayer}{nodelayer}
\node [contact, outer sep=5pt] (6) at (-3, 1.2) {};
\node [contact, outer sep=5pt] (7) at (-3, -0.5) {};
\node [contact, outer sep=5pt] (8) at (-3, 0.5) {};
\node [contact, outer sep=5pt] (9) at (-3, -0) {};
\node [contact, outer sep=5pt] (10) at (-3, -1) {};

\node [contact, outer sep=5pt] (15) at (-1, 0.75) {};
	\coordinate         (15a) at (-1.5,1.2) {};
	\coordinate         (15b) at (-1.5,.5) {};
\node [contact, outer sep=5pt] (18) at (0, 0.75) {};
	\coordinate         (18a) at (.5,1.25) {};
\node [contact, outer sep=5pt](14) at (-0.5, 1.25) {};
\node [contact, outer sep=5pt](19) at (-0.5, 0.25) {};
\node [contact, outer sep=5pt](28) at (1, 0.25) {};
\node [contact, outer sep=5pt] (16) at (-1.5, -0.25) {};
	\coordinate         (16a) at (-1.75,0) {};
	\coordinate         (16b) at (-1.75,-.5) {};

\node [contact, outer sep=5pt] (29) at (0.5, -0.5) {};
\node [contact, outer sep=5pt] (11) at (0.5, 0.0) {};
\node [contact, outer sep=5pt] (17) at (-0.5, -1.5) {};
	\coordinate         (17a) at (-2.5,-1.5) {};
\node [contact, outer sep=5pt] (32) at (-1, -1) {};
\node [contact, outer sep=5pt] (33) at (0, -1) {};
	\coordinate         (33a) at (.25,-.75) {};
\node [contact, outer sep=5pt] (-2) at (2, 1.25) {};
\node [contact, outer sep=5pt] (-1) at (2, 0.75) {};
\node [contact, outer sep=5pt] (0) at (2, 0.25) {};
\node [contact, outer sep=5pt] (1) at (2, -0.25) {};
	\coordinate         (1a) at (1.75,-.5) {};
	\coordinate         (1b) at (1.75,0) {};
\node [contact, outer sep=5pt] (2) at (2, -0.75) {};
\node [contact, outer sep=5pt] (3) at (2, -1.25) {};
	\coordinate         (3a) at (1.75,-1.5) {};

\node [contact, outer sep=5pt] (40) at (0, -0.5) {};
\node [contact, outer sep=5pt] (41) at (0, 0.0) {};
\node [contact, outer sep=5pt] (42) at (-1, -0.5) {};
\node [contact, outer sep=5pt] (43) at (-1, 0.0) {};



	        \node[contact, outer sep=5pt,fill=gray,color=purple] (p6) at (-6, 1) {};
	        \node[contact, outer sep=5pt,fill=gray,color=purple] (p7) at (-6, -.5) {};
	        \node[contact, outer sep=5pt,fill=gray,color=purple] (p8) at (-6, .5) {};
	        \node[contact, outer sep=5pt,fill=gray,color=purple] (p9) at (-6, 0) {};
	        \node[contact, outer sep=5pt,fill=gray,color=purple] (p10) at (-6, -1) {};

	        \node[contact, outer sep=5pt,fill=gray,color=purple] (p-2) at (5, 1) {};
	        \node[contact, outer sep=5pt,fill=gray,color=purple] (p0) at (5, .5) {};
	        \node[contact, outer sep=5pt,fill=gray,color=purple] (p1) at (5, 0) {};
	        \node[contact, outer sep=5pt,fill=gray,color=purple] (p2) at (5, -.5) {};
	        \node[contact, outer sep=5pt,fill=gray,color=purple] (p3) at (5, -1) {};

\end{pgfonlayer}
\begin{pgfonlayer}{edgelayer}

\begin{scope}[every node/.style={sloped,allow upside down}]
\draw (6.center) -- node {\midarrow} (15a.center);
\draw  (8.center) -- node {\midarrow} (15b.center);
\draw  (15.center) to (15a.center);
\draw  (15.center) to (15b.center);
\draw  (-2.center) -- node {\midarrow} (18a.center);
\draw  (18.center) to (18a.center); 
\draw  (-1.center) -- node {\midarrow} (18.center);
\draw  (9.center) -- node {\midarrow} (16a.center);
\draw  (16b.center) -- node {\midarrow} (7.center);
\draw  (10.center) to (17a.center);  
\draw  (17a.center) -- node {\midarrow}  (17.center);   
\draw  (33a.center) to (33.center); 
\draw  (33a.center) -- node {\midarrow} (2.center); 
\draw  (17.center) -- node {\midarrow} (3a.center); 
\draw  (3a.center) to (3.center);   
\draw  (28.center) -- node {\midarrow}(0.center); 
\draw (1a.center) to (1.center);    
\draw  (1a.center) -- node {\midarrow} (29.center); 
\draw  (1b.center) -- node {\midarrow} (11.center); 
\draw  (1.center) to (1b.center);   
\draw  (15.center) -- node {\midarrow} (14.center); 
\draw  (15.center) -- node {\midarrow} (19.center); 
\draw  (18.center) -- node {\midarrow} (14.center); 
\draw  (18.center) -- node {\midarrow} (19.center); 
\draw  (16a.center) to (16.center);  
\draw  (16b.center) to (16.center); 
\draw  (29.center) -- node {\midarrow} (11.center); 
\draw  (32.center) -- node {\midarrow}(17.center); 
\draw  (32.center) -- node {\midarrow} (10.center); 
\draw  (33.center) -- node {\midarrow}(17.center); 
\draw  (33.center) -- node {\midarrow} (32.center); 
\draw  (40.center) -- node {\midarrow} (41.center); 
\draw  (40.center) -- node {\midarrow} (42.center); 
\draw  (41.center) -- node {\midarrow} (43.center); 
\draw  (43.center) -- node {\midarrow} (42.center); 
\end{scope}
	       \path[color=purple, very thick, shorten >=5pt, shorten <=5pt, ->,
	>=stealth] (p6) edge (6);
	       \path[color=purple, very thick, shorten >=4pt, shorten <=5pt, ->,
	>=stealth] (p7) edge (7);
	       \path[color=purple, very thick, shorten >=5pt, shorten <=5pt, ->,
	>=stealth] (p8) edge (6);
	       \path[color=purple, very thick, shorten >=5pt, shorten <=5pt, ->,
	>=stealth] (p9) edge (9);
	       \path[color=purple, very thick, shorten >=5pt, shorten <=5pt, ->,
	>=stealth] (p10) edge (10);
	       \path[color=purple, very thick, shorten >=5pt, shorten <=5pt, ->,
	>=stealth] (p-2) edge (-2);

	       \path[color=purple, very thick, shorten >=5pt, shorten <=5pt, ->,
	>=stealth] (p0) edge (0);
	       \path[color=purple, very thick, shorten >=5pt, shorten <=5pt, ->,
	>=stealth] (p1) edge (1);
	       \path[color=purple, very thick, shorten >=5pt, shorten <=5pt, ->,
	>=stealth] (p2) edge (3);
	       \path[color=purple, very thick, shorten >=5pt, shorten <=5pt, ->,
	>=stealth] (p3) edge (3);
\end{pgfonlayer}
\end{tikzpicture}
\end{aligned}
\]
As we shall see in detail later, for the behavior of a circuit made of ideal wires, all that matters about the underlying graph is whether any given pair of nodes lie in the same connected component or not: that is, whether or not there exists a path of edges and their reverses from one node to another.   If they lie in the same component, current can flow between them; otherwise not, and this is all there is to say.  We may thus replace the graph $\Gamma$ by its set of connected components, $\pi_0(\Gamma)$.    There is map $p_\Gamma \maps N \to \pi_0(\Gamma)$ sending each node to the connected component it lies in.  We thus obtain a cospan of finite sets:
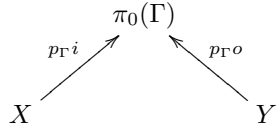
\[
  \xymatrix{
    & \pi_0(\Gamma) \\
    X \ar[ur]^{p_\Gamma i} && Y \ar[ul]_{p_\Gamma o}
  }
\]
In particular, the above example gives this cospan of finite sets:
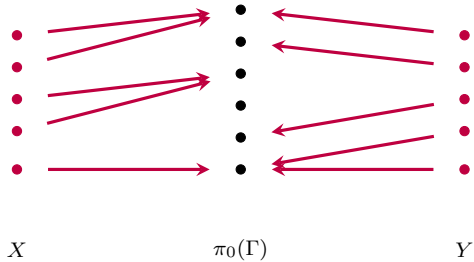
\[
\begin{aligned}
\begin{tikzpicture}[scale=.85, circuit ee IEC]
\begin{pgfonlayer}{nodelayer}
\node [style=none] (5) at (-.5, -2.25) {$\pi_0(\Gamma)$};
\node [style=none] (4a) at (-4, -2.25) {$X$};
\node [style=none] (20a) at (3, -2.25) {$Y$};


	        \node[contact, outer sep=5pt,fill=gray,color=purple] (p6) at (-4, 1.1) {};
	        \node[contact, outer sep=5pt,fill=gray,color=purple] (p7) at (-4, -.4) {};
	        \node[contact, outer sep=5pt,fill=gray,color=purple] (p8) at (-4, .6) {};
	        \node[contact, outer sep=5pt,fill=gray,color=purple] (p9) at (-4, .1) {};
	        \node[contact, outer sep=5pt,fill=gray,color=purple] (p10) at (-4, -1) {};

	        \node[contact, outer sep=5pt,fill=gray,color=purple] (p-2) at (3, 1.1) {};
	        \node[contact, outer sep=5pt,fill=gray,color=purple] (p0) at (3, .6) {};
	        \node[contact, outer sep=5pt,fill=gray,color=purple] (p1) at (3, 0.1) {};
	        \node[contact, outer sep=5pt,fill=gray,color=purple] (p2) at (3, -.4) {};
	        \node[contact, outer sep=5pt,fill=gray,color=purple] (p3) at (3, -1) {};

\node [contact, outer sep=5pt] (m1) at (-.5, 1.5) {};
\node [contact, outer sep=5pt] (m2) at (-.5, 1) {};
\node [contact, outer sep=5pt] (m3) at (-.5, .5) {};
\node [contact, outer sep=5pt] (m4) at (-.5, 0) {};
\node [contact, outer sep=5pt] (m5) at (-.5, -.5) {};
\node [contact, outer sep=5pt] (m6) at (-.5, -1) {};

\end{pgfonlayer}
\begin{pgfonlayer}{edgelayer}
	       \path[color=purple, very thick, shorten >= 5pt, shorten <=5pt, ->,
	>=stealth] (p6) edge (m1);
	       \path[color=purple, very thick, shorten >= 5pt, shorten <=5pt, ->,
	>=stealth] (p7) edge (m3);
	       \path[color=purple, very thick, shorten >=5pt, shorten <=5pt, ->,
	>=stealth] (p8) edge (m1);
	       \path[color=purple, very thick, shorten >=5pt, shorten <=5pt, ->,
	>=stealth] (p9) edge (m3);
	       \path[color=purple, very thick, shorten >=5pt, shorten <=5pt, ->,
	>=stealth] (p10) edge (m6);

	       \path[color=purple, very thick, shorten >=5pt, shorten <=5pt, ->,
	>=stealth] (p-2) edge (m1);

	       \path[color=purple, very thick, shorten >=5pt, shorten <=5pt, ->,
	>=stealth] (p0) edge (m2);
	       \path[color=purple, very thick, shorten >=5pt, shorten <=5pt, ->,
	>=stealth] (p1) edge (m5);
	       \path[color=purple, very thick, shorten >=5pt, shorten <=5pt, ->,
	>=stealth] (p2) edge (m6);
	       \path[color=purple, very thick, shorten >=5pt, shorten <=5pt, ->,
	>=stealth] (p3) edge (m6);
\end{pgfonlayer}
\end{tikzpicture}
\end{aligned}
\]

This way of simplifying a circuit made of ideal wires defines a functor
\[   G \maps \Circ \to \Fin\Cospan  \]
where $\define{\Fin\Cospan}$ is the category where an object is a finite set and a morphism is an isomorphism classes of cospans.  We study this functor in Example \ref{ex:circ_to_fincospan}.  

We can simplify a circuit made of ideal wires even further, because connected 
components of the graph $G$ that contain no terminals---that is, no points in 
$i(X) \cup o(Y)$---can be discarded without affecting the behavior of the circuit.   
In other words, given a cospan of finite sets:
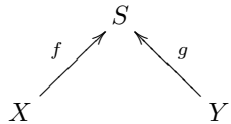
\[
  \xymatrix{
    & S \\
    X \ar[ur]^{f} && Y \ar[ul]_{g}
  }
\]
we can replace $S$ by the subset set $f(X) \cup g(Y)$.   The resulting cospan
is `jointly epic':

\begin{definition}
A cospan $X \stackrel{f}{\rightarrow} S \stackrel{g}{\leftarrow} Y$
in the category of finite sets is \define{jointly epic} if $f(X) \cup g(Y) = S$.   
\end{definition}

\noindent
When we apply this second simplification process to our example, we obtain this
jointly epic cospan:
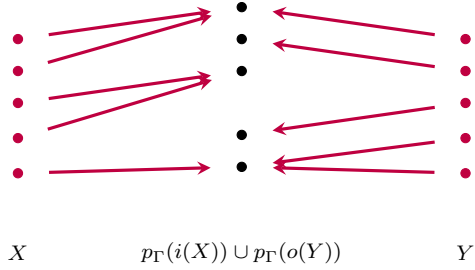
\[
\begin{aligned}
\begin{tikzpicture}[scale=.85, circuit ee IEC]
\begin{pgfonlayer}{nodelayer}
\node [style=none] (5) at (-.5, -2.25) {$p_\Gamma(i(X)) \cup p_\Gamma(o(Y))$};
\node [style=none] (4a) at (-4, -2.25) {$X$};
\node [style=none] (20a) at (3, -2.25) {$Y$};


	   	  \node[contact, outer sep=5pt,fill=gray,color=purple] (p6) at (-4, 1.1) {};
	        \node[contact, outer sep=5pt,fill=gray,color=purple] (p7) at (-4, -.45) {};
	        \node[contact, outer sep=5pt,fill=gray,color=purple] (p8) at (-4, .6) {};
	        \node[contact, outer sep=5pt,fill=gray,color=purple] (p9) at (-4, .1) {};
	        \node[contact, outer sep=5pt,fill=gray,color=purple] (p10) at (-4, -1) {};

	        \node[contact, outer sep=5pt,fill=gray,color=purple] (p-2) at (3, 1.1) {};
	        \node[contact, outer sep=5pt,fill=gray,color=purple] (p0) at (3, .6) {};
	        \node[contact, outer sep=5pt,fill=gray,color=purple] (p1) at (3, 0.1) {};
	        \node[contact, outer sep=5pt,fill=gray,color=purple] (p2) at (3, -.45) {};
	        \node[contact, outer sep=5pt,fill=gray,color=purple] (p3) at (3, -1) {};

		\node [contact, outer sep=5pt] (m1) at (-.5, 1.6) {};
		\node [contact, outer sep=5pt] (m2) at (-.5, 1.1) {};
		\node [contact, outer sep=5pt] (m3) at (-.5, .6) {};
		\node [contact, outer sep=5pt] (m5) at (-.5, -.4) {};
		\node [contact, outer sep=5pt] (m6) at (-.5, -.9) {};

\end{pgfonlayer}
\begin{pgfonlayer}{edgelayer}
	       \path[color=purple, very thick, shorten >=5pt, shorten <=5pt, ->,
	>=stealth] (p6) edge (m1);
	       \path[color=purple, very thick, shorten >=5pt, shorten <=5pt, ->,
	>=stealth] (p7) edge (m3);
	       \path[color=purple, very thick, shorten >=5pt, shorten <=5pt, ->,
	>=stealth] (p8) edge (m1);
	       \path[color=purple, very thick, shorten >=5pt, shorten <=5pt, ->,
	>=stealth] (p9) edge (m3);
	       \path[color=purple, very thick, shorten >=5pt, shorten <=5pt, ->,
	>=stealth] (p10) edge (m6);

	       \path[color=purple, very thick, shorten >=5pt, shorten <=5pt, ->,
	>=stealth] (p-2) edge (m1);

	       \path[color=purple, very thick, shorten >=5pt, shorten <=5pt, ->,
	>=stealth] (p0) edge (m2);
	       \path[color=purple, very thick, shorten >=5pt, shorten <=5pt, ->,
	>=stealth] (p1) edge (m5);
	       \path[color=purple, very thick, shorten >=5pt, shorten <=5pt, ->,
	>=stealth] (p2) edge (m6);
	       \path[color=purple, very thick, shorten >=5pt, shorten <=5pt, ->,
	>=stealth] (p3) edge (m6);
\end{pgfonlayer}
\end{tikzpicture}
\end{aligned}
\]

A jointly epic cospan $X \stackrel{f}{\rightarrow} S \stackrel{g}{\leftarrow} Y$ determines a partition of $X+Y$, where two points $p,q \in X+Y$ are in the same block
of the partition if and only if they map to the same point of $S$ via the function
$\langle f, g \rangle \maps X+Y \to S$.   Moreover, two jointly epic cospans from $X$ to $Y$ are isomorphic, in the usual sense of isomorphism of cospans, if and only if they determine
the same partition of $X+Y$.   For example, the above jointly epic cospan gives this
partition:
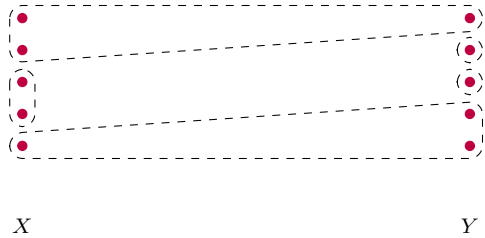
\[
\begin{aligned}
\begin{tikzpicture}[scale=.85, circuit ee IEC]
\begin{pgfonlayer}{nodelayer}
\node [style=none] (4a) at (-4, -2.25) {$X$};
\node [style=none] (20a) at (3, -2.25) {$Y$};


	        \node[contact, outer sep=5pt,fill=gray,color=purple] (p6) at (-4, 1) {};
	        \node[contact, outer sep=5pt,fill=gray,color=purple] (p7) at (-4, -.5) {};
	        \node[contact, outer sep=5pt,fill=gray,color=purple] (p8) at (-4, .5) {};
	        \node[contact, outer sep=5pt,fill=gray,color=purple] (p9) at (-4, 0) {};
	        \node[contact, outer sep=5pt,fill=gray,color=purple] (p10) at (-4, -1) {};

	        \node[contact, outer sep=5pt,fill=gray,color=purple] (p-2) at (3, 1) {};
	        \node[contact, outer sep=5pt,fill=gray,color=purple] (p0) at (3, .5) {};
	        \node[contact, outer sep=5pt,fill=gray,color=purple] (p1) at (3, 0) {};
	        \node[contact, outer sep=5pt,fill=gray,color=purple] (p2) at (3, -.5) {};
	        \node[contact, outer sep=5pt,fill=gray,color=purple] (p3) at (3, -1) {};

\end{pgfonlayer}
\begin{pgfonlayer}{edgelayer}
		\draw [rounded corners=5pt, dashed] 
   (node cs:name=p6, anchor=north west) --
   (node cs:name=p8, anchor=south west) --
   (node cs:name=p-2, anchor=south east) --
   (node cs:name=p-2, anchor=north east) --
   cycle;
		\draw [rounded corners=5pt, dashed] 
   (node cs:name=p9, anchor=north west) --
   (node cs:name=p7, anchor=south west) --
   (node cs:name=p7, anchor=south east) --
   (node cs:name=p9, anchor=north east) --
   cycle;
		\draw [rounded corners=5pt, dashed] 
   (node cs:name=p10, anchor=north west) --
   (node cs:name=p10, anchor=south west) --
   (node cs:name=p3, anchor=south east) --
   (node cs:name=p2, anchor=north east) --
   cycle;
		\draw [rounded corners=5pt, dashed] 
   (node cs:name=p0, anchor=north west) --
   (node cs:name=p0, anchor=south west) --
   (node cs:name=p0, anchor=south east) --
   (node cs:name=p0, anchor=north east) --
   cycle;
		\draw [rounded corners=5pt, dashed] 
   (node cs:name=p1, anchor=north west) --
   (node cs:name=p1, anchor=south west) --
   (node cs:name=p1, anchor=south east) --
   (node cs:name=p1, anchor=north east) --
   cycle;
\end{pgfonlayer}
\end{tikzpicture}
\end{aligned}
\]
Thus, one makes the following definition:

\begin{definition}
Given sets $X$ and $Y$, a \define{corelation} from $X$ to $Y$ is a partition of $X+Y$, or equivalently, an isomorphism class of jointly epic cospans $X \stackrel{f}{\rightarrow} S \stackrel{g}{\leftarrow} Y$.
\end{definition}

The reason for the word `corelation' is that a relation from $X$ to $Y$ corresponds, in 
a similar way, to an isomorphism class of jointly monic spans $X \stackrel{f}{\leftarrow} S \stackrel{g}{\rightarrow} Y$.  Corelations have been studied by Lawvere and Rosebrugh \cite{LR}, and Ellerman \cite{El} used them in an approach
to logic and set theory, dual to the usual approach, in which propositions correspond to partitions rather than subsets.   Fong and the second author \cite{CF,Fo} have continued the study of corelations, and we summarize some of their results in Example \ref{ex:fincorel}.

To begin with, there is a category $\Fin\Corel$ where the objects are finite sets and morphisms are corelations.     We compose two corelations by treating them as
jointly epic cospans:
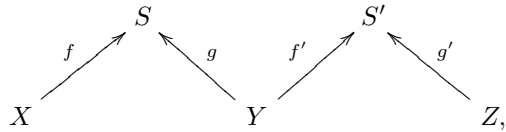
\[
  \xymatrix{
    & \; S \;  && \; S'  \\
 \; \;\; X \ar[ur]^{f} \;\; && \; Y \; \ar[ul]_{g} \ar[ur]^{f'} && \;\; Z, \;\; \ar[ul]_{g'}
  }
\]
forming the pushout:
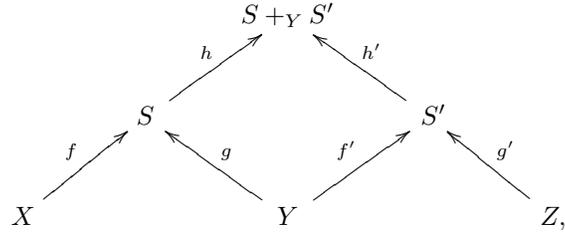
\[
  \xymatrix{
    && S+_Y S' \\
    & \; S \; \ar[ur]^h && \; S' \ar[ul]_{h'} \\
 \; \;\; X \ar[ur]^{f} \;\; && \; Y \; \ar[ul]_{g} \ar[ur]^{f'} && \;\; Z, \;\; \ar[ul]_{g'}
  }
\]
and then forcing the composite cospan to be jointly epic:
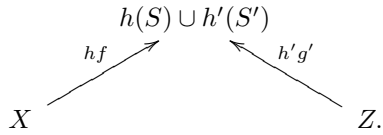
\[
  \xymatrix{
    & h(S) \cup h'(S') \\
     \; X \; \ar[ur]^{hf} && \; Z.\ar[ul]_{h'g'} }
\]
Equivalently, if we treat a corelation $R \maps X \to Y$ as an equivalence relation on 
$X+Y$, we compose corelations $R \maps X \to Y$ and $S \maps Y \to Z$ by forming the 
union of relations $R \cup S$ on $X + Y + Z$, taking its transitive closure to get an equivalence relation, and then restricting this equivalence relation to $X + Z$.   For example, given this corelation from $X$ to $Y$:
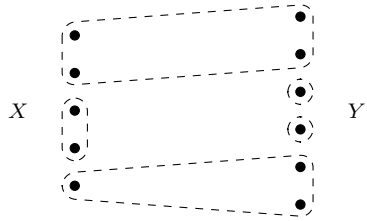
\[
  \begin{tikzpicture}[circuit ee IEC]
	\begin{pgfonlayer}{nodelayer}
		\node [contact, outer sep=5pt] (0) at (-2, 1) {};
		\node [contact, outer sep=5pt] (1) at (-2, 0.5) {};
		\node [contact, outer sep=5pt] (2) at (-2, -0) {};
		\node [contact, outer sep=5pt] (3) at (-2, -0.5) {};
		\node [contact, outer sep=5pt] (4) at (-2, -1) {};
		\node [contact, outer sep=5pt] (5) at (1, 1.25) {};
		\node [contact, outer sep=5pt] (6) at (1, 0.75) {};
		\node [contact, outer sep=5pt] (7) at (1, 0.25) {};
		\node [contact, outer sep=5pt] (8) at (1, -0.25) {};
		\node [contact, outer sep=5pt] (9) at (1, -0.75) {};
		\node [contact, outer sep=5pt] (10) at (1, -1.25) {};
		\node [style=none] (11) at (-2.75, -0) {$X$};
		\node [style=none] (12) at (1.75, -0) {$Y$};
	\end{pgfonlayer}
	\begin{pgfonlayer}{edgelayer}
		\draw [rounded corners=5pt, dashed] 
   (node cs:name=0, anchor=north west) --
   (node cs:name=1, anchor=south west) --
   (node cs:name=6, anchor=south east) --
   (node cs:name=5, anchor=north east) --
   cycle;
		\draw [rounded corners=5pt, dashed] 
   (node cs:name=2, anchor=north west) --
   (node cs:name=3, anchor=south west) --
   (node cs:name=3, anchor=south east) --
   (node cs:name=2, anchor=north east) --
   cycle;
		\draw [rounded corners=5pt, dashed] 
   (node cs:name=4, anchor=north west) --
   (node cs:name=4, anchor=south west) --
   (node cs:name=10, anchor=south east) --
   (node cs:name=9, anchor=north east) --
   cycle;
   		\draw [rounded corners=5pt, dashed] 
   (node cs:name=7, anchor=north west) --
   (node cs:name=7, anchor=south west) --
   (node cs:name=7, anchor=south east) --
   (node cs:name=7, anchor=north east) --
   cycle;
   		\draw [rounded corners=5pt, dashed] 
   (node cs:name=8, anchor=north west) --
   (node cs:name=8, anchor=south west) --
   (node cs:name=8, anchor=south east) --
   (node cs:name=8, anchor=north east) --
   cycle;
	\end{pgfonlayer}
\end{tikzpicture}
\]
and this corelation from $Y$ to $Z$:
\[
\begin{tikzpicture}[circuit ee IEC]
	\begin{pgfonlayer}{nodelayer}
		\node [style=none] (0) at (-2.75, -0) {$Y$};
		\node [style=none] (1) at (1.75, 0) {$Z$};
		\node [contact, outer sep=5pt] (2) at (-2, 1.25) {};
		\node [contact, outer sep=5pt] (3) at (-2, 0.75) {};
		\node [contact, outer sep=5pt] (4) at (-2, 0.25) {};
		\node [contact, outer sep=5pt] (5) at (-2, -0.25) {};
		\node [contact, outer sep=5pt] (6) at (-2, -0.75) {};
		\node [contact, outer sep=5pt] (7) at (-2, -1.25) {};
		\node [contact, outer sep=5pt] (8) at (1, 1) {};
		\node [contact, outer sep=5pt] (9) at (1, 0.5) {};
		\node [contact, outer sep=5pt] (10) at (1, -0) {};
		\node [contact, outer sep=5pt] (11) at (1, -0.5) {};
		\node [contact, outer sep=5pt] (12) at (1, -1) {};
	\end{pgfonlayer}
		\draw [rounded corners=5pt, dashed] 
   (node cs:name=2, anchor=north west) --
   (node cs:name=3, anchor=south west) --
   (node cs:name=8, anchor=south east) --
   (node cs:name=8, anchor=north east) --
   cycle;
		\draw [rounded corners=5pt, dashed] 
   (node cs:name=4, anchor=north west) --
   (node cs:name=4, anchor=south west) --
   (node cs:name=4, anchor=south east) --
   (node cs:name=4, anchor=north east) --
   cycle;
		\draw [rounded corners=5pt, dashed] 
   (node cs:name=5, anchor=north west) --
   (node cs:name=6, anchor=south west) --
   (node cs:name=11, anchor=south east) --
   (node cs:name=10, anchor=north east) --
   cycle;
		\draw [rounded corners=5pt, dashed] 
   (node cs:name=7, anchor=north west) --
   (node cs:name=7, anchor=south west) --
   (node cs:name=12, anchor=south east) --
   (node cs:name=12, anchor=north east) --
   cycle;
		\draw [rounded corners=5pt, dashed] 
   (node cs:name=9, anchor=north west) --
   (node cs:name=9, anchor=south west) --
   (node cs:name=9, anchor=south east) --
   (node cs:name=9, anchor=north east) --
   cycle;
\end{tikzpicture}
\]
we compose them as follows:
\[
  \begin{aligned}
\begin{tikzpicture}[circuit ee IEC]
	\begin{pgfonlayer}{nodelayer}
		\node [contact, outer sep=5pt] (-2) at (1, 1.25) {};
		\node [contact, outer sep=5pt] (-1) at (1, 0.75) {};
		\node [contact, outer sep=5pt] (0) at (1, 0.25) {};
		\node [contact, outer sep=5pt] (1) at (1, -0.25) {};
		\node [contact, outer sep=5pt] (2) at (1, -0.75) {};
		\node [contact, outer sep=5pt] (3) at (1, -1.25) {};
		\node [style=none] (4) at (-2.75, -0) {$X$};
		\node [style=none] (5) at (4.75, -0) {$Z$};
		\node [contact, outer sep=5pt] (6) at (-2, 1) {};
		\node [contact, outer sep=5pt] (7) at (-2, -0.5) {};
		\node [contact, outer sep=5pt] (8) at (-2, 0.5) {};
		\node [contact, outer sep=5pt] (9) at (-2, -0) {};
		\node [contact, outer sep=5pt] (10) at (-2, -1) {};
		\node [contact, outer sep=5pt] (11) at (4, -0) {};
		\node [contact, outer sep=5pt] (12) at (4, -1) {};
		\node [contact, outer sep=5pt] (13) at (4, -0.5) {};
		\node [contact, outer sep=5pt] (14) at (4, 0.5) {};
		\node [contact, outer sep=5pt] (19) at (4, 1) {};
		\node [style=none] (20) at (1, -1.75) {$Y$};
		\node [style=none] (21) at (1, 1.75) {\phantom{$Y$}};
	\end{pgfonlayer}
	\begin{pgfonlayer}{edgelayer}
		\draw [rounded corners=5pt, dashed] 
   (node cs:name=6, anchor=north west) --
   (node cs:name=8, anchor=south west) --
   (node cs:name=-1, anchor=south east) --
   (node cs:name=-2, anchor=north east) --
   cycle;
		\draw [rounded corners=5pt, dashed] 
   (node cs:name=9, anchor=north west) --
   (node cs:name=7, anchor=south west) --
   (node cs:name=7, anchor=south east) --
   (node cs:name=9, anchor=north east) --
   cycle;
		\draw [rounded corners=5pt, dashed] 
   (node cs:name=10, anchor=north west) --
   (node cs:name=10, anchor=south west) --
   (node cs:name=3, anchor=south east) --
   (node cs:name=2, anchor=north east) --
   cycle;
		\draw [rounded corners=5pt, dashed] 
   (node cs:name=-2, anchor=north west) --
   (node cs:name=-1, anchor=south west) --
   (node cs:name=19, anchor=south east) --
   (node cs:name=19, anchor=north east) --
   cycle;
		\draw [rounded corners=5pt, dashed] 
   (node cs:name=0, anchor=north west) --
   (node cs:name=0, anchor=south west) --
   (node cs:name=0, anchor=south east) --
   (node cs:name=0, anchor=north east) --
   cycle;
		\draw [rounded corners=5pt, dashed] 
   (node cs:name=1, anchor=north west) --
   (node cs:name=1, anchor=south west) --
   (node cs:name=1, anchor=south east) --
   (node cs:name=1, anchor=north east) --
   cycle;
		\draw [rounded corners=5pt, dashed] 
   (node cs:name=1, anchor=north west) --
   (node cs:name=2, anchor=south west) --
   (node cs:name=13, anchor=south east) --
   (node cs:name=11, anchor=north east) --
   cycle;
		\draw [rounded corners=5pt, dashed] 
   (node cs:name=3, anchor=north west) --
   (node cs:name=3, anchor=south west) --
   (node cs:name=12, anchor=south east) --
   (node cs:name=12, anchor=north east) --
   cycle;
		\draw [rounded corners=5pt, dashed] 
   (node cs:name=14, anchor=north west) --
   (node cs:name=14, anchor=south west) --
   (node cs:name=14, anchor=south east) --
   (node cs:name=14, anchor=north east) --
   cycle;
	\end{pgfonlayer}
\end{tikzpicture}
\end{aligned}
\:
  =
\:
\begin{aligned}
\begin{tikzpicture}[circuit ee IEC]
	\begin{pgfonlayer}{nodelayer}
		\node [style=none] (0) at (-2.75, -0) {$X$};
		\node [style=none] (1) at (1.75, -0) {$Z$};
		\node [contact, outer sep=5pt] (2) at (-2, 1) {};
		\node [contact, outer sep=5pt] (3) at (-2, -0.5) {};
		\node [contact, outer sep=5pt] (4) at (-2, 0.5) {};
		\node [contact, outer sep=5pt] (5) at (-2, -0) {};
		\node [contact, outer sep=5pt] (6) at (-2, -1) {};
		\node [contact, outer sep=5pt] (7) at (1, -0) {};
		\node [contact, outer sep=5pt] (8) at (1, -1) {};
		\node [contact, outer sep=5pt] (9) at (1, -0.5) {};
		\node [contact, outer sep=5pt] (10) at (1, 0.5) {};
		\node [contact, outer sep=5pt] (13) at (1, 1) {};
		\node [style=none] (20) at (1, -1.75) {\phantom{$Y$}};
		\node [style=none] (21) at (1, 1.75) {\phantom{$Y$}};
	\end{pgfonlayer}
	\begin{pgfonlayer}{edgelayer}
		\draw [rounded corners=5pt, dashed] 
   (node cs:name=2, anchor=north west) --
   (node cs:name=4, anchor=south west) --
   (node cs:name=13, anchor=south east) --
   (node cs:name=13, anchor=north east) --
   cycle;
		\draw [rounded corners=5pt, dashed] 
   (node cs:name=5, anchor=north west) --
   (node cs:name=3, anchor=south west) --
   (node cs:name=3, anchor=south east) --
   (node cs:name=5, anchor=north east) --
   cycle;
		\draw [rounded corners=5pt, dashed] 
   (node cs:name=6, anchor=north west) --
   (node cs:name=6, anchor=south west) --
   (node cs:name=8, anchor=south east) --
   (node cs:name=7, anchor=north east) --
   cycle;
		\draw [rounded corners=5pt, dashed] 
   (node cs:name=10, anchor=north west) --
   (node cs:name=10, anchor=south west) --
   (node cs:name=10, anchor=south east) --
   (node cs:name=10, anchor=north east) --
   cycle;
	\end{pgfonlayer}
\end{tikzpicture}
\end{aligned}
\]

Fong  \cite[Corollary 3.18]{Fo2} showed that the process of taking a cospan of finite sets:
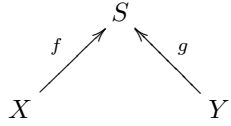
\[
  \xymatrix{
    & S \\
    X \ar[ur]^{f} && Y \ar[ul]_{g}
  }
\]
and making it jointly epic as follows:
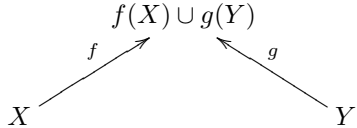
\[
  \xymatrix{
    & f(X) \cup g(Y) \\
    X \ar[ur]^{f} && Y \ar[ul]_{g}
  }
\]
defines a functor
\[      H \maps \Fin\Cospan \to \Fin\Corel  .\]
We describe this functor in more detail in Example \ref{ex:fincospan_to_fincorel}
below. 

In short, we can simplify a circuit made of ideal conductive wires
in two successive stages:
\[     \Circ \stackrel{G}{\longrightarrow} \Fin\Cospan \stackrel{H}{\longrightarrow} 
\Fin\Corel.  \]
In Section \ref{sec:black-boxing_conductive} we define a `black-boxing functor' 
\[   \blacksquare \maps \Circ \to \Fin\Rel_k, \]
which maps any such circuit to its `behavior': that is, the linear relation $R\subseteq k^{2n}$ 
 that it imposes between potentials and currents at its inputs and outputs. 
 We construct this
by composing $HG \maps \Circ \to \Fin\Corel$ with a functor
\[     K \maps \Fin\Corel \to \Fin\Rel_k  .\]
This makes precise the sense in which the behavior of such a circuit depends only on
its underlying corelation.

\section{Props} 
\label{sec:props}

We now introduce the machinery of `props' in order to use generators and
relations to describe the symmetric monoidal categories we are studying.
Mac Lane \cite{Ma65} introduced these structures in 1965 to generalize Lawvere's algebraic theories \cite{Law} to contexts where the tensor product is not cartesian.   He called them `PROPs', which is an acronym for `products and permutations'.   We feel it is finally time to drop the rather ungainly capitalization here and treat props as ordinary mathematical citizens like groups and rings:

\begin{definition}
A \define{prop} is a strict symmetric monoidal category having the natural
numbers as objects, with the tensor product of objects given by addition.  We define a morphism of props to be a strict symmetric monoidal functor that is the identity on objects.  Let \define{$\PROP$} be the category of props.
\end{definition}

A prop $\T$ has, for any natural numbers $m$ and $n$, a homset 
$\T(m,n)$.  In circuit theory we take this to be the set of circuits with $m$ inputs
and $n$ outputs.   In other contexts $\T(m,n)$ can serve as a set of `operations' with 
$m$ inputs and $n$ outputs.  In either case, we often study props
by studying their algebras:
 
\begin{definition}
If $\T$ is a prop and $\C$ is a strict symmetric monoidal category, an 
\define{algebra of} $\T$ \define{in} $\C$ is a strict symmetric monoidal functor 
$F \maps \T \to \C$.   We define a morphism of algebras of $\T$ in  
$\C$ to be a monoidal natural transformation between such functors. 
\end{definition}

For example, if $\T$ is a prop for which morphisms $f \in \T(m,n)$ are circuits of some sort, `black-boxing' should be a strict symmetric monoidal functor $F \maps \T \to \C$ that describes the `behavior' of each circuit as a morphism in $\C$.  We work out many examples of this in the sections to come.

It has long been interesting to take familiar symmetric monoidal categories, treat
them as props, and study their algebras.  The most important feature of a prop is that every object is isomorphic to a tensor power of some chosen object $x$.   However, not every symmetric monoidal category of this sort is a prop, or even isomorphic to one.  There are two reasons: first, it may not be strict, and second, not every object may be \emph{equal} to some tensor power of $x$.  Luckily these two problems are `purely technical', and can be fixed as follows:

\begin{proposition} 
\label{prop:strictification_1} 
A symmetric monoidal category $\C$ is equivalent to a prop if and only if there is an object $x \in \C$ such that every object of $\C$ is isomorphic to $x^{\otimes n} = 
x \otimes (x \otimes (x \otimes \cdots ))$ for some $n \in \N$.  
\end{proposition}

\begin{proof} 
See Section \ref{sec:symmoncats} for a precise statement and proof.  The proof gives a recipe for actually constructing a prop equivalent to $\C$ when this is possible.
\end{proof}

\begin{proposition}
\label{prop:strictification_2}
Suppose $\T$ and $\C$ are props and $F \maps \T \to \C$ is a symmetric monoidal
functor.  Then $F$ is isomorphic to a strict symmetric monoidal functor $G \maps \T \to \C$.   If $F(1) = 1$, then $G$ is a morphism of props.
\end{proposition}

\begin{proof}
See Section \ref{sec:symmoncats} for a precise statement and proof.  
\end{proof}

We can use Proposition\ \ref{prop:strictification_1} to turn some categories we have
been discussing into props:

\begin{example}
\label{ex:finset}
Consider the category of finite sets and functions, made into a symmetric monoidal category where the tensor product of sets is their coproduct, or disjoint union. By Proposition\  \ref{prop:strictification_1} this symmetric monoidal category is equivalent to a prop.  From now on, we use \define{$\Fin\Set$} to stand for this prop.  We identify  this prop with a skeleton of the category of finite sets and functions, having finite ordinals $0, 1, 2, \dots$ as objects.

It is well known that the algebras of $\Fin\Set$ are commutative monoids \cite{Pi}. 
Recall that a \define{commutative monoid} $(x,\mu,\iota)$ in a strict symmetric monoidal
category $\C$ is an object $x \in \C$ together with a \define{multiplication}
$\mu \maps x \otimes x \to x$ and \define{unit} $\iota \maps I \to x$ obeying
the associative, unit and commutative laws.  If we draw the multiplication and unit
using string diagrams:
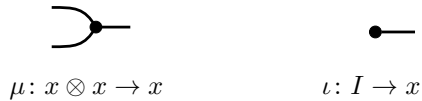
\[
  \xymatrixrowsep{1pt}
  \xymatrix{
    \mult{.075\textwidth} & & \unit{.075\textwidth} \\
    \mu\maps x \otimes x \to x & & \iota\maps I \to x
  }
\]
these laws are as follows:
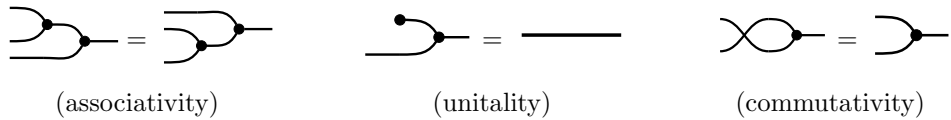
\[
  \xymatrixrowsep{1pt}
  \xymatrixcolsep{25pt}
  \xymatrix{
    \assocl{.1\textwidth} = \assocr{.1\textwidth} & \unitl{.1\textwidth} =
    \idone{.1\textwidth} & \commute{.1\textwidth} = \mult{.07\textwidth} \\
    \textrm{(associativity)} & \textrm{(unitality)} & \textrm{(commutativity)}
  }
\]
where $\swap{1em}$ is the braiding on $x \otimes x$. In addition to the
`upper' or `right' unit law shown above, the mirror image `lower' or `left' 
unit law also holds, due to commutativity and the naturality of the braiding.  

To get a sense for why the algebras of $\Fin\Set$ are commutative monoids, 
write $m \maps 2 \to 1$ and $i\maps 0 \to 1$ for the unique
functions of their type. Then given a strict symmetric monoidal functor
$F\maps \Fin\Set \to \C$, the object $F(1)$ becomes a commutative monoid in $\C$ with 
multiplication $F(m) \maps F(1) \otimes F(1) \to F(1)$ and unit $F(i)$.   The associative, 
unit and commutative laws are easy to check.   Conversely---and this requires 
more work---any commutative monoid in $\C$ arises in this way from a unique 
choice of $F$.  

Similarly, morphisms between algebras of $\Fin\Set$ in $\C$ correspond to morphisms 
of commutative monoids in $\C$. We thus say that $\Fin\Set$ is \define{the prop for 
commutative monoids}.
\end{example}

\begin{example}
\label{ex:fincospan}
Consider the category of finite sets and isomorphism classes of cospans, made into a symmetric monoidal category where the tensor product is given by disjoint union.   By Proposition\ \ref{prop:strictification_1} this symmetric monoidal category is equivalent to a prop.    We henceforth call this prop \define{$\Fin\Cospan$}.   As with $\Fin\Set$, we can identify the objects of $\Fin\Cospan$ with finite ordinals.

Lack \cite{La} has shown that $\Fin\Cospan$ is the prop for special commutative
Frobenius monoids.   To understand this, recall that a \define{cocommutative 
comonoid} $(x,\delta,\epsilon)$ in $\C$ is an object $x \in \C$ together with morphisms
\[
  \xymatrixrowsep{1pt}
  \xymatrix{
    \comult{.075\textwidth} & & \counit{.075\textwidth} \\
    \delta\maps x \to x \otimes x & & \epsilon\maps x \to I
  }
\]
obeying these equations:
\[
  \xymatrixrowsep{1pt}
  \xymatrixcolsep{25pt}
  \xymatrix{
    \coassocl{.1\textwidth} = \coassocr{.1\textwidth} & \counitl{.1\textwidth} =
    \idone{.1\textwidth} & \cocommute{.1\textwidth} = \comult{.07\textwidth} \\
    \textrm{(coassociativity)} & \textrm{(counitality)} &
    \textrm{(cocommutativity)}
  }
\]
A \define{commutative Frobenius monoid} in $\C$ is a
commutative monoid $(x,\mu,\iota)$ and a cocommutative comonoid 
$(x,\delta,\epsilon)$ in $\C$ which together obey the \define{Frobenius law}:
  \[
  \xymatrix{
    \frobs{.1\textwidth} = \frobx{.1\textwidth} = \frobz{.1\textwidth} \\
  }
  \]
In fact if any two of these expressions are equal so are all three.  Furthermore, a 
monoid and comonoid obeying the Frobenius  law is commutative if and only if it is cocommutative.   A \define{morphism} of commutative Frobenius monoids in $\C$ is
a morphism between the underlying objects that preserves the multiplication, unit,
comultiplication and counit.  In fact, any morphism of commutative Frobenius 
monoids is an isomorphism.

A commutative Frobenius monoid is \define{special} if comultiplication followed 
by multiplication is the identity:
\[
  \xymatrix{
    \spec{.1\textwidth} =  \idone{.1\textwidth} 
  }
\]
To get a sense for why the algebras of $\Fin\Cospan$ are special commutative 
Frobenius monoids, note that any function $f \maps X \to Y$ between finite sets
gives rise to a cospan $X \stackrel{f}{\rightarrow} Y \stackrel{1}{\leftarrow} Y$
but also a cospan $Y \stackrel{1}{\rightarrow} Y \stackrel{f}{\leftarrow} X$.
The aforementioned functions $m \maps 2 \to 1$ and $i \maps 0 \to 1$ thus
give cospans $\mu \maps 2 \to 1$, $\iota \maps 0 \to 1$ but also cospans
$\delta \maps 1 \to 2$, $\epsilon \maps 1 \to 0$.  These four cospans make the 
object $1 \in \Fin\Cospan$ into a special commutative Frobenius monoid.  

This much is easy; Lack's accomplishment was to find an elegant
proof that the category of algebras of $\Fin\Cospan$ in any strict symmetric monoidal
category $\C$ is equivalent to the category of special commutative Frobenius
monoids in $\C$.  We thus say that $\Fin\Cospan$ is the \define{prop for 
special commutative Frobenius monoids}.
\end{example}

\begin{example}
\label{ex:fincorel}
Consider the category of finite sets and corelations, again made into a symmetric monoidal category where the tensor product is given by disjoint union.  
By Proposition\ \ref{prop:strictification_1}, this symmetric monoidal category is 
equivalent to a prop.   We call this prop \define{$\Fin\Corel$}, and we identify this prop with a skeleton of the category of finite sets and corelations having finite ordinals as objects.

A special commutative Frobenius monoid is \define{extraspecial} if the unit followed by the counit is the identity:
\[
  \xymatrix{
\extral{.1\textwidth} = 
  }
\]
where the blank at right denotes the identity on the unit object for the tensor 
product.  The Frobenius law and the special law go back at least to Carboni and 
Walters \cite{CW}, but this `extra' law is newer, appearing under this name in the 
work of Baez and Erbele \cite{BE}, as the `bone law' in \cite{BSZ,FRS}, and 
as the `irredundancy law' in \cite{Za}.  In terms of circuits, it says that a loose wire
not connected to anything else can be discarded without affecting the behavior of
the circuit.

Fong and the second author \cite{CF} showed that the category of algebras of $\Fin\Corel$  in a strict symmetric monoidal category $\C$ is equivalent to the category of extraspecial commutative Frobenius monoids in $\C$.  We thus say that $\Fin\Corel$ is the \define{prop for extraspecial commutative Frobenius monoids}.
\end{example}
 
We conclude with two examples that are in some sense `dual' to the previous two.  These are important in their own right, but also useful in understanding the prop of linear relations, discussed in Section \ref{sec:linear_relations}.
 
\begin{example}
\label{ex:finspan}
Consider the category of finite sets and isomorphism classes of spans, made into a symmetric monoidal category where the tensor product is given by disjoint union.   By Proposition\ \ref{prop:strictification_1} this symmetric monoidal category is equivalent to a prop.    We call this prop \define{$\Fin\Span$}.   

Lack \cite{La} has shown that $\Fin\Span$ is the prop for bicommutative
bimonoids.   Recall that a \break \define{bimonoid} is a monoid and comonoid where all the monoid operations are comonoid homomorphisms---or equivalently, all the comonoid operations are monoid homomorphisms.  
A bimonoid is \define{bicommutative} if its underlying monoid is commutative and its underlying comonoid is cocommutative. 
\end{example}

\begin{example}
\label{ex:finrel}
Consider the category of finite sets and relations, made into a symmetric monoidal category where the tensor product is given by disjoint union.   Yet again this is equivalent to a prop, which we call \define{$\Fin\Rel$}.    Fong and the second author \cite{CF} have shown that $\Fin\Rel$ is the prop for special bicommutative bimonoids.  Here a bimonoid is \define{special} if its comultiplication followed by its multiplication is the identity.  The `extra' law, saying that the unit followed by the counit is the identity, holds in any bimonoid.
\end{example}

In short, we have this picture \cite{CF}:
\[
\begin{tabular}{c|c}
  \textbf{spans} & \textbf{cospans} \\
  extra bicommutative bimonoids & special bicommutative Frobenius monoids \\ \hline
  \textbf{relations} & \textbf{corelations} \\
  extraspecial bicommutative bimonoids & extraspecial bicommutative Frobenius monoids \\
\end{tabular}
\]

\noindent
Here we are making the pattern clearer by speaking of `extra' and `extraspecial' bicommutative bimonoids, even though the `extra' law holds for any bimonoid.  We also mention `bicommutative' Frobenius monoids, even though any commutative Frobenius monoid is automatically bicommutative.

\section{Presenting props}
\label{sec:presenting_props}

Just as we can present groups using generators and relations, we can do the same for props.  Such presentations play a key role in our work.  We can handle them using the tools of universal algebra, which in its modern form involves monads arising 
from algebraic theories.   Just as we can talk about the free group on a set, we
can talk about the free prop on a `signature'.  The underlying signature of a prop $\T$ is roughly 
the  collection of all its homsets $\T(m,n)$ where $m,n \in \N$.  It is good to
organize this collection into a functor from $\N \times \N$ to $\Set$, where we 
consider $\N \times \N$ as a discrete category.   Thus:

\begin{definition} We define the category of \define{signatures} to be the functor category $\Set^{\N \times \N}$.  
\end{definition}

A signature $\Sigma$ has \define{elements}, which are the elements of the sets
$\Sigma(m,n)$ for $m,n\in \N$.   Given an element $f \in \Sigma(m,n)$ we write 
$f \maps m \to n$.

\begin{proposition}
\label{prop:monadic}
There is a forgetful functor
\[           U \maps \PROP \to \Set^{\N \times \N}  \]
sending any prop to its underlying signature and any morphism of props to its
underlying morphism of signatures.  This functor is monadic: that is, it has a left
adjoint 
\[          F \maps \Set^{\N \times \N} \to \PROP  \]
and the category of algebras of the monad $UF$ is equivalent, via the canonical
comparison functor, to the category $\PROP$.
\end{proposition}

\begin{proof} This follows from the fact that props are algebras of a multi-sorted or
`typed' Lawvere theory, together with a generalization of Lawvere's fundamental
result \cite{Law} to this context.  We prove this in Appendix \ref{sec:monadic}.  \end{proof}

For any signature $\Sigma$, we call $F\Sigma$ the \define{free prop} on $\Sigma$.   
The first benefit of the previous theorem is that it lets us describe any prop using a
presentation.  In other words, any prop can be obtained from a free prop by taking
a coequalizer:

\begin{corollary}
\label{cor:presentation}
The category $\PROP$ is cocomplete, and any prop $\T$ is the coequalizer of some diagram
\[
    \xymatrix{
      F(E) \ar@<-.5ex>[r]_{\rho} \ar@<.5ex>[r]^{\lambda} & F(\Sigma).
    }
\]
\end{corollary}

\begin{proof}
Cocompleteness follows from a result of Trimble \cite[Prop.\ 3.1]{Tr}: the
category of algebras of a multi-sorted Lawvere theory in a category $C$ is cocomplete if 
$C$ is cocomplete and has finite products, with finite products distributing over colimits.  The 
latter fact holds simply because $\PROP$ is equivalent to the category of
algebras of a monad: as noted by Barr and Wells 
\cite[Sec.\ 3.2, Prop.\ 4]{BW}, we can take $\Sigma = U(\T)$
and $E = UFU(\T)$, with $\lambda = FU\epsilon_{\T}$ and $\rho = \epsilon_{FU(\T)}$,
where $\epsilon \maps FU \To 1$ is the counit of the adjunction in Proposition
\ref{prop:monadic}.  \end{proof}

Here elements of the signature $\Sigma$ serve as generators for $\T$, while 
elements of $E$ give relations---though as we will often be discussing 
relations of another sort, we prefer to call elements of $E$ `equations'.
The idea is that given $e \in E(m,n)$, the morphisms $\lambda(e)$ and $\rho(e)$ 
in the free prop on $\Sigma$ are set equal to each other in $T$.  
To illustrate this idea, we give presentations for the props $\Fin\Set$, 
$\Fin\Cospan$ and $\Fin\Corel$:

\begin{example} 
To present $\Fin\Set$ we can take $\Sigma_0$ to be the signature with
$\mu \maps 2 \to 1$ and $\iota \maps 0 \to 1$ as its only elements,
and let the equations $\lambda, \rho \maps F(E_0) \to F(\Sigma_0)$ be those governing a commutative monoid with multiplication $\mu$ and unit $\iota$, namely the associative law:
\[        \mu (\mu \otimes 1) = \mu (1 \otimes \mu), \]
the left and right unit laws:
\[        \mu (\iota \otimes 1) = 1  = \mu (1 \otimes \iota), \]
and the commutative law:
\[        \mu = \mu B ,\]
where $1$ above denotes the identity morphism on the object $1 \in F(\Sigma_0)$, while $B$ is the braiding on two copies of this object.    We obtain a coequalizer diagram
\[
    \xymatrix{
      F(E_0) \ar@<-.5ex>[r]_{\rho} \ar@<.5ex>[r]^{\lambda} & F(\Sigma_0) \ar[r] & \Fin\Set.
    }
\]
\end{example}

\begin{example}
\label{ex:fincospan_presentation}
To present $\Fin\Cospan$ we can take $\Sigma$ to be the signature with elements $\mu \maps 2\to 1, \iota \maps 0 \to 1, \delta \maps 1 \to 2, \epsilon \maps 1\to 0,$ and let the equations be those governing a special commutative Frobenius monoid with multiplication $\mu$, unit $\iota$, comultiplication $\delta$ and counit $\epsilon$.  We obtain a coequalizer diagram
\[
    \xymatrix{
      F(E) \ar@<-.5ex>[r]_{\rho} \ar@<.5ex>[r]^{\lambda} & F(\Sigma) \ar[r] & \Fin\Cospan.
    }
\]
\end{example}

\begin{example}
\label{ex:fincospan_to_fincorel}
To present $\Fin\Corel$ we can use the same signature $\Sigma$ as for $\Fin\Cospan$, but the equations $E'$ include one additional equation, the so-called \define{extra} law:
\[    \epsilon \iota = 1  \]
in the definition of extraspecial commutative Frobenius monoids. 
Fong and the second author \cite{CF} show we obtain a coequalizer diagram
\[
    \xymatrix{
      F(E')\ar@<-.5ex>[r]_{\rho'} \ar@<.5ex>[r]^{\lambda'} & F(\Sigma) \ar[r] & \Fin\Corel.
    }
\]
They also show that the inclusion $i \maps E \to E'$ gives rise to a diagram of this
form:
\[
    \xymatrix{
      F(E)\ar@<-.5ex>[r]_{\rho} \ar[d]_{Fi} \ar@<.5ex>[r]^{\lambda} & F(\Sigma) \ar[d]^{F1} \ar[r] & \Fin\Cospan \\
F(E')\ar@<-.5ex>[r]_{\rho'} \ar@<.5ex>[r]^{\lambda'} & F(\Sigma)
 \ar[r] & \Fin\Corel
    }
\]
where the two squares at left, one involving $\lambda$ and $\lambda'$ and the other involving $\rho$ and $\rho'$, each commute.  Thus, by the universal property of the coequalizer, we obtain a morphism $H \maps \Fin\Cospan \to \Fin\Corel$.  This expresses $\Fin\Corel$ as a `quotient' of $\Fin\Cospan$.
\end{example}

The morphism of props $H \maps \Fin\Cospan \to \Fin\Corel$ is, in fact, unique:

\begin{proposition}
\label{prop:fincospan_to_fincorel}
There exists a unique morphism of props $H \maps \Fin\Cospan \to \Fin\Corel$.
\end{proposition}

\begin{proof}
It suffices to show that $H$ is uniquely determined on the generators $\mu , \iota, \delta, \epsilon$.   By Example \ref{ex:fincospan_presentation}, 
$H(\mu) \maps 2 \to 1, H(\iota) \maps 0 \to 1, H(\delta) \maps 1 \to 2$ and $H(\epsilon) \maps 1 \to 0$ must make $1 \in \Fin\Corel$ into a special commutative
Frobenius monoid.   

There is a unique corelation from $0$ to $1$, so $H(\iota)$ is uniquely determined.
Similarly, $H(\epsilon)$ is uniquely determined.  Each partition of a 3-element
set gives a possible choice for $H(\mu)$.   If we write $H(\mu)$ as 
a partition of the set $2 + 1$ with $2 = \{a,b\}$ and $1 = \{c\}$ these choices are:
\[      \{\{a,b,c\}\},  \;\;
\{\{a\}, \{b\}, \{c\}\} , \;\;
\{\{a,b\}, \{c\} \},  \;\;
\{\{a\}, \{b,c\} \},  \;\;
\{\{a,c\}, \{b\}\} .
\]
The commutative law rules out those that are not invariant under switching
$a$ and $b$, leaving us with these:
\[      \{\{a,b,c\}\},  \;\;
\{\{a\}, \{b\}, \{c\}\}, \;\;
\{\{a,b\}, \{c\} \} .
\]
The associative law rules out the last of these.  Either of the remaining two
makes $1 \in \Fin\Corel$ into a commutative monoid.   Dually, there are two
choices for $H(\delta)$ making $1$ into a commutative comonoid.  However, 
the `special' law demands that $H(\mu) H(\delta)$ be the identity corelation.
This forces $H(\mu)$ to be the partition with just one block, namely $\{\{a,b,c\}\}$,
and similarly for $H(\delta)$.  
\end{proof}

\section{The prop of linear relations}
\label{sec:linear_relations}

Since the simplest circuits impose \emph{linear} relations between potentials and currents at their terminals, the study of circuits forces us to think carefully about linear relations.  As we shall see, for any field $k$ there is a prop $\Fin\Rel_k$ where a morphism $f \maps m \to n$ is a linear relation from $k^m$ to $k^n$.  A presentation for this prop has been worked out by Erbele and the first author \cite{BE,E} and independently by Bonchi, Soboci\'nski and Zanasi \cite{BSZ,BSZ2,Za}.  Since we need some facts about this presentation to describe the black-boxing of circuits, we recall it here.

For any field $k$, there is a category where an object is a finite-dimensional vector space over $k$, while a morphism from $U$ to $V$ is a \define{linear relation}, meaning a linear subspace $L \subseteq U \oplus V$.    We write a linear relation
from $U$ to $V$ as $L \maps U \asrelto V$ to distinguish it from a linear map.
Since the direct sum $U \oplus V$ is also the cartesian product of $U$ and $V$, a linear relation is a relation in the usual sense, and we can compose linear relations $L \maps U \asrelto V$ and $L' \maps V \asrelto W$ in the usual way:
\[         L'L = \{(u,w) \colon \; \exists\; v \in V \;\; (u,v) \in L \textrm{ and } 
(v,w) \in L'\} \]
the result being a linear relation $L'L \maps U \asrelto W$.  Given linear relations 
$L \maps U \asrelto V$ and $L' \maps U' \asrelto V'$, the direct sum of subspaces gives a linear relation $L \oplus L' \maps U \oplus U' \asrelto V \oplus V'$, and this gives our category a symmetric monoidal structure.  By Proposition \ref{prop:strictification_1}, this symmetric monoidal category is equivalent to a prop.  Concretely, we may describe this prop as follows:

\begin{definition}
Let \define{$\Fin\Rel_k$} be the prop where a morphism $f \maps m \to n$ is a linear
relation from $k^m$ to $k^n$, composition is the usual composition of relations, and the symmetric monoidal structure is given by direct sum.
\end{definition}

To give a presentation of $\Fin\Rel_k$, we use a simple but nice fact: the object 
$1 \in \Fin\Rel_k$, or in less pedantic terms the 1-dimensional vector space $k$, is an extraspecial Frobenius monoid in two fundamentally different ways.   To understand these, first note that for 
any linear relation $L \maps U \to V$ there is a linear relation 
$L^\dagger \maps V \to U$ given by
\[     L^\dagger = \{(v,u) : \; (u,v) \in L  \} .\]
This makes $\Fin\Rel_k$ into a dagger-compact category \cite{AC,BE,Se}.  Also recall that a 
linear map is a special case of a linear relation.

The first way of making $k$ into an extraspecial commutative Frobenius monoid in $\Fin\Rel_k$ uses these morphisms:
\begin{itemize}
\item 
as comultiplication, the linear map called \define{duplication}:
\[  \Delta \maps k \asrelto k^2 \]
\[  \Delta = \{(x,x,x): \; x \in k\} \subseteq k \oplus k^2 \]
\item 
as counit, the linear map called \define{deletion}:
\[  ! \maps k \asrelto \{0\} \]
\[  ! = \{ (x,0) : \; x \in k\} \subseteq k \oplus \{0\}  \]
\item
as multiplication, the linear relation called \define{coduplication}:
\[    \Delta^\dagger \maps  k^2 \asrelto k \]
\[    \Delta^\dagger =\{(x,x,x) : \;  x \in k \} \subseteq k^2 \oplus k \]
\item 
as unit, the linear relation called \define{codeletion}:
\[   !^\dagger \maps \{0\} \asrelto k \]
\[   !^\dagger = \{(0,x): \; x \in k \} \subseteq \{0\} \oplus k .\]
\end{itemize}
We call this the \define{duplicative Frobenius structure} on $k$.   In circuit theory this is
important for working with the electric potential.  The reason is that in a circuit of ideal conductive wires the potential is constant on each connected component, so wires like this:
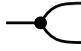
\[
  \xymatrixrowsep{10pt}
  \xymatrixcolsep{10pt}
  \xymatrix@1{
    \comult{.075\textwidth} 
    }
\]
have the effect of duplicating the potential.   

The second way of making $k$ into an extraspecial commutative Frobenius monoid in
$\Fin\Rel_k$ uses these morphisms:
\begin{itemize}
\item as multiplication, the linear map called \define{addition}:
\[   +\maps  k^2 \asrelto k \]
\[   + = \{(x,y,x+y) : \; x,y \in k \} \subseteq k^2 \oplus k \] 
\item as unit, the linear map called \define{zero}:
\[    0 \maps  \{0\} \asrelto k \]
\[    0 = \{(0,x) : \; x \in k \} \subseteq \{0\} \oplus k \]
\item as comultiplication, the linear relation called \define{coaddition}:
\[   +^\dagger \maps  k \asrelto k \oplus k \]
\[   +^\dagger = \{(x+y,x,y) : \; x,y \in k \} \subseteq k \oplus k^2 \]
\item as counit, the linear relation called \define{cozero}:
\[     0^\dagger \maps k \asrelto \{0\} \]
\[    0^\dagger = \{(x,0)\} \subseteq k \oplus \{0\} .\] 
\end{itemize}
We call this the \define{additive Frobenius structure} on $k$.   In circuit theory this structure is
important for working with electric current.  The reason is that Kirchhoff's current law says that the sum of input currents must equal the sum of the output currents, so wires like this:
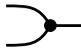
\[
  \xymatrixrowsep{10pt}
  \xymatrixcolsep{10pt}
  \xymatrix@1{
    \mult{.075\textwidth} 
    }
\]
have the effect of adding currents.

The prop $\Fin\Rel_k$ is generated by the eight morphisms listed above together with a
morphism for each element $c \in k$, namely the map from $k$ to itself given by 
multiplication by $c$.   We denote this simply as $c$:
\[    \begin{array}{cccl}
c \maps & k &\to & k \\
        & x & \mapsto &  c x.
\end{array}
\] 

From the generators we can build two other important morphisms:
\begin{itemize}
\item the \define{cup} $\cup \maps k^2 \to \{0\}$: this is the composite of coduplication 
$\Delta^\dagger \maps k^2 \to k$ and deletion $! \maps k \to \{0\}$.  
\item the \define{cap} $\cap \maps \{0\} \to k^2$: this is the composite of codeletion $!^\dagger \maps \{0\} \to k$ and duplication $\Delta \maps k \to k^2$;
\end{itemize}
These are the unit and counit for an adjunction making $k$ into its own dual.  Since every object in $\Fin\Rel_k$ is a tensor product of copies of $k$, every object becomes self-dual.  Thus, $\Fin\Rel_k$ becomes a dagger-compact category.  This explains the use of the dagger notation for half of the
eight morphisms listed above.   

Thanks to the work of Baez and Erbele \cite{BE,E} and also Bonchi, Soboci\'nski and Zanasi \cite{BSZ,BSZ2,Za}, the prop $\Fin\Rel_k$ has a presentation of this form:
\[ \xymatrix{
F(E_k) \ar@<-.5ex>[r]_-{\rho_k} \ar@<.5ex>[r]^-{\lambda_k} & F(\Sigma) + F(\Sigma) + F(k) \ar[r]^-{\square} & \Fin\Rel_k
 }
 \]
Here the signature $\Sigma$ has elements $\mu \maps 2\to 1, \iota \maps 0 \to 1, \delta \maps 1 \to 2, \epsilon \maps 1\to 0$.   In this presentation for $\Fin\Rel_k$, 
the first copy of $F(\Sigma)$ is responsible for the duplicative Frobenius structure on 
$k$, so we call its generators
\begin{itemize}
\item \define{coduplication}, $\Delta^\dagger \maps 2 \to 1$,
\item \define{codeletion}, $!^\dagger \maps 0 \to 1$,
\item \define{duplication}, $\Delta \maps 1 \to 2$,
\item \define{deletion}, $! \maps 1 \to 0$.
\end{itemize}
The second copy of $F(\Sigma)$ is responsible for the additive Frobenius structure,
so we call its generators
\begin{itemize}
\item \define{addition}, $+ \maps 2 \to 1$,
\item \define{zero}, $0 \maps 0 \to 1$,
\item \define{coaddition}, $+^\dagger \maps 1 \to 2$,
\item \define{cozero}, $0^\dagger \maps 2 \to 1$.
\end{itemize}
Finally, we have a copy of $F(k)$, consisting of elements we call
\begin{itemize}
\item \define{scalar multiplication}, $c \maps 1 \to 1$,
\end{itemize}
 one for each $c \in k$.  All these generators are mapped by $\square$ to the previously described morphisms with the same names in $\Fin\Rel_k$.   

We do not need a complete list of the equations in this presentation of $\Fin\Rel_k$, but among them are equations saying that in $\Fin\Rel_k$
\begin{enumerate}
\item $(k, \Delta^\dagger, !, \Delta, !^\dagger)$ is an extraspecial commutative Frobenius monoid;
\item $(k, +, 0, +^\dagger, 0^\dagger)$ is an extraspecial commutative Frobenius monoid;
\item $(k, +, 0, \Delta, !)$ is a bicommutative bimonoid;
\item $(k, \Delta^\dagger, !^\dagger, +^\dagger, 0^\dagger)$ is a bicommutative bimonoid.
\end{enumerate}
In Example \ref{ex:fincorel} we saw that $\Fin\Corel$ is the prop for extraspecial commutative Frobenius monoids.  Thus, items 1 and 2 give two prop morphisms from
$\Fin\Corel$ to $\Fin\Rel_k$.  We use these to define the black-boxing functor in 
Section \ref{sec:black-boxing_conductive}.   Similarly, in Example \ref{ex:finspan} we saw that $\Fin\Span$ is the prop for bicommutative bimonoids.  Thus, items 3 and 4 give two prop morphisms from $\Fin\Span$ to $\Fin\Rel_k$, but these play no role in this paper.

\section{Props of circuits}
\label{sec:prop_of_circuits}

We now introduce the most important props in this paper, give a presentation for them, and describe their algebras.  All this is a rephrasing of the fundamental work of Rosebrugh, Sabadini and Walters \cite{RSW2}.

Fix a set $\L$.   Recall the symmetric monoidal category $\Circ_\L$ described in Proposition \ref{prop:lcirc_as_symmoncat}, where an object is a finite set and a morphism from $X$ to $Y$ is an isomorphism class of $\L$-circuits from $X$ to $Y$.
By Proposition \ref{prop:strictification_1}, $\Circ_\L$ is equivalent to a prop.   Henceforth, by a slight abuse of language, we use \define{$\Circ_\L$} to denote this prop.   

To describe the algebras of $\Circ_\L$, we make a somewhat nonstandard definition. We say a set $\L$ acts on an object $x$ if for each element of $\L$ we have a morphism from $x$ to itself:

\begin{definition} 
An \define{action} of a set $\L$ on an object $x$ in a category $\C$ is a function
$\alpha \maps \L \to \hom(x,x)$.   We also call this an \define{$\L$-action}.
Given $\L$-actions $\alpha \maps \L \to \hom(x,x)$ 
and $\beta \maps \L \to \hom(y,y)$, a \define{morphism of $\L$-actions} is
a morphism $f \maps x \to y$ in $\C$ such that 
$f \alpha(\ell) = \beta(\ell) f$ for all $\ell \in \L$.   
\end{definition}

\begin{proposition}
\label{prop:lcirc_algebras}
An algebra of $\Circ_\L$ in a strict symmetric monoidal category $\C$ is a 
special commutative Frobenius monoid in $\C$ whose underlying object 
is equipped with an action of $\L$.  A morphism of algebras of $\Circ_\L$ in $\C$ is 
a morphism of special commutative Frobenius monoids that is also a morphism
of $\L$-actions.

\end{proposition}

\begin{proof}
This was proved by Rosebrugh, Sabadini and Walters \cite{RSW2}, though stated
in quite different language.  
\end{proof}

\noindent
We may thus say that $\Circ_\L$ is the prop for special commutative Frobenius monoids whose underlying object is equipped with an action of $\L$.  

Unsurprisingly, $\Circ_\L$ is coproduct of two props: the prop for special commutative Frobenius monoids and the prop for $\L$-actions.  To describe the latter, consider a signature with one unary operation for each element of $\L$, and no other operations.  For simplicity we call this signature simply $\L$.   The free prop $F(\L)$ has a morphism $\ell\maps 1\to 1$ for each $\ell \in \L$.  For any strict symmetric monoidal category $\C$, the category of algebras of $F(\L)$ in $\C$ is the category of $\L$-actions and morphisms of $\L$-actions.  We thus call $F(\L)$ \define{the prop for $\L$-actions}.

\begin{proposition} 
\label{prop:lcirc_coproduct}
$\Circ_\L$ is the coproduct of $\Fin\Cospan$ and the prop for $\L$-actions.
\end{proposition}

\begin{proof} 
Let
\[  \xymatrix{
F(E)\ar@<-.5ex>[r]_-{\rho} \ar@<.5ex>[r]^-{\lambda} & F(\Sigma)}\]
be the presentation of $\Fin\Cospan$ given in Example \ref{ex:fincospan_presentation}.  Here $\Sigma$ is the signature with elements 
$\mu \maps 2\to 1, \iota \maps 0 \to 1, \delta \maps 1 \to 2$ and
$\epsilon \maps 1\to 0$, and the equations are the laws for a special commutative Frobenius monoid.

Since left adjoints preserve colimits, we have a natural isomorphism 
$F(\Sigma)+ F(\L) \cong F(\Sigma+\L)$.   Let $\iota \maps F(\Sigma) \to 
F(\Sigma+\L)$ be the resulting monomorphism.
By Corollary \ref{cor:presentation} we can form the coequalizer $X$ here:
\[ \xymatrix{
F(E)\ar@<-.5ex>[r]_-{\iota \rho} \ar@<.5ex>[r]^-{\iota \lambda} & F(\Sigma+\L) \ar[r] & X. } \]
We claim that $X \cong \Circ_\L$.  

On the one hand, there is a morphism $f \maps F(\Sigma) \to \Circ_\L$ sending $\mu, \iota, \delta$ and $\epsilon$ in $F(\Sigma)$ to the corresponding morphisms in $\Circ_\L$, and $f \lambda = f \rho$ because these morphisms make $1 \in \Circ_\L$ into
a special commutative Frobenius monoid.    We thus have a commutative diagram
\[
    \xymatrix{
      F(E)\ar@<-.5ex>[r]_-{\iota \rho} \ar@<.5ex>[r]^-{\iota \lambda} & F(\Sigma+\L) \ar[r]^-{j} \ar[dr]_{f} & X
\\
 & & \Circ_\L . }
\]
By the universal property of the coequalizer, there is a unique morphism $g \maps 
X \to \Circ_\L$ with $g j = f$.  On the other hand, the object $1 \in X$ is, by construction, a special commutative Frobenius monoid with $\L$-action.  By Proposition \ref{prop:lcirc_algebras} we thus obtain an algebra of $\Circ_\L$ in $X$, that is, a morphism $h$ as follows:
\[
    \xymatrix{
      F(E)\ar@<-.5ex>[r]_-{\iota\rho} \ar@<.5ex>[r]^-{\iota\lambda} & F(\Sigma+\L) \ar[r]^-{j} \ar[dr]_{f} & X \ar@<.5ex>[d]^{g}
\\
 & & \Circ_\L . \ar@<.5ex>[u]^{h} }
\]
It is easy to check that $h f = j$ by seeing how both sides act on the elements
$\mu, \iota, \delta, \epsilon,\ell$ of $F(\Sigma+\L)$.  By the universal property 
of $X$ we have $h g = 1$, and because the Frobenius monoid in $\Circ_\L$ has no nontrivial automorphisms we also have $g h = 1$.  Thus $X \cong \Circ_\L$ and 
\begin{equation}
\label{eq:lcirc_1}
\xymatrix{F(E)\ar@<-.5ex>[r]_-{\rho} \ar@<.5ex>[r]^-{\lambda} & F(\Sigma+\L)\ar[r]^{f} & \Circ_\L } 
\end{equation}
is a coequalizer.

In general, given presentations of props $X_1$ and $X_2$, we have coequalizers:
\[
    \xymatrix{
      F(E_i) \ar@<-.5ex>[r]_{\rho_i} \ar@<.5ex>[r]^{\lambda_i} & F(\Sigma_i)
	\ar@<0ex>[r]^-{g_i} & X_i ,
    }  \qquad i = 1, 2
  \]
whose coproduct is another coequalizer:
\[
    \xymatrix{
      F(E_1)+F(E_2) \ar@<-.5ex>[rr]_{\rho_1 + \rho_2} \ar@<.5ex>[rr]^{\lambda_1 + \lambda_2} && F(\Sigma_1)+F(\Sigma_2) 		         	\ar@<0ex>[rr]^-{g_1 + g_2} &&  X_1 + X_2.
    }
  \]
Since left adjoints preserve coproducts, we obtain a coequalizer 
\[
    \xymatrix{
      F(E_1+E_2) \ar@<-.5ex>[r]_{} \ar@<.5ex>[r]^{} & F(\Sigma_1+\Sigma_2) 		         	\ar@<0ex>[r]^-{} &  X_1 + X_2
    }
  \]
which gives a presentation for $X_1 + X_2$. 

In our situation $E \cong E + 0$ where $0$ is the initial or `empty' signature, 
so we have a coequalizer 
\[
    \xymatrix{
      F(E) + F(0)\ar@<-.5ex>[r]_{} \ar@<.5ex>[r]^{} & F(\Sigma) + F(\L) 		         	\ar@<0ex>[r]^-{} & \Fin\Corel + F(\L)
    }
  \]
Combining this with Equation (\ref{eq:lcirc_1}), we see $\Circ_\L \cong \Fin\Cospan + F(\L)$.
\end{proof}

Proposition \ref{prop:lcirc_coproduct} plays a large part in the rest of the paper.   
For example, we can use it to get a more algebraic description of the functor 
$G \maps \Circ \to \Fin\Cospan$ discussed in Section \ref{sec:conductive}.

\begin{example}  
\label{ex:circ_to_fincospan}
We can take $\L$ to be a one-element set, say $\{\ell\}$.  In this case we abbreviate the prop $\Circ_\L$ as $\Circ$,  abusing language slightly since in Definition \ref{defn:Circ} we used the same name for an equivalent symmetric monoidal category.  

By Proposition \ref{prop:lcirc_coproduct} we can identify $\Circ$ with the coproduct $\Fin\Cospan + F(\{\ell\})$.  There thus exists a unique morphism of props
\[              G \maps \Circ \to \Fin\Cospan \]
such that 
\[           G(f) = f  \]
for any morphism $f$ in $\Fin\Cospan$ and
\[           G(\ell) = 1_1. \]
In other words, $G$ does nothing to morphisms in the sub-prop $\Fin\Cospan$, while it sends 
the morphism $\ell$ to the identity.  It thus a retraction for the inclusion of $\Fin\Cospan$ in $\Circ$.

In Section \ref{sec:conductive} we explained how a morphism in $\Circ$ can be seen as a cospan of finite sets
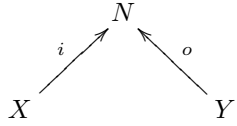
\[
  \xymatrix{
    & N \\
    X \ar[ur]^{i} && Y \ar[ul]_o
  }
\]
together with a graph $\Gamma$ having $N$ as its set of vertices.  In our new description of $\Circ$, each edge of $\Gamma$ corresponds to a copy of the morphism $\ell$.   The functor $G$ has 
the effect of collapsing each edge of $\Gamma$ to a point, since it sends $\ell$ to the identity.  The result is cospan of finite sets where the apex is the set of connected components $\pi_0(\Gamma)$. 
\end{example}

\section{Black-boxing circuits of ideal conductive wires}
\label{sec:black-boxing_conductive}

In Section \ref{sec:conductive} we looked at circuits made of ideal perfectly conductive wires
and described symmetric monoidal functors
\[     \Circ \stackrel{G}{\longrightarrow} \Fin\Cospan \stackrel{H}{\longrightarrow} 
\Fin\Corel.  \]
In Example \ref{ex:fincospan_to_fincorel} we described $H$ as a morphism of props, and in
Example \ref{ex:circ_to_fincospan} we did the same for $G$.   However, so far we have only briefly touched on the `behavior' of such circuits: that is, what they actually \emph{do}.  A circuit provides a relation between potentials and currents at its inputs and outputs.  For a circuit with $m$ inputs and $n$ outputs, this is a linear relation on $2m+2n$ variables.  We now describe a functor called `black-boxing', which takes any circuit of ideal conductive wires and extracts this linear relation.  

In Section \ref{sec:linear_relations} we saw that the object $k \in \Fin\Rel_k$ has two extraspecial commutative Frobenius monoids: the `duplicative' structure and the `additive'  structure. The first is relevant to potentials, while the second is relevant to currents \cite{BF}.  In any symmetric monoidal category, the tensor product of two  monoids is a monoid in standard way, and dually for comonoids.  In the same way, the tensor product of extraspecial commutative Frobenius monoids becomes another extraspecial commutative Frobenius monoid.   Thus, we can make $k \oplus k$ into an extraspecial commutative Frobenius monoid in $\Fin\Rel_k$ where the first copy of $k$ has the duplicative Frobenius structure and the second copy has the additive structure.   Thanks to Proposition \ref{prop:lcirc_algebras}, this determines a strict symmetric monoidal functor
\[    K \maps \Fin\Corel \to \Fin\Rel_k. \]
Composing this with 
\[     HG \maps \Circ \to  \Fin\Corel  \]
we get the \define{black-boxing} functor 
\[   \blacksquare = KHG \maps \Circ \to \Fin\Rel_k  .\]
Here is what black-boxing does to the generators of $\Circ$:

\[
  \xymatrixrowsep{10pt}
  \xymatrixcolsep{10pt}
  \xymatrix@1{
    \mult{.075\textwidth} \ar@{|->}@<.25ex>[r]^-{\blacksquare} \;\; & \;\; \{( \phi_1, I_1,\phi_2, I_2,\phi_3, I_3) :\; \phi_1= \phi_2 = \phi_3,  I_1+ I_2=I_3 \}  \\
    \comult{.075\textwidth} \ar@{|->}@<.25ex>[r]^-{\blacksquare} \;\; & \;\; \{( \phi_1, I_1,\phi_2, I_2,\phi_3, I_3)  :\;  \phi_1= \phi_2 = \phi_3, I_1= I_2 + I_3\}  \\
 \unit{.075\textwidth} \ar@{|->}@<.25ex>[r]^-{\blacksquare} \;\; & \;\; \{(\phi_2, I_2) : \; I_2 = 0\}  \hspace{34ex} \\
\counit{.075\textwidth} \ar@{|->}@<.25ex>[r]^-{\blacksquare} \;\; & \;\; \{( \phi_1, I_1) :\; I_1= 0\} 
\hspace{34ex} \\
\singlegen{.075\textwidth} \ar@{|->}@<.25ex>[r]^-{\blacksquare} \;\; & \;\; \{( \phi_1,I_1,\phi_2,I_2) :\; \phi_1=\phi_2,  I_1= I_2 \} 
 \hspace{18ex} \\
 \hspace{20ex}
  }
\]
\vskip -1em \noindent
Here $\ell$ is the generator corresponding to an ideal conductive wire; black-boxing maps it to the identity morphism on $k^2$.  Since black-boxing is a symmetric monoidal functor, we can decompose a large circuit made of ideal conductive wires into simple building blocks in order to determine the relation it imposes between the potentials and currents at its inputs and outputs.

The black-boxing functor as described so far is not a morphism of props, since it sends the object
$1 \in \Circ$ to the object $2 \in \Fin\Rel_k$, that is, the vector space $k^2$.   However, 
it can reinterpreted as a morphism of props with the help of some symplectic geometry.  Instead of linear relations between finite-dimensional vector spaces, we should use Lagrangian relations between symplectic vector spaces.  For a detailed explanation of this idea, see our work with Fong \cite{BF}; here we simply state the key facts.

\begin{definition}
A  \define{symplectic} vector space $V$ over a field $k$ is a finite-dimensional vector space equipped with a map $\omega \maps V \times  V \to k$ that is:
\begin{itemize}
\item bilinear,
\item alternating: $\omega(v,v) = 0$ for all $v \in V$,
\item nondegenerate: if $\omega(u,v) = 0 $ for all $u \in V$ then $v = 0$.
\end{itemize}
Such a map $\omega$ is called a \define{symplectic structure}.
\end{definition}
\noindent
There is a standard way to make $k\oplus k$ into a symplectic vector space, namely
\[   \omega((\phi, I), (\phi',I')) =  \phi I' - \phi' I . \]
Given two symplectic vector spaces $(V_1,\omega_1)$ and $(V_2, \omega_2)$,
we give their direct sum the symplectic structure 
\[   (\omega_1 \oplus \omega_2)((u_1,u_2), (v_1, v_2)) = \omega_1(u_1, v_1) + 
\omega_2(u_2, v_2)  .\]
In what follows, whenever we treat as $(k \oplus k)^n$ as a symplectic vector space, we
give it the symplectic structure obtained by taking a direct sum of copies of $k \oplus k$ with 
the symplectic structure described above.   Every symplectic vector space is isomorphic to 
$(k \oplus k)^n$ for some $n$, so every symplectic vector space is even dimensional \cite[Thm.\ 21.2]{GS}.  

The concept of a `Lagrangian relation' looks subtle at first, but it has become clear 
in mathematical physics that for many purposes this is the right notion of
morphism between symplectic vector spaces \cite{We1,We2}.    Lagrangian relations
are also known as `canonical relations'.  The definition has a few prerequisites: 

\begin{definition}
Given a symplectic structure $\omega$ on a vector space $V$, we define its \define{conjugate} to be the symplectic structure $\overline\omega = -
\omega$, and write the conjugate symplectic vector space $(V,\overline\omega)$ as
$\overline V$.
\end{definition}

\begin{definition}
A subspace $L$ of a symplectic vector space $(V,\omega)$ is \define{isotropic} if
$\omega (v,w) = 0$ for all $v,w\in V$.   It is \define{Lagrangian} if is isotropic
and not properly contained in any other isotropic subspace.
\end{definition}

\noindent
One can show that a subspace $L \subseteq V$ is Lagrangian if and only it is
isotropic and $\dim(L) = \frac{1}{2}\dim(V)$.  This condition is often easier
to check.  

\begin{definition}
Given symplectic vector spaces $(V,\omega)$ and $(V',\omega')$, a
\define{linear Lagrangian relation} $L \maps V \asrelto V'$ is a Lagrangian
subspace $L \subseteq \overline{V} \oplus V'$.
\end{definition}

We need the conjugate symplectic structure on $V$ to show that the identity
relation is a linear Lagrangian relation.   With this twist, linear Lagrangian relations are 
also closed under composition: for a self-contained proof of this well-known fact, see \cite[Prop.\ 6.8]{BF}.  There is thus a category with symplectic vector spaces as objects and linear Lagrangian relations as morphisms.  This becomes symmetric monoidal using direct sums: in particular, if the linear relations $L \maps U \asrelto V$ and $L' \maps U' \asrelto V'$ are Lagrangian, so is $L \oplus L' \maps U \oplus U' \asrelto V \oplus V'$.  One can show using Proposition \ref{prop:strictification_1} that this symmetric monoidal category is equivalent to the following prop.

\begin{definition}
Let \define{$\Lag\Rel_k$} be the prop where a morphism from $m$ to $n$ is a  linear Lagrangian relation from $(k \oplus k)^m$ to $(k \oplus k)^n$, composition is the usual composition of relations, and the symmetric monoidal structure is given by direct sum.
\end{definition}

We can now redefine
 the functor $K \maps \Fin\Corel \to \Fin\Rel_k$, and the black-boxing
functor $\blacksquare \maps \Circ \to \Fin\Rel_k$, to be morphisms of props taking values in 
$\Lag\Rel_k$.  This is the view we take henceforth.

\begin{proposition}
\label{prop:K_1}
The strict symmetric monoidal functor $K \maps \Fin\Corel \to \Fin\Rel_k$
maps any morphism $f \maps m \to n$ is any morphism in $\Fin\Cospan$ to 
a Lagrangian linear relation $K(f) \maps (k \oplus k)^m \asrelto (k \oplus k)^n$.  
It thus defines a morphism of props, which we call
\[   K \maps \Circ \to \Lag\Rel_k . \]
\end{proposition}

\begin{proof}
First note that $\Lag\Rel_k$ is a symmetric monoidal 
subcategory of $\Fin\Rel_k$: composition and direct sum for linear Lagrangian relations is a special case of composition and direct sum for linear relations.  Second, note that while $K$ applied to the object $n \in \Fin\Cospan$ gives the vector space $(k \oplus k)^n$, which is the object $2n$ in $\Fin\Rel_k$, this is the object $n$ in $\Lag\Rel_k$.  Thus, to check that strict symmetric monoidal functor $K \maps \Fin\Cospan \to \Fin\Rel_k$ defines a morphism of props from $\Fin\Cospan$ to $\Lag\Rel_k$, we just need to check that $K(f)$ is Lagrangian for each
generator $f$ of $\Fin\Cospan$.  We have
\[
\begin{array}{lcl}
  K(m) &=& \{(\phi_1,I_1,\phi_2,I_2,\phi_3,I_3) :\;  \phi_1= \phi_2 = \phi_3, \; I_1+ I_2=I_3\}  \\
  K(\Delta) &=& \{(\phi_1,I_1,\phi_2,I_2,\phi_3,I_3) :\; \phi_1= \phi_2 = \phi_3, \;  I_1= I_2 + I_3\}  \\
  K(\iota) &=& \{(\phi_2, I_2) : \; I_2 = 0\}   \\
  K(\epsilon) &=& \{(\phi_1, I_1) :\; I_1= 0\}
\end{array}
\]
In each case the relation is an isotropic subspace of half
the total dimension, so it is Lagrangian.
\end{proof}

We can characterize this new improved $K$ as follows:

\begin{proposition}
\label{prop:K_2}
There exists a unique morphism of props
\[    K \maps \Fin\Corel \to \Lag\Rel_k \]
sending the extraspecial commutative Frobenius monoid $1 \in \Fin\Corel$ to the extraspecial commutative Frobenius monoid $k \oplus k \in \Lag\Rel_k$, where the first copy of $k$ is equipped with its additive Frobenius structure and the second is equipped with its duplicative Frobenius structure.
\end{proposition}

\begin{proof} 
Existence follows from Proposition \ref{prop:K_1}; uniqueness
follows from the fact that $\Fin\Corel$ is the prop for extraspecial commutative Frobenius
monoids.
\end{proof}

We can now reinterpret black-boxing of circuits of ideal conductive wires as a morphism of props:

\begin{definition}
We define \define{black-boxing} to be the morphism of props $\blacksquare \maps \Circ \to \Lag\Rel_k$ given by the composite
\[     \Circ \stackrel{G}{\longrightarrow} \Fin\Cospan \stackrel{H}{\longrightarrow} 
\Fin\Corel \stackrel{K}{\longrightarrow} \Lag\Rel_k. \]
\end{definition}

\section{Black-boxing linear circuits}
\label{sec:black-boxing}

Black-boxing circuits of ideal conductive wires is just the first step: 
one can extend black-boxing to circuits made of wires labeled by elements of any set $\L$.  The elements of $\L$ play the role of `circuit elements' such as resistors, inductors and capacitors.  The extended black-boxing functor can be chosen so that these circuit elements are mapped to arbitrary Lagrangian linear relations from $k\oplus k$ to itself.   

The key to doing this is Proposition \ref{prop:lcirc_coproduct}, which says that
\[   \Circ_\L \cong \Fin\Cospan + F(\L) . \]
We also need Propositions \ref{prop:fincospan_to_fincorel} and 
\ref{prop:K_2}, which give prop morphisms $H \maps \Fin\Cospan \to \Fin\Corel$
and $K \maps \Fin\Corel \to \Lag\Rel_k$, respectively.

\begin{theorem} 
\label{thm:black-boxing_1}
For any field $k$ and label set $\L$, there exists a unique morphism of props 
\[ \blacksquare \maps \Circ_\L \cong \Fin\Cospan + F(\L) \to \Lag\Rel_k \] 
such that $\blacksquare \vert_{\Fin\Cospan}$ is the composite
\[    
       \Fin\Cospan \stackrel{H}{\longrightarrow} 
       \Fin\Corel \stackrel{K}{\longrightarrow} \Lag\Rel_k  \]
and $\blacksquare \vert_{F(\L)}$ maps each $\ell \in \L$ to an arbitrarily chosen
Lagrangian linear relation from $k \oplus k$ to itself.    
\end{theorem}

\begin{proof}
By the universal property of the coproduct and the fact that $F(\L)$ is the prop for 
$\L$-actions, there exists a unique morphism of props
\[     \blacksquare \maps \Fin\Cospan + F(\L) \to \Lag\Rel_k \]
such that $\blacksquare \vert_{\Fin\Cospan} = K \circ H $ and $\phi \vert_{F(\L)}$ maps 
each $\ell \in \L$ to an arbitrarily chosen Lagrangian linear relation from $k \oplus k$ to itself. 
\end{proof}

We can apply this theorem to circuits made of resistors.   Any resistor has a \define{resistance} $R$, which is a positive real number.  Thus, if we take the label set $\L$ to be $\R^+$, we obtain a prop $\Circ_\L$ that models circuits made of resistors.  Electrical engineers typically draw a resistor as a wiggly line:
\[
\begin{tikzpicture}[circuit ee IEC, set resistor graphic=var resistor IEC graphic]
\node[contact] (C) at (0,2) {};
\node[contact] (D) at (2,2) {};
\node (A) at (-.75,2) {$(\phi_1,I_1)$};
\node (B) at (2.75,2) {$(\phi_2,I_2)$};
  \draw (0,2) to [resistor={info={$R$}}] ++(2,0);
\end{tikzpicture}
\]
Here $(\phi_1,I_1) \in \R \oplus \R$ are the \define{potential} and \define{current} at the resistor's input and $(\phi_2,I_2) \in \R \oplus \R$ are the potential and current at its output.  To define a black-boxing functor 
\[  \blacksquare \maps \Circ_\L \to \Fin\Rel_k   \]
we need to choose a linear relation between these four quantities for each choice of the
resistance $R$.   To do this, recall first that Kirchhoff's current law requires that the current flowing in equals the current flowing out: $I_1 = I_2$.  Second, Ohm's Law says that $V=RI$ where $I = I_1 = I_2$ is called the \define{current through} the resistor and $V = \phi_2 - \phi_1$ is called the \define{voltage across} the resistor.   Thus, for each $R \in \R^+$ we choose
\[ \blacksquare (R) = \{(\phi_1,I_1,\phi_2,I_2) : \; \phi_2-\phi_1 = R I_1, I_1 = I_2 \} .\]

We could stop here, but suppose we also want to include inductors and capacitors. 
An inductor comes with an inductance $L \in \R^+$ (not to be confused with our notation for a label set), while a capacitor comes with a capacitance $C \in \R^+$. These circuit elements are drawn as follows:
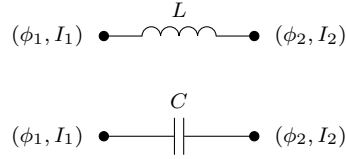
\[
\begin{tikzpicture}[circuit ee IEC, set resistor graphic=var resistor IEC graphic]
\node[contact] (C) at (0,2) {};
\node[contact] (D) at (2,2) {};
\node (A) at (-.75,2) {$(\phi_1,I_1)$};
\node (B) at (2.75,2) {$(\phi_2,I_2)$};
  \draw (0,2) to [inductor={info={$L$}}] ++(2,0);
\end{tikzpicture}
\]
\[
\begin{tikzpicture}[circuit ee IEC, set resistor graphic=var resistor IEC graphic]
\node[contact] (C) at (0,2) {};
\node[contact] (D) at (2,2) {};
\node (A) at (-.75,2) {$(\phi_1,I_1)$};
\node (B) at (2.75,2) {$(\phi_2,I_2)$};
  \draw (0,2) to [capacitor={info={$C$}}] ++(2,0);
\end{tikzpicture}
\]
These circuit elements apply to time-dependent currents and voltages, and they 
impose the relations $V = L \dot{I}$ and $I = C \dot{V}$, where the dot stands
for the time derivative.  Engineers deal with this using the Laplace transform.  
As explained in detail elsewhere \cite{BE,BF}, this comes down to adjoining a variable $s$ to the field $\mathbb{R}$ and letting $k = \R(s)$ be the field of rational functions in one real variable.   The variable $s$ has the meaning of a time derivative.   We henceforth use $I$ and $V$ to denote the Laplace transforms of current and voltage, respectively, and obtain the relations $V = sLI$ for the inductor, $I = sCV$ for the capacitor, and $V = RI$ for the resistor.   Thus if we extend our label set to the disjoint union of three copies of $\R^+$, defining
\[  RLC = \R^+ + \R^+ + \R^+ ,\]
we obtain a prop $\RLCCirc$ that describes circuits of resistors, inductors and capacitors.  The name `$\RLCCirc$' is a bit of a pun, since electrical engineers 
call a circuit made of one resistor, one inductor and one capacitor an `$RLC$ circuit'.

To construct the black-boxing functor
\[   \blacksquare \maps \RLCCirc \to \Fin\Rel_k \]
we specify it separately on each kind of circuit element.  Thus, on the first copy 
of $\R^+$, corresponding to resistors, we set
\[ \blacksquare (R) = \{(\phi_1,I_1,\phi_2,I_2) : \;\phi_2-\phi_1 = R I_1, \;  I_1 = I_2\}\]
as before.  On the second copy we set
\[ \blacksquare (L) = \{(\phi_1,I_1,\phi_2,I_2) : \; \phi_2-\phi_1 = s L I_1, \;  I_1 = I_2\}\]
and on the third we set
\[ \blacksquare (C) = \{(\phi_1,I_1,\phi_2,I_2) : \; sC(\phi_2-\phi_1) = I_1, \; \; I_1 = I_2\}  .\]
We have:

\begin{proposition}
\label{prop:black-boxing_1}
If $f \maps m \to n$ is any morphism in $\RLCCirc$, the linear relation $\blacksquare(f): (k \oplus k)^m\asrelto (k \oplus k)^n$ is Lagrangian.  We thus
obtain a morphism of props
\[   \blacksquare \maps \RLCCirc \to \Lag\Rel_k  \]
where $k = \R(s)$.
\end{proposition}

\begin{proof}
By Theorem \ref{thm:black-boxing_1} it suffices to check that 
the linear relations $\blacksquare(R), \blacksquare(L)$ and $\blacksquare(C)$ are
Lagrangian for any $R,L,C \in \R^+$.  To do this, one can check that these relations are 2-dimensional isotropic subspaces of $\overline{(k \oplus k)} \oplus (k \oplus k)$.
\end{proof}

A similar result was proved by Baez and Fong \cite{BF,Fo2} using different methods: decorated cospan categories rather than props.  In their work, resistors, inductors and capacitors were subsumed in a mathematically more natural class of circuit elements.   We can do something similar here.   At the same time, we might as well generalize to an arbitrary field $k$ and work with the prop $\Circ_k$, meaning $\Circ_\L$ where the label set $\L$ is taken to be $k$.   

\begin{definition} We call a morphism in $\Circ_k$ a \define{linear circuit}.
\end{definition}

Engineers might instead call such a morphism a `passive' linear circuit \cite{BF}, but
we will never need any other kind.

\begin{theorem}
\label{thm:black-boxing_2}
For any field $k$ there exists a unique morphism of props
\[   \blacksquare \maps \Circ_k \cong \Fin\Cospan + F(k) \to \Lag\Rel_k \]
such that $\blacksquare \vert_{\Fin\Cospan}$ is the composite
\[    \Fin\Cospan \stackrel{H}{\longrightarrow} 
       \Fin\Corel \stackrel{K}{\longrightarrow} \Lag\Rel_k  \]
and for each \(Z \in k\), the linear Lagrangian relation
\[   \blacksquare(Z) \maps k \oplus k \asrelto k \oplus k \]
is given by
\[  \blacksquare (Z) =\{(\phi_1,I_1,\phi_2,I_2) : \; \phi_2-\phi_1 = Z I_1 ,\;  I_1 = I_2\}. \]
\end{theorem}

\begin{proof} 
By Theorem \ref{thm:black-boxing_1} it suffices to check that $\blacksquare(Z)$ is a linear Lagrangian relation for any \(Z \in k\).    This follows from the fact that $\blacksquare(Z)$ is a 2-dimensional isotropic subspace of $\overline{(k \oplus k)} \oplus (k \oplus k)$.
\end{proof}

In electrical engineering, $Z$ is called the `impedance': a circuit
element with one input and one output has \define{impedance} $Z$ if the voltage across it is $Z$ times the current through it.   Resistance is a special case of impedance.  In particular, an ideal conductive wire has impedance zero.  Mathematically, this fact is reflected in an inclusion of props
\[        \Circ \hookrightarrow \Circ_k  \]
that sends the generator $\ell$ in $\Circ \cong \Fin\Cospan + F(\{\ell\})$ to 
the generator $0 \in k$ in  $\Circ \cong \Fin\Cospan + F(k)$, while it is the identity
on $\Fin\Cospan$.    Black-boxing for linear circuits then extends black-boxing
as previously defined for circuits of ideal conductive wires.   That is, we have a commutative
triangle:
\begin{equation}
\label{eq:extension_of_black-boxing}
    \xymatrix{
        \Circ \phantom{ |} \ar[dr]^-{\blacksquare} \ar@{^{(}->}[d] &  \\
        \Circ_k  \phantom{ |}  \ar[r]_-{\blacksquare} & \Lag\Rel_k. }
\end{equation}

\section{Signal-flow diagrams and circuit diagrams}
\label{sec:signal}

Control theory is the branch of engineering that studies the behavior of `open' dynamical systems: that is, systems with inputs and outputs.  Control theorists have intensively studied \emph{linear} open dynamical systems, and they specify these using `signal-flow diagrams'.  We now know that signal-flow diagrams are a syntax for linear relations.  In other words, we can see signal-flow diagrams as morphisms in a free prop that maps onto the prop of linear relations, $\Fin\Rel_k$.   This is a nice example of functorial semantics drawn from engineering.

The machinery of props lets us map circuit diagrams to signal-flow diagrams in a manner compatible with composition, addressing a problem raised by Willems \cite{Willems}.  To do this, because circuits as we have defined them obey some nontrivial equations, while the prop of signal-flow diagrams is free, we need to introduce free props that map onto $\Circ$ and $\Circ_k$.  

In Section \ref{sec:linear_relations} we discussed this presentation of $\Fin\Rel_k$: 
\[ \xymatrix{
F(E_k) \ar@<-.5ex>[r]_-{\rho_k} \ar@<.5ex>[r]^-{\lambda_k} & F(\Sigma) + F(\Sigma) + F(k) \ar[r]^-{\square} & \Fin\Rel_k. }
 \]
The morphisms in the free prop $F(\Sigma) + F(\Sigma) + F(k)$ can be drawn as string diagrams, and these roughly match what control theorists call signal flow diagrams.  
So, we make the following definition:

\begin{definition} Define \define{$\SigFlow_k$} to be $F(\Sigma) + F(\Sigma) + F(k)$. We call a morphism in $\SigFlow_k$ a \define{signal-flow diagram}.
\end{definition}

The prop $\SigFlow_k$ is free on eight generators together with one generator for each element of  $k$.  The meaning of these generators is best understood in terms of the linear relations they are mapped to under $\square$.  We discussed those linear relations in Section \ref{sec:black-boxing_conductive}.   So, we give the generators of $\SigFlow_k$ the same names.  Erbele also drew pictures of them, loosely modeled after the notation in signal-flow diagrams \cite{BE,E}.  The generators of the first copy of $F(\Sigma)$ are then:
\begin{itemize}
\item \define{coduplication}, $\Delta^\dagger \maps 2 \to 1$
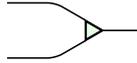
\[
  \xymatrix@1{
    \sigflowpiccodup{.15\textwidth} 
  }
\]
 \item \define{codeletion}, $!^\dagger \maps 0 \to 1$
\[
  \xymatrix@1{
    \;\;\quad\quad\sigflowpiccodel{.1\textwidth} 
  }
\]
\item \define{duplication}, $\Delta \maps 1 \to 2$
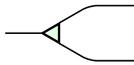
\[
  \xymatrix@1{
    \sigflowpicdup{.15\textwidth}\;\;\quad\quad 
  }
\]

\item \define{deletion}, $! \maps 1 \to 1$
\[
  \xymatrix@1{
    \sigflowpicdel{.1\textwidth} 
  }
\]
\end{itemize}
The generators of the second copy of $F(\Sigma)$ are:
\begin{itemize}
\item 
\define{addition}, $+ \maps 2 \to 1$
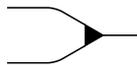
\[
  \xymatrix@1{
    \sigflowpicadd{.15\textwidth} 
  }
\]
\item \define{coaddition}, $+^\dagger \maps 1 \to 2$
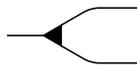
\[
  \xymatrix@1{
    \sigflowpiccoadd{.15\textwidth} 
  }
\]
\item \define{zero}, $0 \maps 0 \to 1$ 
\[
  \xymatrix@1{
    \;\;\quad\quad \sigflowpiczero{.1\textwidth} 
  }
\]
\item \define{cozero}, $0^\dagger \maps 1 \to 0$
\[
  \xymatrix@1{
    \sigflowpiccozero{.1\textwidth} \quad\quad\;\;
  }
\]
\end{itemize}
The generators of $F(k)$ are:
\begin{itemize}
\item for each $c \in k$, \define{scalar multiplication}, $c \maps 1 \to 1$
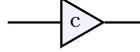
\[
  \xymatrixrowsep{5pt}
  \xymatrix@1{
    \SigLabelpic{.15\textwidth}\\
  }
\]
\end{itemize}
\vskip -2.5em
Since $\SigFlow_k$ is a free prop, while $\Circ_k$ is not, there is no useful morphism of props from $\Circ_k$ to $\SigFlow_k$.   However, there is a free prop $\Ccirc_k$ having $\Circ_k$ as a quotient, and a morphism from this free prop to $\SigFlow_k$.  This morphism lifts the black-boxing functor described in Theorem \ref{thm:black-boxing_2} to a morphism between free props.   

\begin{definition}
For any set $\L$ define the prop \define{$\Ccirc_\L$} by
\[     \Ccirc_\L = F(\Sigma) + F(\L).  \]
Here $\L$ stands for the signature with one unary operation for each element of $\L$ while $\Sigma$ is the signature with elements $\mu \maps 2\to 1, \iota \maps 0 \to 1, \delta \maps 1 \to 2$ and $\epsilon \maps 1\to 0$.
\end{definition}

In Example \ref{ex:fincospan_presentation} we saw this presentation for $\Fin\Cospan$:
\[ \xymatrix{
F(E)\ar@<-.5ex>[r]_-{\rho} \ar@<.5ex>[r]^-{\lambda} & F(\Sigma) \ar[r] & \Fin\Cospan } \]
where the equations in $E$ are the laws for a special commutative Frobenius monoid.   In the proof of Proposition \ref{prop:lcirc_coproduct} we derived this presentation for $\Circ_\L$:
\[ \xymatrix{
F(E)\ar@<-.5ex>[r]_-{\iota \rho} \ar@<.5ex>[r]^-{\iota \lambda} & F(\Sigma) + F(\L) \ar[r] & \Fin\Cospan + F(\L) \cong \Circ_\L} \]
where $\iota$ is the inclusion of $F(\Sigma)$ in $F(\Sigma) + F(\L)$.  Since 
$F(\Sigma) + F(k) = \Ccirc_k$, we can rewrite this as 
\[ \xymatrix{
F(E)\ar@<-.5ex>[r]_-{\iota \rho} \ar@<.5ex>[r]^-{\iota \lambda} & \Ccirc_\L \ar[r] & \Circ_\L } \]
The last arrow here, which we call $P \maps \Ccirc_\L \to \Circ_\L$, imposes the laws of a special commutative Frobenius monoid on the object $1$.  

The most important case of this construction is when $\L$ is some field $k$:

\begin{theorem}
\label{thm:black-and-white-boxing}
For any field $k$, there is a strict symmetric monoidal functor $T \maps \Ccirc_k \to \SigFlow_k$
giving a commutative square of strict symmetric monoidal functors
\[
    \xymatrix{
      \Ccirc_k \phantom{ |} \ar[r]^-{P} \ar[d]_{T} & \Circ_k \phantom{ |} \ar[rr]^-{\blacksquare} & & \Lag\Rel_k \phantom{ |} \ar@{^{(}->}[d] \\
      \SigFlow_k \phantom{ |}\ar[rrr]^-{\square} & & & \Fin\Rel_k. \phantom{ |}
      }
\]
\end{theorem}

\noindent
The horizontal arrows in this diagram are morphisms of props.  The vertical ones are not, because they send the object $1$ to the object $2$.

\begin{proof}
We define the strict symmetric monoidal functor $T \maps \Ccirc_k \to \SigFlow_k$
as follows.  It sends the object $1$ to $2$ and has the following action on the generating morphisms of $\Ccirc_k = F(\Sigma) + F(k)$:
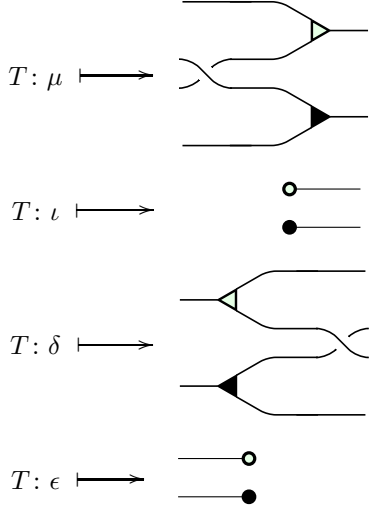
\[
  \xymatrixrowsep{5pt}
  \xymatrix@1{
  T \maps \mu \;\; \ar@{|->}@<.25ex>[r] & \SigMultpic{.20\textwidth}  \\
 T \maps \iota \;\; \ar@{|->}@<.25ex>[r] & \quad\quad\quad\quad \;\SigUnitpic{.10\textwidth} \\
 T \maps \delta \;\; \ar@{|->}@<.25ex>[r] & \SigCoMultpic{.20\textwidth}  \\
 T \maps \epsilon \;\; \ar@{|->}@<.25ex>[r] & \SigCoUnitpic{.10\textwidth} \quad\quad \quad\quad\quad \\
  }
\]
and for each element $Z \in k$, 
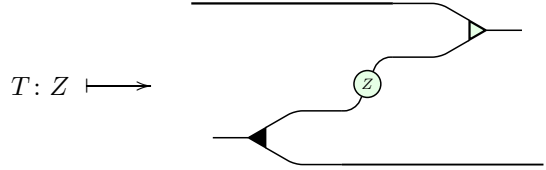
\[
  \xymatrixrowsep{5pt}
  \xymatrix@1{
   T \maps Z \;\; \ar@{|->}@<.25ex>[r] & \quad \SigLabelEdgepic{.32\textwidth}\\
  }
\]
where we use string diagram notation for morphisms in $\SigFlow_k$.    To check that $T$ with these properties exists and is unique, let $\SigFlow_k^{\mathrm{ev}}$ be the full subcategory of $\SigFlow_k$ whose objects are even natural numbers.  This becomes a prop if we rename each object $2n$, calling it $n$.  Then, since $\Ccirc_k$ is free, there exists a unique morphism of props $T \maps \Ccirc _k\to \SigFlow_k^{\mathrm{ev}}$ defined on generators as above.   Since $\SigFlow_k^{\mathrm{ev}}$ is a symmetric monoidal subcategory of $\SigFlow_k$, we can reinterpret $T$ as a strict symmetric monoidal functor $T \maps \Ccirc_k \to \SigFlow_k$, and this too is uniquely determined by its action on the generators.

To prove that the square in the statement of the theorem commutes, it suffices to check it on the generators of $\Ccirc_k$.   For this we use the properties of black-boxing stated in Theorem \ref{thm:black-boxing_2}.   First, note that these morphisms in $\SigFlow_k$:
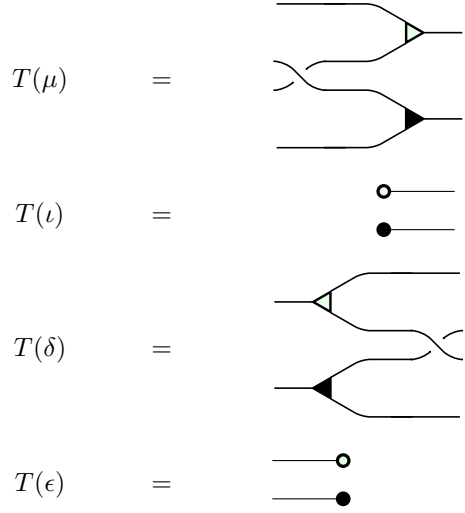
\[
  \xymatrixrowsep{5pt}
  \xymatrix@1{
  T(\mu) \;\; & =& \SigMultpic{.20\textwidth}  \\
 T(\iota) \;\; &=& \quad\quad\quad\quad \;\SigUnitpic{.10\textwidth} \\
 T(\delta) \;\; &=& \SigCoMultpic{.20\textwidth}  \\
 T(\epsilon) \;\; &=& \SigCoUnitpic{.10\textwidth} \quad\quad \quad\quad\quad \\
  }
\]
are mapped by $\square$ to the same multiplication, unit, comultiplication and counit
on $k \oplus k$ as given by $\blacksquare(\mu), \blacksquare(\iota), \blacksquare(\delta)$
and $\blacksquare(\epsilon)$.   Namely, these four linear relations make $k \oplus k$ into Frobenius monoid where the first copy of $k$ has the duplicative Frobenius structure and the second copy has the additive Frobenius structure.  Second, note that the morphism
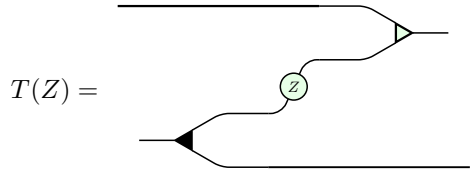
\[   T(Z) = \SigLabelEdgepic{.32\textwidth}  \]
in $\SigFlow_k$ is mapped by $\square$ to the linear relation
\[   \{(\phi_1,I_1,\phi_2,I_2) : \; \phi_2-\phi_1 = Z I_1 ,\;  I_1 = I_2\}, \]
while Theorem \ref{thm:black-boxing_2} states that $\blacksquare(Z)$ is the same
relation, viewed as a Lagrangian linear relation.
 \end{proof}

Recall that when $\L = \{\ell\}$, we call $\Circ_\L$ simply $\Circ$.    We have seen that the map $\{\ell\} \to k$ sending $\ell$ to $0$ induces a morphism of props $\Circ \hookrightarrow \Circ_k$, which expresses how circuits of ideal conductive wires are a special case of linear circuits.  We can define a morphism of props $\Ccirc \hookrightarrow \Ccirc_k$ in an analogous
way.  Due to the naturality of the above construction, we obtain a commutative square
\[
    \xymatrix{
     \Ccirc \phantom{ |} \ar[r]^-{P} \ar@{^{(}->}[d] & \Circ \phantom{ |} \ar@{^{(}->}[d]  \\
      \Ccirc_k \phantom{ |} \ar[r]^-{P} & \Circ_k \phantom{ |} 
      }
\]
We can combine this with the commutative square in Theorem \ref{thm:black-and-white-boxing} and the commutative triangle in Equation \ref{eq:extension_of_black-boxing}, which expands to a square when we use the definition of black-boxing for circuits of ideal conductive wires.  The resulting diagram summarizes the relationship between linear circuits, cospans, corelations, and signal-flow diagrams:
\[
    \xymatrix{
     \Ccirc \phantom{ |} \ar[r]^-{P} \ar@{^{(}->}[d] & \Circ \phantom{ |} \ar@{^{(}->}[d] \ar[r]^-{G} & \Fin\Cospan \phantom{ |} \ar[r]^-{H} &
      \Fin\Corel \phantom{ |} \ar[d]^{K} \\
      \Ccirc_k \phantom{ |} \ar[r]^-{P} \ar[d]_{T} & \Circ_k \phantom{ |} \ar[rr]^-{\blacksquare} & & \Lag\Rel_k \phantom{ |} \ar@{^{(}->}[d] \\
      \SigFlow_k \phantom{ |}\ar[rrr]^-{\square} & & & \Fin\Rel_k. \phantom{ |}
      }
\]

In conclusion, we warn the reader that Erbele \cite[Defn.\ 21]{E} uses a different definition of $\SigFlow_k$.  His prop with this name is free on the following generators:
\begin{itemize}
\item addition, $+ \maps 2 \to 1$
\item zero, $0 \maps 0 \to 1$
\item duplication, $\Delta \maps 1 \to 2$
\item deletion, $! \maps 1 \to 0$
\item for each $c \in k$, scalar multiplication $c \maps 1 \to 1$
\item the cup, $\cup \maps 2 \to 0$
\item the cap, $\cap \maps 0 \to 2$
\end{itemize}
The main advantage is that string diagrams for morphisms in his prop more closely resemble the signal-flow diagrams actually drawn by control theorists; however, see his discussion of some subtleties.  All the results above can easily be adapted to Erbele's definition.  

\section{Voltage and current sources}
\label{sec:affine}

In our previous work on electrical circuits \cite{BF}, batteries were not included.  Resistors, capacitors, and inductors define linear relations between potential and current.   Batteries, also known as `voltage sources', define \emph{affine} relations between these quantities.  The same is true of current sources.    Thus, to handle these additional circuit elements, we need a black-boxing functor that takes values in a different prop.  The ease with which we can do this illustrates the flexibility of working with props. In what follows, we continue to work over an arbitrary field $k$, which in electrical engineering is either $\R$ or $\R(s)$.

A voltage source is typically drawn as follows:
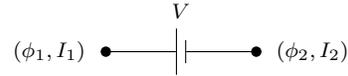
\[
\begin{tikzpicture}[circuit ee IEC, set resistor graphic=var resistor IEC graphic]
\node[contact] (C) at (0,2) {};
\node[contact] (D) at (2,2) {};
\node (A) at (-.75,2) {$(\phi_1,I_1)$};
\node (B) at (2.75,2) {$(\phi_2,I_2)$};
  \draw (0,2) to [battery={info={$V$}}] ++(2,0);
\end{tikzpicture}
\]
It sets the difference between the output and input potentials to a
constant value $V \in k$.    Thus, to define a black-boxing functor 
for voltage sources, we want to set
\[  \blacksquare(V) =  \{(\phi_1,I_1,\phi_2,I_2) : \; \phi_2-\phi_1 = V, \;  I_1 = I_2 \} .\]
Similarly, a current source is drawn as
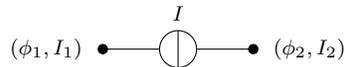
\[
\begin{tikzpicture}[circuit ee IEC, set resistor graphic=var resistor IEC graphic]
\node[contact] (C) at (0,2) {};
\node[contact] (D) at (2,2) {};
\node (A) at (-.75,2) {$(\phi_1,I_1)$};
\node (B) at (2.75,2) {$(\phi_2,I_2)$};
  \draw (0,2) to [current source={info={$I$}}] ++(2,0);
\end{tikzpicture}
\]
and it fixes the current at both input and output to a constant value $I$, giving
this relation:
\[  \blacksquare(I) = \{(\phi_1,I_1,\phi_2,I_2) : \;  I_1 = I_2 = I \} .\]

We could define a black-boxing functor suitable for voltage and current sources
by using a prop where the morphisms $f \maps m \to n$ are arbitrary relations
from $(k \oplus k)^m$ to $(k \oplus k)^n$.  However, the relations shown above 
are better than average.  First, they are `affine relations': that is, translates of linear
subspaces of $(k \oplus k)^m \oplus (k \oplus k)^n$.  For voltage sources we have
\[    \blacksquare(V)  = (0,0,V,0) +  \{(\phi_1,I_1,\phi_1,I_1) \}  \]
and for current sources we have
\[    \blacksquare(I) =  (I,0,I,0) +\{(\phi_1,0,\phi_2,0) \} .\]
Second, these affine relations are `Lagrangian': that is, they are translates of
Lagrangian linear relations.    

Thus, we proceed as follows:

\begin{definition}  
Given symplectic vector spaces $(V,\omega)$ and $(V',\omega')$,
an \define{Lagrangian affine relation} $R \maps V \asrelto V'$ is an affine subspace
$R \subseteq \overline{V} \oplus V'$ that is also a Lagrangian subvariety of 
$\overline{V} \oplus V'$.
\end{definition}

Here recall that a subset $A$ of a vector space is said to be an 
\define{affine subspace} if it is closed under affine linear combinations: 
if $a,a' \in A$ then so is $t a + (1-t)a'$ for all $t \in \R$.   A subvariety
$R \subseteq \overline{V} \oplus V'$ is said to be \define{Lagrangian} if each 
of its tangent spaces, when identified with a linear subspace of
$\overline{V} \oplus V'$, is Lagrangian.   If $R$ is an affine subspace of 
$\overline{V} \oplus V'$ it is automatically a subvariety, and it is either empty or 
a translate $L + (v,v')$ of some linear subspace $L \subseteq \overline{V} \oplus V'$.
In the latter case all its tangent spaces become the same when identified with linear subspaces of $\overline{V} \oplus V'$: they are all simply $L$.   We thus have:

\begin{proposition}
Given symplectic vector spaces $(V,\omega)$ and $(V',\omega')$, any Lagrangian affine relation $R \maps V \asrelto V'$ is either empty or a translate of a Lagrangian linear relation.
\end{proposition}

This allows us to construct the following category:

\begin{proposition}
There is a category where the objects are symplectic vector spaces, the morphisms
are Lagrangian affine relations, and composition is the usual composition of relations.  This is a symmetric monoidal subcategory of the category of sets and relations with the symmetric monoidal structure coming from the cartesian product of sets.  
\end{proposition}

\begin{proof}
It suffices to check that morphisms are closed under composition and tensor product, and that the braiding is a Lagrangian affine relation. Let $R \maps U \asrelto V$ and $S \maps V \asrelto W$ be two Lagrangian affine relations. If the composite $S\circ R$ is empty then we are done, so suppose it is not.   In this case there exist $u\in U, v\in V, w\in W$ such that $(u,v)\in R, (v,w) \in S$, and $(u,w) \in S\circ R $.   Consider the subspaces $-(u,v) + R= L$ and $-(v,w) + S = M$. These are both affine subspaces containing the origin, so they are linear subspaces. Since $R$ and $S$ are Lagrangian affine subspaces, their translates $L$ and $M$ are Lagrangian linear relations from $U$ to $W$.   It follows that $M\circ L\maps V \asrelto W $ is a Lagrangian linear relation.  We claim that $S\circ R = (u,w) + M\circ L$, so that morphisms in our proposed category are closed under composition.  First write $R = L +(u,v)$ and $S = M + (v,w)$, so that
\begin{align*}
S\circ R &=  \{(x,z)  | \exists y \in V \text{ s.t. } (x,y) \in R \text{ and } (y,z) \in S\} \\
& = \{(x,z) | \exists y \in V \text{ s.t. } (x,y) \in L+(u,v) \text{ and } (y,z) \in M+ (v,w)\} \\[-3em]
& = \{(x,z) | \exists y \in V \text{ s.t. } 
	\begin{aligned}
                            \\ 
                             \\
          x= l_1 + u \\ 
          y = l_2 + v \\
          (l_1,l_2) \in L
        \end{aligned}
	\begin{aligned} \\
                              \\
            y = m_1 + v \\
           z = m_2 + w \\\hspace{2ex}
         (m_1,m_2) \in M\} \\
	\end{aligned}
\end{align*}
This gives $l_2 = m_1$, so finally we have
\begin{align*}
S \circ R & = \{(l_1+u, m_2+w) | \exists l_2 \in V \text{ s.t. } (l_1,l_2) \in L \text{ and } (l_2,m_2) \in M\}\\
&  = (u,w) + M\circ L
\end{align*}
as desired.

The tensor product of Lagrangian affine relations $R \maps U \asrelto V$ and $R'  \maps U' \asrelto V'$ is given by
\[  R \oplus R' = \{(u,u',v,v'):  \; (u,v) \in R, \; (u',v') \in R'\} \maps U \oplus U' \asrelto V \oplus V' ,\]
and this is a Lagrangian affine relation because it is a translate of a Lagrangian linear relation.  Finally, note that the braiding morphism $B_{U,V} \maps U\oplus V \asrelto V \oplus U$ defined by $ B_{U,V} = \{(u,v,v,u) |u\in U, v\in V\}$ is a Lagrangian linear relation and thus a Lagrangian affine relation.
\end{proof}

By Proposition \ref{prop:strictification_1}, the above symmetric monoidal category is equivalent to the following prop:

\begin{definition}
Let \define{$\Aff\Lag\Rel_k$} be the prop where a morphism from $m$ to $n$ is a 
Lagrangian affine relation from $(k \oplus k)^m$ to $(k \oplus k)^n$, composition
is the usual composition of relations, and the symmetric monoidal structure is given 
as above.
\end{definition}

We can extend the black-boxing functor from linear circuits to circuits that
include voltage and/or current sources.   The target of this extended black-boxing functor will be, not $\Lag\Rel_k$, but $\Aff\Lag\Rel_k$.  

\begin{theorem} 
\label{thm:black-boxing_3}
For any field $k$ and label set $\L$, there exists a unique morphism of props 
\[ \blacksquare \maps \Circ_\L \cong \Fin\Cospan + F(\L) \to \Aff\Lag\Rel_k \] 
such that $\blacksquare \vert_{\Fin\Cospan}$ is the composite
\[    
       \Fin\Cospan \stackrel{H}{\longrightarrow} 
       \Fin\Corel \stackrel{K}{\longrightarrow} \Lag\Rel_k \hookrightarrow \Aff\Lag\Rel_k \]
and $\blacksquare \vert_{F(\L)}$ maps each $\ell \in \L$ to an arbitrarily chosen
Lagrangian affine relation from $k \oplus k$ to itself.    
\end{theorem}

\begin{proof}
The proof mimics that of Theorem \ref{thm:black-boxing_1}.
\end{proof}

Using the formulas given above for the relations between potentials and currents for voltage sources and current sources, we can use this theorem to define black-boxing functors for circuits that include these additional circuit elements.   A similar strategy can be used to define black-boxing for circuits containing other nonlinear circuit elements, such as transistors.  We merely need to expand the target of this black-boxing functor to include all the relations between potentials and currents that arise.

\appendix

\section{Props from symmetric monoidal categories}
\label{sec:symmoncats}

There is a 2-category $\Symm\Mon\Cat$ where:

\begin{itemize}
\item objects are symmetric monoidal categories,
\item morphisms are symmetric monoidal functors, and 
\item 2-morphisms are monoidal natural transformations.   
\end{itemize}
Here our default notions are the `weak' ones (which Mac Lane \cite{Ma13} calls `strong'),  where all laws hold up to coherent natural isomorphism.     Props, on the other hand, are \emph{strict} symmetric monoidal categories where every object is \emph{equal} to a natural number, and morphisms between them are \emph{strict} symmetric monoidal functors sending each object $n$ to itself.   Categorical structures found in nature are often weak.  Thus, to study them using props, one needs to `strictify' them.  Thanks to conversations with Steve Lack we can state the following results, which accomplish this strictification.   

The first question is when a symmetric monoidal category $\C$ is equivalent, as an object of $\Symm\Mon\Cat$, to a prop.    In other words: when is does there exist a prop $\T$ and symmetric monoidal functors $j \maps \T \to \C$, $k \maps \C \to \T$ together with monoidal natural isomorphisms $jk \cong 1_\C$ and $kj \cong 1_\T$?
This is answered by Proposition \ref{prop:strictification_1}:

\vskip 1em 
\noindent \textbf{Proposition 4.3.}  \textit{A symmetric monoidal category $\C$ is equivalent, as an object of $\Symm\Mon\Cat$, to a prop if and only if there is an object $x \in \C$ such that every object of $\C$ is isomorphic to the $n$th tensor power of $x$ for some $n \in \N$.}

\begin{proof}
The `only if' condition is obvious, so suppose that $\C$ is any symmetric monoidal category with $x \in \C$ such that every object of $\C$ is isomorphic to a tensor power of $x$.   We use a method due to A.\ J.\ Power, based on this lemma:

\begin{lemma}[Lemma 3.3, \cite{Po}]  
\label{lem:power} 1. Any functor $f \maps A \to B$ can be factored as $je$ where $e$ is bijective on objects and $j$ is fully faithful.  

2.  Given a square that commutes up to a natural isomorphism $\alpha$:
\[
  \xymatrix{
     A \ddrrtwocell<\omit>{<0>\alpha} \ar[rr]^h \ar[dd]_u && B \ar[dd]^v \\ 
    & &\\
    C \ar[rr]_{g} && D 
  }
\]
where $h$ is bijective on objects and $g$ is fully faithful, there exists a unique functor $w \maps B \to C$ and natural transformation $\beta \maps v \Rightarrow g w$ such that $wh = u$ and $\beta h = \alpha$.  Moreover $\beta$ is a natural isomorphism.
\end{lemma}

Let $\NN$ be the strict symmetric monoidal category with one object for each natural number and only identity morphisms, with the tensor product of objects being given by addition.  There exists a symmetric monoidal functor $f \maps \NN \to \C$ that sends the $n$th object of $\NN$ to $x^{\otimes n} = x \otimes (x \otimes (x \otimes \cdots))$.   By Part 1 of Lemma \ref{lem:power} we can factor $f$ as a composite
\[   \NN \stackrel{e}{\longrightarrow} \T \stackrel{j}{\longrightarrow} \C \]
where $e$ is bijective on objects and $j$ is fully faithful.     By our condition on $\C$, $f$ is essentially surjective.  It follows that $j$ is also essentially surjective, and thus an equivalence of categories.   We claim that $\T$ can be given the structure of a strict symmetric monoidal category making $j$ symmetric monoidal.    It will follow that $\T$ is a prop and $j \maps \T \to \C$ can be promoted to an equivalence in $\Symm\Mon\Cat$.

To prove the claim, we use the 2-monad $P$ on $\Cat$ whose strict algebras are strict symmetric monoidal categories.    For any category $\A$, $P(\A)$ is the `free strict symmetric monoidal category' on $\A$.  Explicitly, an object of $P(\A)$ consists of a finite list $(a_1\dots,a_n)$ of objects of $\A$.  A morphism from $(a_1,\dots,a_n)$ to $(b_1,\dots b_m)$ exists only if $n=m$, in which case it consists of a permutation $\sigma \in S_n$ and a morphism from $a_i$ to $b_{\sigma(i)}$ in $\A$ for each $i$.    For the rest of the 2-monad structure see for example \cite[Sec.\ 4.1]{FGHW}.  Any symmetric monoidal category can be made into a pseudoalgebra of $P$, and then the pseudomorphisms between such pseudoalgebras are the symmetric monoidal functors.   

Power's method \cite{Po} applies to 2-monads that preserve the class of functors that are bijective on objects.  The 2-monad $P$ has this property.

We can take $\NN$ above to be $P(1)$, since they are isomorphic.    Any object $x \in \C$ determines a functor $1 \to \C$ and so a pseudomorphism $F \maps P(1) \to \C$.    Replacing $F$ by an equivalent pseudomorphism if necessary, we can assume $F(n) = x^{\otimes n}$, so the situation in the first paragraph holds with this choice of $F$, and we can factor $F$ as $P(1) \stackrel{e}{\to} \T \stackrel{j}{\to} \C$ as before.  Since $F$ is a pseudomorphism, this square commutes up to a natural isomorphism $\alpha$:
\[
  \xymatrix{
     P(P(1)) \ddrrtwocell<\omit>{<0>\alpha} \ar[rr]^{P(e)} \ar[d]_{m_1} && P(\T) \ar[d]^{P(j)} \\ 
 P(1) \ar[d]_e   & & P(\C) \ar[d]^{a} \\
    \T \ar[rr]_{j} && \C 
  }
\]
where $a$ comes from the pseudoalgebra structure  and $m_1$ comes the multiplication in the 2-monad. The functor $P(e)$ is bijective on objects because $P$ is, and $j\maps \T \to \C$ is fully faithful.  Thus, by Part 2 of Lemma \ref{lem:power}, there exists a unique functor $w\maps P(\T) \to \T$ and natural isomorphism $\beta \maps a P(j) \to jw$ such that $w P(e) =e m_1$ and $\beta P(e) = \alpha$.   Thus, $w$ makes $\T$ into a strict algebra of $P$ and $\beta$ makes $j$ into a pseudomorphism from $\T$ to $\C$.  This proves the claim: $\T$ has been given the structure of a symmetric monoidal category for which $j \maps \T \to \C$ is a symmmetric monoidal functor.
\end{proof}

The second question is when a symmetric monoidal functor $f \maps \T \to \C$
between props is isomorphic, in $\Symm\Mon\Cat$, to a morphism of props.   In other words: when is there a morphism of props $g \maps \T \to \C$ and a monoidal natural isomorphism $f \cong g$?  This is answered by Proposition \ref{prop:strictification_2}:

\vskip 1em
\noindent \textbf{Proposition 4.4.} \textit{Suppose $\T$ and $\C$ are props and $f \maps \T \to \C$ is a symmetric monoidal functor. Then $f$ is isomorphic to a strict symmetric monoidal functor $g \maps \T \to \C$.   If $f(1) = 1$, then $g$ is a morphism of props.}

\begin{proof}
As in the previous proof, let $P$ be the 2-monad on $\Cat$ whose strict algebras
are strict monoidal categories.    The objects of $P(1)$ correspond to natural
numbers, with tensor product being given addition, so we can write the $n$th object
as $n$.  There is a unique strict monoidal functor $e \maps P(1) \to \T$ with $e(n) = n$ for all $n$.   By a result of Blackwell, Kelly and Power \cite[Cor.\ 5.6]{BKP}, any free algebra of a 2-monad is `flexible', meaning that pseudomorphisms out of this algebra are isomorphic to strict morphisms.   Thus, the symmetric monoidal functor $f e \maps P(1) \to \C$ is isomorphic, in $\Symm\Mon\Cat$, to a strict symmetric monoidal functor $h \maps P(1) \to \C$.      Let $\alpha \maps fe \Rightarrow h$ be the isomorphism.

We can define a strict symmetric monoidal functor $g \maps \T \to \C$ as follows.
On objects, define $g(n) = h(n)$.  For any morphism $\phi \maps m \to n$, 
there is a unique morphism $g(\phi)$ making this square commute:
\[    
 \xymatrix{
     f(m)  \ar[r]^{f(\phi)} \ar[d]_{\alpha_m} & f(n) \ar[d]^{\alpha_n} \\ 
      g(m) \ar[r]_{g(\phi)} & g(n)  
  }
\]
One can check that $g$ is a strict symmetric monoidal functor.  The above square
gives a natural isomorphism between $f$ and $g$, which by abuse of language
we could call $\alpha \maps f \Rightarrow g$.   It is easy to check that this is
a monoidal natural isomorphism.
\end{proof}

It is worth noting that Propositions \ref{prop:strictification_1} and \ref{prop:strictification_2}, and the proofs just given, generalize straightforwardly from props to `$C$-colored' props, with 
$\N$ replaced everywhere by the free commutative monoid on the set of colors, $C$. 

\section{The adjunction between props and signatures}
\label{sec:monadic}

Our goal in this section is to prove Proposition \ref{prop:monadic}:

\vskip 1em 
\noindent \textbf{Proposition 5.2.} \textit{There is a forgetful functor
\[           U \maps \PROP \to \Set^{\N \times \N}  \]
sending any prop to its underlying signature and any morphism of props to its
underlying morphism of signatures.  This functor is monadic.}

\begin{proof}
The plan of the proof is as follows.  We show that props are models of a typed Lawvere theory 
$\Theta_\PROP$ whose set of types is $\N \times \N$.  We write this as
follows:
\[              \PROP \simeq \Mod(\Theta_\PROP)  .\]
This lets us apply the following theorem, which says that for any typed
Lawvere theory $\Theta$ with $T$ as its set of types, the forgetful functor
\[       U \maps \Mod(\Theta) \to \Set^T \]
is monadic.    The desired result follows.

\begin{theorem} 
\label{thm:monadic}
 If $\Theta$ is a typed Lawvere theory with $T$ as its set of types,
then the forgetful functor $U\maps \Mod(\Theta) \to \Set^T$ is monadic.
\end{theorem}

\begin{proof}
The origin of this theorem may be lost in the mists of time, though the case $T = 1$
is famous, and was proved in Lawvere's thesis \cite{Law}. The general theorem is 
Theorem A.41 in Ad\'amek, Rosick\'y and Vitale's book on algebraic theories \cite{ARV}.
Trimble has proved a further generalization where $\Set$ is replaced by any category $\C$ that 
is cocomplete and has the property that finite products distribute over colimits \cite{Tr}.
An even more general result appears in the work of Nishisawa and Power \cite{NP}.
\end{proof}

Now let us define the terms here and see how this result applies to our situation.
First, for any set $T$, let $\N[T]$ be the set of finite linear combinations of elements
of $T$ with natural number coefficients.  This becomes a commutative monoid under
addition, in fact the free commutative monoid on $T$.  Define a \define{$T$-typed Lawvere theory} to be a category $\Theta$ with finite products whose set of objects is $\mathbb{N}[T]$, with the product of objects given by addition in $\N[T]$.   We call the elements of $T$ \define{types}.  

Suppose $\Theta$ is an $T$-typed Lawvere theory.  
Let \define{$\Mod(\Theta)$} be the category whose objects are functors
$F \maps \Theta \to \Set$ preserving finite products, and whose morphisms are
natural transformations between such functors.  We call an object of $\Mod(\Theta)$ a \define{model} of $\Theta$, and call a morphism in $\Mod(\Theta)$ a 
\define{morphism of models}.

There is an inclusion $T \hookrightarrow \N[T]$, since $\N[T]$ is the free commutative
monoid on $T$.  Thus, any model $M$ of $\Theta$ gives, for each type
$t \in T$, a set $M(t)$.   Similarly, any morphism of models 
$\alpha \maps M \to M'$ gives, for each type $t \in T$, a function
$\alpha_t \maps M(t) \to M'(t)$.  Indeed, there is a functor
\[         U \maps \Mod(\Theta) \to \Set^T   \]
with $U(M)(t) = M(t)$ for each model $M$ and $U(\alpha)(t) = \alpha_t$ for each
morphism of models $\alpha \maps M \to M'$.   If we call $\Set^T$ the category of 
\define{signatures} for $T$-typed Lawvere theories, then $U$ sends models to their
underlying signatures and morphisms of models to morphisms of their underlying
signatures.  Theorem \ref{thm:monadic} says that $U$ is monadic.

To complete the proof of Proposition \ref{prop:monadic} we need to give 
a typed Lawvere theory $\Theta_{\PROP}$ whose models are props.  We can do this by giving a 
`sketch'.  Since this idea has been described very carefully by Barr and Wells \cite{BW} and others, we content ourselves with a quick intuitive explanation of the special case we really need which could be called a `products sketch'.   This is a way of presenting a $T$-typed Lawvere
theory by specifying a set $T$ of generating objects (or types), a set of generating morphisms between formal products of these generating objects, and a set of relations given as commutative diagrams.  These commutative diagrams can involve the generating morphisms and also morphisms built from these using the machinery available in a category with finite products.  

To present the typed Lawvere theory $\Theta_{\PROP}$ we start by taking $T = \N \times \N$.  For the purposes of easy comprehension, we call the corresponding generating objects $\hom(m,n)$ for $m,n \in \N$, since these will be mapped by any model of $\Theta_{\PROP}$ to 
the homsets that a prop must have.   We then include the following generating morphisms:

\begin{itemize}
\item
For any $m$ we include a morphism $\iota_m \maps 1 \rightarrow \hom(m,m)$. These give rise to the identity morphisms in any model of $\Theta_{\PROP}$.
\item
For any $\ell,m,n$, we include a morphism $\circ_{\ell,m,n} \maps \hom(m,n) \times \hom(\ell,m) \rightarrow \hom(\ell,n)$. These give us the ability to compose morphisms in any model of $\Theta_{\PROP}$.
\item 
For any $m,n,m',n'$, we include a morphism $\otimes_{m,n,m',n'} \maps \hom(m',n') \times \hom(m,n) \rightarrow \hom(m+m',n+n')$. These allow us to take the tensor product of morphisms in any model of $\Theta_{\PROP}$.
\item 
For any $m,m'$, we include a morphism $b_{m,m'} \maps 1 \rightarrow \hom(m+m',m'+m)$. These give the braidings in any model of $\Theta_{\PROP}$.
\end{itemize}

Finally, we impose relations via the following commutative diagrams.  In these diagrams, unlabeled arrows are morphisms provided by the structure of a category with 
finite products.  We omit the subscripts on morphisms since they can be inferred 
from context.  We begin with a set of diagrams, one for each $m,n \in \N$, that ensure associativity of composition in any model of $\Theta_{\PROP}$:

\begin{center}
\begin{tikzpicture}[->,>=stealth',node distance=2.2cm, auto]
 \node (A) {$\hom(m,n)\times \hom(l,m)\times \hom(k,l)$};
 \node (B) [below of=A,left of=A] {$\hom(l,n)\times \hom(k,l)$};
 \node (C) [below of=A,right of=A] {$\hom(m,n)\times \hom(k,m)$};
 \node (D) [right of=B,below of=B] {$\hom(k,n)$};

 \draw[->] (A) to node [swap]{$\circ \times 1$}(B);
 \draw[->] (A) to node {$1 \times \circ$}(C);
 \draw[->] (B) to node [swap]{$\circ$}(D);
\draw[->] (C) to node {$\circ$}(D);
\end{tikzpicture}
\end{center}

\noindent Next, a diagram that ensures each $\iota_m$ picks out an identity morphism:

\begin{center}
\begin{tikzpicture}[->,>=stealth',node distance=3.5cm, auto]
 \node (A) {$1 \times \hom(m,n)$};
 \node (B) [right of=A,xshift=1cm] {$\hom(m,n)$};
 \node (C) [right of=B,xshift=1cm] {$\hom(m,n) \times 1$};
 \node (D) [below of=A,yshift=1cm] {$\hom(n,n)\times \hom(m,n)$};
\node (E) [below of=B,yshift=1cm] {$\hom(m,n)$};
\node (F) [below of=C,yshift=1cm] {$\hom(m,n)\times \hom(m,m)$};
 \draw[->] (A) to node {$$}(B);
 \draw[->] (C) to node {$$}(B);
 \draw[->] (A) to node {$\iota \times 1$}(D);
\draw[->] (B) to node {$1$}(E);
 \draw[->] (C) to node {$1\times \iota$}(F);
 \draw[->] (D) to node {$\circ$}(E);
 \draw[->] (F) to node [swap]{$\circ$}(E);
\end{tikzpicture}
\end{center}

\noindent Next, a diagram that ensures associativity of the tensor product of morphisms:

\begin{center}
\begin{tikzpicture}[->,>=stealth',node distance=2.8cm, auto]
 \node (A) {$\hom(m,n)\times \hom(m',n')\times \hom(m'',n'')$};
 \node (B) [below of=A,left of=A] {$\hom(m+m',n+n')\times \hom(m''+n'')$};
 \node (C) [below of=A,right of=A] {$\hom(m,n)\times \hom(m'+m'',n'+n'')$};
 \node (D) [right of=B,below of=B] {$\hom(m+m'+m'',n+n'+n'')$};

 \draw[->] (A) to node [swap]{$\otimes \times 1$}(B);
 \draw[->] (A) to node {$1 \times \otimes$}(C);
 \draw[->] (B) to node [swap]{$\otimes$}(D);
\draw[->] (C) to node {$\otimes$}(D);
\end{tikzpicture}
\end{center}

\noindent Next, a diagram that ensures that tensoring with the identity morphism on $0$ acts trivially on morphisms:

\begin{center}
\begin{tikzpicture}[->,>=stealth',node distance=3.5cm, auto]
 \node (A) {$1 \times \hom(m,n)$};
 \node (B) [right of=A,xshift=1cm] {$$};
 \node (C) [right of=B,xshift=1cm] {$\hom(m,n) \times 1$};
 \node (D) [below of=A,yshift=1cm] {$\hom(0,0)\times \hom(m,n)$};
\node (E) [below of=B,yshift=1cm] {$\hom(m,n)$};
\node (F) [below of=C,yshift=1cm] {$\hom(m,n)\times \hom(0,0)$};
 \draw[->] (A) to node {$$}(E);
 \draw[->] (C) to node {$$}(E);
 \draw[->] (A) to node [swap]{$\iota \times 1$}(D);
 \draw[->] (C) to node {$1\times \iota$}(F);
 \draw[->] (D) to node {$\otimes$}(E);
 \draw[->] (F) to node [swap]{$\otimes$}(E);
\end{tikzpicture}
\end{center}

\noindent Next, a diagram that ensures that the tensor product preserves composition:

\begin{center}
\begin{tikzpicture}[->,>=stealth',node distance=1.2cm, auto]
 \node (A) {$\hom(m',n')\times \hom(l',m')\times \hom(m,n)\times \hom(l,m)$};
 \node (B) [below of=A] {$\hom(l',n')\times \hom(l,n)$};
 \node (C) [right of=A,xshift=7.25cm] {$\hom(m',n')\times \hom(m,n) \times \hom(l',m') \times \hom(l,m)$};
 \node (D) [below of=C] {$\hom(m+m',n+n')\times \hom(l+l',m+m')$};
 \node (E) [below of=B,right of=B,xshift=2.25cm] {$\hom(l+l',n+n')$};

 \draw[->] (A) to node [swap]{$\circ \times \circ$}(B);
 \draw[->] (A) to node {$$}(C);
 \draw[->] (B) to node [swap]{$\otimes$}(E);
\draw[->] (C) to node {$\otimes \times \otimes$}(D);
\draw[->] (D) to node {$\circ$}(E);
\end{tikzpicture}
\end{center}

\noindent Next, a diagram that ensures the naturality of the braiding:

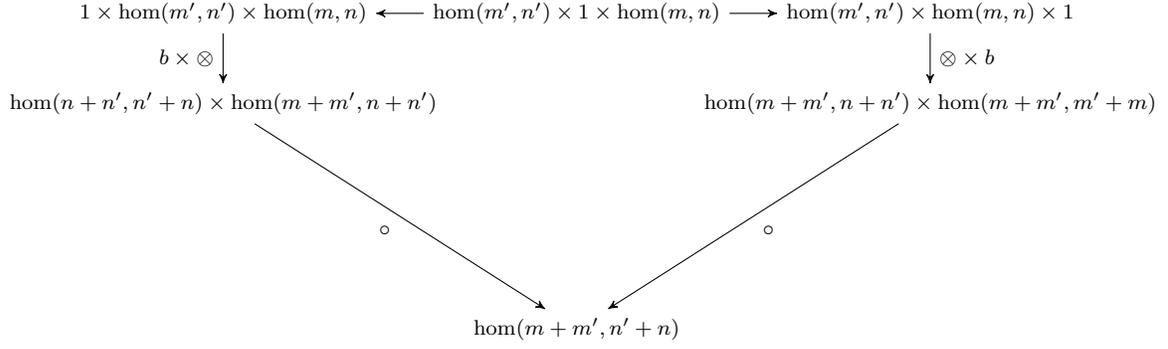
\begin{center}
\begin{tikzpicture}[->,>=stealth',node distance=2.2cm, auto]
 \node (A) {$1 \times \hom(m',n')\times \hom(m,n)$};
 \node (B) [right of=A,xshift=2.5cm] {$\hom(m',n')\times 1 \times \hom(m,n)$};
 \node (C) [right of=B,xshift=2.5cm] {$\hom(m',n') \times \hom(m,n) \times 1$};
 \node (D) [below of=A,yshift=1cm] {$\hom(n+n',n'+n)\times \hom(m+m',n+n')$};
\node (E) [below of=C,yshift=1cm] {$\hom(m+m',n+n')\times \hom(m+m',m'+m)$};
\node (F) [below of=B,yshift=-2cm] {$\hom(m+m',n'+n)$};

 \draw[->] (B) to node [swap]{$$}(A);
 \draw[->] (B) to node {$$}(C);
 \draw[->] (A) to node [swap]{$b\times \otimes$}(D);
\draw[->] (D) to node [swap]{$\circ$}(F);
 \draw[->] (C) to node {$\otimes \times b$}(E);
 \draw[->] (E) to node {$\circ$}(F);
\end{tikzpicture}
\end{center}

Next, a diagram that ensures that the braiding is a symmetry:

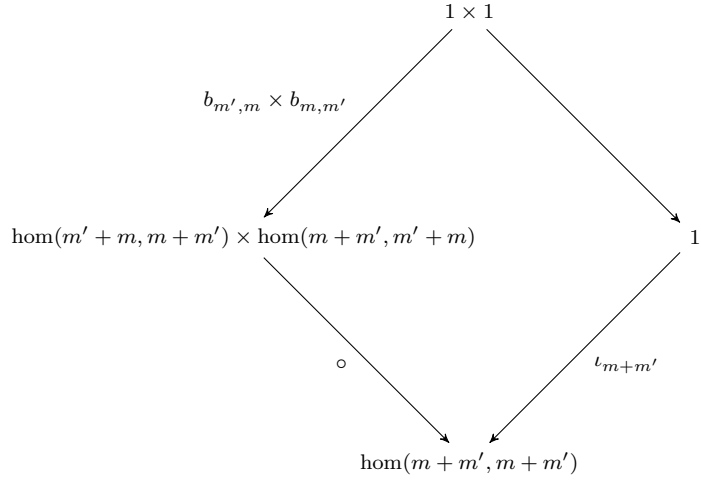
\begin{center}
\begin{tikzpicture}[->,>=stealth',node distance=3cm, auto]
 \node (A) {$1 \times 1$};
 \node (B) [below of=A,left of=A] {$\hom(m'+m,m+m')\times \hom(m+m',m'+m)$};
 \node (C) [below of=A,right of=A] {$1$};
 \node (D) [right of=B,below of=B] {$\hom(m+m',m+m')$};

 \draw[->] (A) to node [swap]{$b_{m',m}\times b_{m,m'}$}(B);
 \draw[->] (A) to node {$$}(C);
 \draw[->] (B) to node [swap]{$\circ$}(D);
\draw[->] (C) to node {$\iota_{m+m'}$}(D);
\end{tikzpicture}
\end{center}

\noindent
Next, a diagram that ensures that the braidings $b_{0,n}$ are identity morphisms:

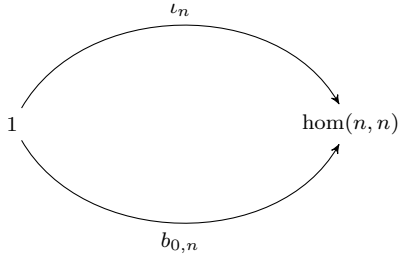
\begin{center}
\begin{tikzpicture}[->,>=stealth',node distance=2cm, auto]
 \node (B) {$1$};
 \node (C) [right of=B,xshift=2.5cm] {$\hom(n,n)$};

 \draw[->]  [in=120, out=60, looseness=1.0](B) to node {$\iota_n$}(C);
 \draw[->] [in=-120, out=-60, looseness=1.0] (B) to node [swap]{$b_{0,n}$}(C);
\end{tikzpicture}
\end{center}

\noindent
Finally, we need a diagram to ensure the braiding obeys the hexagon identities. However, since the associators are trivial and the braiding is a symmetry, the two hexagons reduce to a 
single triangle.   To provide for this, we use the following diagram:

\begin{center}
\begin{tikzpicture}[->,>=stealth',node distance=2.5cm, auto]
 \node (A) {$1$};
 \node (B) [right of=A,xshift=7.5cm] {$1 \times 1 \times 1 \times 1$};
 \node (C) [below of=B] {$\hom(m+m'',m''+m)\times \hom(m',m')\times \hom(m'',m'') \times \hom(m+m',m'+m)$};
 \node (D) [below of=C] {$\hom(m'+m+m'',m'+m''+m)\times \hom(m+m'+m'',m'+m+m'')$};
 \node (E) [below of=D,left of=D,xshift=-1.75cm] {$\hom(m+m'+m'',m'+m''+m)$};

 \draw[->] (B) to node {$$}(A);
 \draw[->] (B) to node {$b\times \iota \times \iota \times b$}(C);
 \draw[->] (C) to node {$\otimes \times \otimes$}(D);
\draw[->] (D) to node {$\circ$}(E);
\draw[->] (A) to node  [swap]{$b$}(E);

\end{tikzpicture}
\end{center}

This completes the list of commutative diagrams in the sketch for $\Theta_{\PROP}$.  These diagrams simply state the definition of a PROP, so there is a 1-1 correspondence between models of $\Theta_{\PROP}$ in $\Set$ and props.  Similarly, morphisms of models of $\Theta_{\PROP}$ in $\Set$ correspond to morphisms of props.  This gives an isomorphism of
categories $\PROP \cong \Mod(\Theta_{\PROP})$ as desired. This concludes the proof.
\end{proof}

It is worth noting that Proposition \ref{prop:monadic}, and the proof just given, generalize straightforwardly from props to `$C$-colored' props, with $\N$ replaced everywhere by the free commutative monoid on a set $C$, called the set of `colors'.    There is also a version for operads, a version for $C$-colored operads, and a version for $T$-typed Lawvere theories: there is a typed Lawvere theory whose models in $\Set$ are $T$-typed Lawvere theories!  In each case we simply need to write down a sketch that describes the structure under consideration.

Proposition \ref{prop:monadic} has a wealth of consequences; we conclude with two that
Erbele needed in his work on control theory \cite[Prop.\ 6]{E}.  In rough terms, these results 
say that adding generators to a presentation of a prop $P$ gives a new prop $P'$ having
$P$ as a sub-prop, while adding equations gives a new prop $P'$ that is a quotient of $P$.  
We actually prove more general statements.

In all that follows, let $\Theta$ be a $T$-typed Lawvere theory.    Let $U \maps \Mod(\Theta) \to \Set^T$ be the forgetful functor and $F \maps \Set^T \to \Mod(\Theta)$ its left adjoint.  Further, suppose we have two coequalizer diagrams in $\Mod(\Theta)$:
\[
    \xymatrix{
      F(E) \ar@<-.5ex>[r]_{\rho} \ar@<.5ex>[r]^{\lambda} & F(\Sigma) \ar[r]^-{\pi} & P   
      }
\] 
\[
	\xymatrix{     
       F(E') \ar@<-.5ex>[r]_{\rho'} \ar@<.5ex>[r]^{\lambda'} & F(\Sigma') \ar[r]^-{\pi'} & P'
    }
\]
together with morphisms $f \maps E \to E'$, $g \maps \Sigma \to \Sigma'$ such that these squares commute:
\[ 
\xymatrix{
	F(E) \ar[r]^{\lambda} \ar[d]_{F(f)} & F(\Sigma) \ar[d]^{F(g)} \\
	F(E') \ar[r]^{\lambda'}                        & F(\Sigma') 
	}
\qquad \qquad
	\xymatrix{
	F(E) \ar[r]^{\rho} \ar[d]_{F(f)} & F(\Sigma) \ar[d]^{F(g)} \\
	F(E') \ar[r]^{\rho'}                        & F(\Sigma') 
	}
\] 
Then, thanks to the universal property of $P$, there exists a unique morphism $h \maps P \to P'$ making the square at right commute:
\[ 
\xymatrix{
F(E) \ar@<-.5ex>[r]_{\rho} \ar@<.5ex>[r]^{\lambda} \ar[d]_{F(f)} 
& F(\Sigma) \ar[r]^-{\pi} \ar[d]_{F(g)}             & P \ar[d]^{h} \\
F(E') \ar@<-.5ex>[r]_{\rho'} \ar@<.5ex>[r]^{\lambda'} 
&	F(\Sigma') \ar[r]^-{\pi'}                        & P' 
	}
\]
In this situation, adding extra equations makes $P'$ into a quotient object of $P$.  More precisely, and also more generally:

\begin{corollary}
If $g$ is an epimorphism, then $h$ is a regular epimorphism.  
\end{corollary}

\begin{proof} 
Given that $g$ is an epimorphism in $\Set^T$, it is a regular epimorphism.
So is $F(g)$, since left adjoints preserve regular epimorphisms, and so is $\pi'$, by definition.
It follows that $\pi' \circ F(g) = h \circ \pi$ is a regular epimorphism, and thus so is $h$.
\end{proof}
	
One might hope that in the same situation, adding extra generators makes $P$ into a subobject of $P'$.    More precisely, one might hope that if $f$ is an isomorphism and $g$ is a monomorphism, $h$ is a monomorphism.   This is not true in general, but it is when the typed Lawvere theory $\Theta$ is $\Theta_{\PROP}$.

To see why some extra conditions are needed, consider a counterexample provided by Charles Rezk \cite{Rezk}.  There is a typed Lawvere theory with two types whose models consist of:
\begin{itemize}
\item a ring $R$,
\item a set $S$,
\item a function $f \maps S \to R$ with $f(s) = 0$ and $f(s) = 1$ for 
all $s \in S$.
\end{itemize}
Thanks to the peculiar laws imposed on $f$, the only models are pairs $(R,S)$ where $R$ is an arbitrary ring and $S$ is empty, and pairs $(R,S)$ where $R$ is a terminal ring (one with $0 = 1$) and $S$ is an arbitrary set.   The free model on $(\emptyset,\emptyset) \in \Set^2$ is $(\Z,\emptyset)$, while the free model on $(\emptyset, 1)$ is $(\{0\},1)$.   Thus, the monomorphism $(\emptyset,\emptyset) \to (\emptyset,1)$ in $\Set^2$ does not induce a monomorphism between the corresponding free models: the extra generator in the set part of $(R,S)$ causes the ring part to `collapse'.

This problem does not occur for Lawvere theories with just one type, nor does it happen for typed Lawvere theories that arise from typed operads, more commonly known as `colored' operads \cite{BM,Yau}. A typed Lawvere theory arises from a typed operad when it can be presented in terms of operations obeying purely equational laws for which each variable appearing in an equation shows up exactly once on each side.  The laws governing props are of this form: for example, the operations for composition of morphisms obey the associative law
\[            (f \circ g) \circ h = f \circ (g \circ h)  .\]
It follows that $\Theta_{\PROP}$ arises from a typed operad, so the following corollary 
applies to this example:

\begin{corollary}
Suppose that either $\Theta$ is a $T$-typed Lawvere theory with $T = 1$ or $\Theta$ arises from a $T$-typed operad.  If $f$ is an isomorphism and $g$ is a monomorphism, then $h$ is a monomorphism.
\end{corollary}
  	
\begin{proof}
We may assume without loss of generality that $f \maps E \to E'$ is the identity and $g \maps \Sigma \to \Sigma'$ is monic.  Since monomorphisms in $\Set^T$ are just $T$-tuples of injections, we can write $\Sigma' \cong \Sigma + \Delta$ for some signature $\Delta$ in such a way that 
$g \maps \Sigma \to \Delta$ is isomorphic to the coprojection $\Sigma \to \Sigma + \Delta$.   It follows that $F(g)$ is isomorphic to the coprojection $i \maps F(\Sigma) \to F(\Sigma) + F(\Delta)$, and the diagram
\[ 
\xymatrix{
F(E) \ar@<-.5ex>[r]_{\rho} \ar@<.5ex>[r]^{\lambda} \ar[d]_{F(f)} 
& F(\Sigma) \ar[r]^-{\pi} \ar[d]_{F(g)}             & P \ar[d]^{h} \\
F(E') \ar@<-.5ex>[r]_{\rho'} \ar@<.5ex>[r]^{\lambda'} 
&	F(\Sigma') \ar[r]^-{\pi'}                        & P' 
	}
\]
is isomorphic to this diagram:
\[ 
\xymatrix{
F(E) \ar@<-.5ex>[r]_{\rho} \ar@<.5ex>[r]^{\lambda} \ar[d]_{1} 
& F(\Sigma) \ar[r]^-{\pi} \ar[d]_{i}             & P \ar[d]^{j} \\
F(E') \ar@<-.5ex>[r]_-{\rho} \ar@<.5ex>[r]^-{\lambda} 
&	F(\Sigma) + F(\Delta) \ar[r]^-{\pi + 1}                        & P + F(\Delta)
	}
\]
where $j$ is the coprojection from $P$ to $P + F(\Delta)$.  Thus it suffices to prove the following:

\begin{lemma}
Suppose that either $\Theta$ is a $T$-typed Lawvere theory with $T = 1$ or $\Theta$ arises from a $T$-typed operad.  If $P \in \Mod(\Theta)$ and $\Delta \in \Set^T$ then the coprojection
$j \maps P \to P + F(\Delta)$ is a monomorphism.
\end{lemma}

\noindent \textsl{Proof.}  We thank Todd Trimble for this proof.  First suppose $T = 1$.  To show that $j$ is monic it suffices to show that $U(j)$ is injective, since $U \maps \Mod(\Theta) \to \Set$ is faithful and thus it reflects monomorphisms \cite[Prop.\ 11.8]{ARV}.  Either $U(P)$ is empty and the injectivity is trivial, or $U(P)$ is nonempty, in which case we can split the coprojection $j \maps P \to P + F(\Delta)$, since all we need for this is a morphism $F(\Delta) \to P$, or equivalently, a function $\Delta \to U(P)$.

Next suppose that $\Theta$ is a $T$-typed Lawvere theory that comes from a $T$-typed operad $O$.  Here we can use the following construction: given $P \in \Mod(\Theta)$, we can form a model $P^\ast \in \Mod(\Theta)$ that has an extra element for each type $t \in T$.  To do this, we first set
\[  M^\ast(t) = M(t) \sqcup \{x_t\}   \]
for all $t \in T$, where $x_t$ is an arbitrary extra element.  Then, we make $P^*$ into an algebra of $O$ as follows.  Suppose $f \in O(t_1, \dots, t_n; t)$ is any operation of $O$ with inputs of type
$t_1, \dots, t_n \in T$ and output of type $t \in T$.   Since $P$ is an algebra of $O$, $f$ acts on $P$ as some function 
\[ P(f) \maps P(t_1) \times \cdots \times P(t_n) \to P(t)   .\]
Then we let $f$ act on $P^*$ as the function 
\[ P^*(f) \maps P^\ast(t_1) \times \cdots \times P^\ast(t_n) \to P^\ast(t)   .\]
that equals $P(f)$ on $n$-tuples $(p_1, \dots, p_n)$ with $p_i \in P(t_i)$ for all $i$, and otherwise gives $x_t$.   One can readily check that this really defines an algebra of $O$ and thus a model of $\Theta$.   The evident morphism of models $k \maps P \to P^*$ is monic because again $U$ is faithful \cite[Prop.\ 14.8]{ARV} and the underlying morphism of signatures $U(k) \maps U(P) \to U(P^*)$ is monic.

With this construction in hand, we can show that the coprojection $j \maps P \to P + F(\Delta)$ is monic.  We have just constructed a monomorphism $k \maps P \to P^*$. Now extend this to a morphism $\ell \maps P + F(\Delta) \to P^*$: to do this, we just need a morphism $F(\Delta) \to  P^*$, which we can take to be the one corresponding to the map $X \to U(P^*)$ whose component $X(t) \to U(P^*)(t)$ is the function mapping every element of $X(t)$ to $x_t$.   Then, we have $k = \ell \circ j$, and since $k$ is monic, $j$ must be as well.  
\end{proof}


\begin{thebibliography}{99}

\bibitem{ARV} J.\ Ad\'amek, J.\ Rosick\'y and E.\ M.\ Vitale, \textsl{Algebraic Theories}, 
Cambridge U.\ Press, Cambridge, 2011.

\bibitem{AC} S.\ Abramsky and B.\ Coecke, A categorical 
semantics of quantum protocols, in {\sl Proceedings of the 19th IEEE Conference on 
Logic in Computer Science (LiCS04)}, IEEE Computer Science Press, 2004.  Also available as
\href{http://arxiv.org/abs/quant-ph/0402130}{arXiv:quant-ph/0402130}.
      
\bibitem{BE} J.\ C.\ Baez, J.\ Erbele, Categories in control, 
\textsl{Theory Appl.\ Categ.} \textbf{30} (2015), 836--881. Available at
\href{http://www.tac.mta.ca/tac/volumes/30/24/30-24abs.html}
{http://www.tac.mta.ca/tac/volumes/30/24/30-24abs.html}.

\bibitem{BF} J.\ C.\ Baez, B.\ Fong, A compositional framework for passive linear circuits.  
Available at \href{http://arxiv.org/abs/1504.05625}{arXiv:1504.05625}. 

\bibitem{BFP} J.\ C.\ Baez, B.\ Fong and B.\ S.\ Pollard, A compositional framework for
Markov processes, \textsl{Jour.\ Math.\ Phys.} \textbf{57} (2016), 033301.   Available at \href{https://arxiv.org/abs/1508.06448}{arXiv:1508.06448}.

\bibitem{BW} M.\ Barr, C.\ Wells, \textsl{Toposes, Triples and Theories}, 
\textsl{Reprints in Theory Appl.\ Categ.} \textbf{12} (2005), 1--288.
 Available at \href{http://www.tac.mta.ca/tac/reprints/articles/12/tr12abs.html}{http://www.tac.mta.ca/tac/reprints/articles/12/tr12abs.html}.
 

\bibitem{BM} C.\ Berger and I.\ Moerdijk, Resolution of coloured operads and rectification of homotopy algebras, \textsl{Contemp.\ Math.\ } \textbf{431} (2007), 31--58.
Available as \href{https://arxiv.org/abs/math/0512576}{arXiv:math/0512576}.

\bibitem{BKP} R.\ Blackwell, G.\ M.\ Kelly, A.\ J.\ Power,  Two-dimensional
monad theory,  \textsl{J.\ Pure Appl.\ Alg.} \textbf{59} (1989), 1--41.
Available at \href{http://www.sciencedirect.com/science/article/pii/0022404989901606}{http://www.sciencedirect.com/science} \break \href{http://www.sciencedirect.com/science/article/pii/0022404989901606}{/article/pii/0022404989901606}.

\bibitem{BSZ} F.\ Bonchi, P.\ Soboci\'nski, F. Zanasi, A categorical
semantics of signal flow graphs, in \textsl{CONCUR 2014}, Springer Lecture Notes in
Computer Science \textbf{8704}, Springer, Berlin, 2014, pp.\ 435--450.

\bibitem{BSZ2} F.\ Bonchi, P.\ Soboci\'nski, F. Zanasi, Interacting
Hopf algebras, \textsl{J.\ Pure Appl.\ Alg.} \textbf{221} (2017), 144--184.  
Available as \href{https://arxiv.org/abs/1403.7048}{arXiv:1403.7048}.

\bibitem{Brown} F.\ T.\ Brown, \textsl{Engineering System Dynamics: a 
Unified Graph-Centered Approach}, Taylor and Francis, New York, 2007. 

\bibitem{CW}  A.\ Carboni, R.\ F.\ C.\ Walters, Cartesian bicategories I,
\textsl{J. Pure Appl. Alg.} {\bf 49} (1987), 11--32. 

\bibitem{Co} K.\ Courser, A bicategory of decorated cospans, 
\textsl{Theory Appl.\ Categ.} \textbf{32} (2017), 985--1027.  Available
at \href{http://www.tac.mta.ca/tac/volumes/32/29/32-29abs.html}{http://www.tac.mta.ca/tac/volumes/32/29/32-29abs.html}.

\bibitem{CF} B.\ Coya, B.\ Fong, Corelations are the prop for extraspecial commutative Frobenius monoids, \textsl{Theory Appl.\ Categ.} \textbf{32} (2017), 380--395. Available at
\href{http://www.tac.mta.ca/tac/volumes/32/11/32-11abs.html} {http://www.tac.mta.ca/tac/volumes/32/11/32-11abs.html}.

\bibitem{El} D.\ Ellerman, An introduction to partition logic, \textsl{Logic
Journal of the Interest Group in Pure and Applied Logic} \textbf{22} (2014), 94--125.  
Available at \href{http://doi.org/10.1093/jigpal/jzt036}{http://doi:10.1093/jigpal/jzt036}.

\bibitem{E} J.\ Erbele, \textsl{Categories in Control: Applied PROPs}, Ph.D.\ thesis,
Department of Mathematics, University of California, 2016.  Available as
\href{https://arxiv.org/abs/1611.07591}{arXiv:1611.07591}.

\bibitem{Fe} R.\ Feynman, Space-time approach to quantum electrodynamics,
\textsl{Phys.\ Rev.} \textbf{76} (1949), 769--789.

\bibitem{FGHW} M.\ Fiore, N.\ Gambino, M.\ Hyland, G.\ Wynskel, The cartesian closed bicategory of generalised species of structures, \textsl{J.\ Lond.\ Math.\ Soc.}
 \textbf{77} (2008), 203--220.  Available at \href{https://www.cl.cam.ac.uk/~mpf23/papers/PreSheaves/Esp.pdf}{https://www.cl.cam.ac.uk/$\sim$mpf23/papers/PreSheaves/Esp.pdf}.

\bibitem{Fo} J.\ W.\ Forrester, {\sl Industrial Dynamics}, MIT Press, Cambridge, Massachusetts, 1961.

\bibitem{Fo3} B.\ Fong, Decorated corelations.  Available as \href{https://arxiv.org/abs/1703.09888}{arXiv:1703.09888}. 

\bibitem{Fo1} B.\ Fong, Decorated cospans, \textsl{Theory Appl.\ Categ.} 
\textbf{30} (2015), 1096--1120.   Available as 
\href{http://arxiv.org/abs/1502.00872}{arXiv:1502.00872}.

\bibitem{Fo2} B.\ Fong, \textsl{The Algebra of Open and Interconnected Systems},
Ph.D.\ thesis, Computer Science Department, University of Oxford, 2016.
Available as \href{https://arxiv.org/abs/1609.05382}{arXiv:1609.05382}.

\bibitem{FRS} B.\ Fong, P.\ Rapisarda, P.\ Sobocinski, A categorical approach
to open and interconnected systems. Available as
\href{http://arxiv.org/abs/1510.05076}{arXiv:1510.05076}.

\bibitem{GS} V.\ Guillemin and S.\ Sternberg, \textsl{Symplectic Techniques in Physics},
Cambridge U.\ Press, Cambridge, 1990.

\bibitem{JS1} A.\ Joyal, R.\ Street, The geometry of tensor calculus I,
\textsl{Adv.\ Math.} \textbf{88} (1991), 55–112.  Available at \href{http://doi.org/10.1016/0001-8708(91)90003-P}{http://doi.org/10.1016/0001-8708(91)90003-P}.

\bibitem{Ka} D.\ Kaiser, \textsl{Drawing Theories Apart: The Dispersion of Feynman
Diagrams in Postwar Physics}, U.\ Chicago Press, Chicago, 2005.

\bibitem{KMR} D.\ C.\ Karnopp, D.\ L.\ Margolis, R.\ C.\ Rosenberg, 	
\textsl{System Dynamics: a Unified Approach}, Wiley, New York, 1990.

\bibitem{La} S.\ Lack, Composing PROPs, \textsl{Theory Appl.\ Categ.} \textbf{13} (2004), 147--163. Available at \href{http://www.tac.mta.ca/tac/volumes/13/9/13-09abs.html}{http://www.tac.mta.ca/tac/volumes/13/9/13-09abs.html}.

\bibitem{Law} F.\ W.\ Lawvere, Functorial semantics of algebraic theories, Ph.D.\ thesis,
Department of Mathematics, Columbia University, 1963.   Also in
\textsl{Reprints in Theory Appl.\ Categ.} {\bf 5}  (2003), 1--121.  Available at
\href{http://www.tac.mta.ca/tac/reprints/articles/5/tr5abs.html}
{http://www.tac.mta.ca/tac/reprints/articles/5/tr5abs.html}.

\bibitem{LR} F.\ W.\ Lawvere, R.\ Rosebrugh, \textsl{Sets for Mathematics}, Cambridge 
U.\ Press, Cambridge, UK, 2003.

\bibitem{Ma65} S.\ Mac\ Lane, Categorical algebra, \textsl{Bull.\ Amer.\ Math.\ Soc.} \textbf{71} (1965), 40--106.  Available at \href{http://doi.org/10.1090/S0002-9904-1965-11234-4}{doi:10.1090/S0002-9904-1965-11234-4}.

\bibitem{Ma13} S.\ Mac\ Lane, \textsl{Categories for the Working Mathematician},
Springer, Berlin, 2013.

\bibitem{NP} K.\ Nishizawa and J.\ Power, Lawvere theories enriched over a general base, 
\textsl{J.\ Pure Appl.\ Alg.} \textbf{213} (2009), 377--386.

\bibitem{Od} H.\ T.\ Odum, \textsl{Ecological and General Systems: An
Introduction to Systems Ecology}, Wiley, New York, 1984.

\bibitem{Ol} H.\ F.\ Olson, \textsl{Dynamical Analogies}, Van Nostrand, New York, 1943.  
Available at \href{https://archive.org/details/DynamicalAnalogies}
{https://archive.org/details/DynamicalAnalogies}.

\bibitem{Pa} H.\ M.\ Paynter, \textsl{Analysis and Design of Engineering Systems}, 
MIT Press, Cambridge, Massachusetts, 1961.

\bibitem{Pi} T.\ Pirashvili, On the PROP corresponding to bialgebras,
\textsl{Cah.\ Top. G\'{e}om.\ Diff.\ Cat.} \textbf{43}(3):221--239, 2002. Available as
\href{http://arxiv.org/abs/math/0110014}{arXiv:math/0110014}.
  
\bibitem{Po} A.\ J.\ Power. A general coherence result, \textsl{J.\ Pure Appl.\ Alg.} \textbf{57} (1989), 165--173.

\bibitem{Rezk} C.\ Rezk, comment on the $n$-Category Caf\'e, August 27, 2017.  Available at
\href{https://golem.ph.utexas.edu/category/2017/08/a_puzzle_on_multisorted_lawver.html\#c052663}{https://golem.ph.utexas.edu/category/2017/08/a$\underline{\;}$puzzle$\underline{\;}$on$\underline{\;}$multisorted$\underline{\;}$lawver.html\#c052663}.

\bibitem{RSW2} R.\ Rosebrugh, N.\ Sabadini, R.\ F.\ C. Walters, Generic 
commutative separable algebras and cospans of graphs, \textsl{Theory Appl.\ Categ.} \textbf{15} (2005), 164--177.  Available at \href{http://www.tac.mta.ca/tac/volumes/15/6/15-06abs.html}{http://www.tac.mta.ca/tac/volumes/15/6/15-06abs.html}.

\bibitem{Se} P.\ Selinger, Dagger compact closed categories and completely positive maps, in {\sl Proceedings of the 3rd International Workshop on  Quantum Programming Languages} (QPL 2005), {\sl ENTCS} {\bf 170} (2007), 139--163.  Available at \href{http://www.mscs.dal.ca/~selinger/papers/dagger.pdf}{http://www.mscs.dal.ca/\~{}selinger/papers/dagger.pdf}.
      
\bibitem{St} M.\ Stay, Compact closed bicategories, \textsl{Theory Appl.\ Categ.}
\textbf{31} (2016), 755--798.  Available at \href{http://www.tac.mta.ca/tac/volumes/31/26/31-26.pdf}{http://www.tac.mta.ca/tac/volumes/31/26/31-26.pdf}.

\bibitem{Tr} T.\ Trimble, Multisorted Lawvere theories, version April 27 2014, nLab. Available at \href{http://ncatlab.org/toddtrimble/published/multisorted+Lawvere+theories}{http://ncatlab.org/toddtrimble/published/multisorted+Lawvere+theories}.


\bibitem{We1} A.\ Weinstein, Symplectic geometry, \textsl{Bull.\ Amer.\ Math.\ Soc.}     
 \textbf{5} (1981), 1--13.   Available at \href{https://projecteuclid.org/euclid.bams/1183548217}{https://projecteuclid.org/euclid.bams/1183548217}.

\bibitem{We2} A.\ Weinstein, Symplectic categories, \textsl{Port.\ Math.} \textbf{67} (2010) 261--278.   Available as \href{https://arxiv.org/abs/0911.4133}{arXiv:0911.4133}.

\bibitem{Willems} J.\ Willems, The behavioral approach to open and interconnected systems, \textsl{IEEE Control Systems Magazine} \textbf{27} (2007), 46--99.
Available at \href{http://homes.esat.kuleuven.be/~Ejwillems/Articles/JournalArticles/2007.1.pdf}{http://homes.esat.kuleuven.be/$\sim$jwillems/Articles/JournalArticles/2007.1.pdf}.

\bibitem{Yau} D.\ Yau, \textsl{Colored Operads}, American Mathematical Society, Providence, Rhode Island, 2016.

\bibitem{Za} F.\ Zanasi, \textsl{Interacting Hopf Algebras---the Theory of Linear
Systems}, Ph.D.\ thesis, Laboratoire de l'Informatique du Parall\'elisme, \'Ecole Normale Sup\'eriere de Lyon, 2015.   Available at \href{https://tel.archives-ouvertes.fr/tel-01218015/} {https://tel.archives-ouvertes.fr/tel-01218015/}.

\end{thebibliography}
\end{document}